\documentclass[a4paper,oneside]{memoir} 

\input xy
\xyoption{all}

\usepackage{amsmath,amssymb,amsthm,wasysym}
\usepackage{pgf,tikz,tikz-cd}

\usepackage{pstricks-add}

\usetikzlibrary{arrows,patterns}
\usetikzlibrary{matrix,positioning}
\tikzstyle{every node}=[font=\normalsize]
\usepackage{tkz-euclide}
\usepackage[utf8x]{inputenc}
\usepackage[T1]{fontenc}
\usepackage[colorlinks=true,allcolors=cyan]{hyperref}
\usepackage{memhfixc}
\usepackage{tikzsymbols}
\usepackage{pdflscape}
\usepackage{ stmaryrd }
\usepackage{makecell}
\usepackage{esvect}
\usepackage{enumerate}
\usepackage{mathtools}
\chapterstyle{bianchi} 

\newtheorem{theorem}{Theorem}
\numberwithin{theorem}{chapter}
\newtheorem{lemma}[theorem]{Lemma}

\newtheorem{corollary}[theorem]{Corollary}
\theoremstyle{definition}
\newtheorem{remark}[theorem]{Remark}
\newtheorem{example}[theorem]{Example}
\newtheorem{definition}[theorem]{Definition}
\newtheorem*{definition*}{Definition}

\newcommand{\bQ}{\mathbb{Q}}

\newcommand{\R}{\mathbb{R}}
\newcommand{\Z}{\mathbb{Z}}
\newcommand{\bC}{\mathbb{C}}
\newcommand{\N}{\mathbb{N}}

\newcommand{\F}{\mathbb{F}}

\newcommand{\bV}{\mathbb{V}}

\newcommand{\bO}{\mathbb{O}}
\newcommand{\bJ}{\mathbb{J}}
\newcommand{\bI}{\mathbb{I}}
\newcommand{\vN}{\mathbb{\bV\hspace{-0,61mm}\N}} 
\newcommand{\gvN}{\widehat{\mathbb{\bV\hspace{-0,61mm}\N}}} 
\newcommand{\On}{\mathbb{\bO\hspace{-0,3mm}{\mathbb N}}}
\newcommand{\Oz}{\mathbb{\bO\hspace{-0,3mm}{\mathbb Z}}}
\newcommand{\No}{\mathbb{\N\hspace{-0,3mm}{\mathbb O}}}

\usepackage{tikz}

\makeatletter
\DeclareRobustCommand{\ootimes}{%
  \mathbin{\mathpalette\o@plus@times\relax}%
}
\newcommand{\o@plus@times}[2]{%
  \ooalign{$\m@th#1\oplus$\cr$\m@th#1\otimes$\cr}%
}
\makeatother

\newcommand{\cP}{{\mathcal P}}

\newcommand{\op}{\mathsf{op}}

\newcommand{\I}{\mathbf{I}}
\newcommand{\II}{\mathbf{I\hspace{-0,6mm}I}}
\newcommand{\III}{ \mathbf{I\hspace{-0,6mm}I\hspace{-0,6mm}I} \hspace{0.2mm}}
\newcommand{\eset}{\emptyset}

\newcommand{\Alpha}{\boldsymbol{\alpha}}
\newcommand{\OM}{\boldsymbol{\omega}}

\newcommand{\Beta}{\mathrm{b}} 
\newcommand{\tip}{\mathrm{t}} 
\newcommand{\w}{\mathrm{w}} 
\newcommand{\sgn}{\mathrm{sgn}} 
\newcommand{\yca}{\mathrm{yca}} 

\newcommand{\card}{\mathrm{card}}

\newcommand{\mex}{\mathrm{mex}} 

\newcommand{\rk}{\mathrm{rank}}

\newcommand{\Co}{\mathrm{Co}}

\newcommand{\CD}{\mathrm{CD}} 
 
\newcommand{\desc}{\mathrm{desc}}

\newcommand{\vn}{\mathrm{vN}} 

\newcommand{\interior}[1]{    {\kern0pt#1}^{\mathrm{o}}}

\newcommand{\eps}{\varepsilon}

\newcommand{\msk}{\medskip}
\newcommand{\ssk}{\smallskip}
\newcommand{\nin}{\noindent}

\newcommand{\llbrack}{{[}\hspace{-0,4mm}{ [}} 
\newcommand{\rrbrack}{] \hspace{-0,4mm} ]} 

\usepackage{scalerel,stackengine}
\stackMath
\newcommand\reallywidehat[1]{%
\savestack{\tmpbox}{\stretchto{%
  \scaleto{%
    \scalerel*[\widthof{\ensuremath{#1}}]{\kern-.6pt\bigwedge\kern-.6pt}%
    {\rule[-\textheight/2]{1ex}{\textheight}}
  }{\textheight}%
}{0.5ex}}%
\stackon[1pt]{#1}{\tmpbox}%
}
\title{\LARGE{On Conway's Numbers and Games,  \\ 
the Von Neumann Universe, and Pure Set Theory}}

\date{(February 23, 2025)}

\author{\Large{Wolfgang Bertram}\\
\small{http://wolfgang.bertram.perso.math.cnrs.fr}\\
\\
\small{Institut \'Elie Cartan de Lorraine (IECL)}\\
\small{Universit\'e de Lorraine, Nancy, France}\\
}

\begin{document}

\maketitle
\begin{abstract}
We take up Dedekind's question
 "Was sind und was sollen die Zahlen?" ("What are numbers, and would should they be?"),
 with the aim to describe the place that Conway's (Surreal) Numbers and Games take, or 
 deserve to take, in the whole of mathematics. 
Rather than just reviewing the work of Conway, and subsequent one by Gonshor, Alling, 
Ehrlich, and others, we propose a new setting which puts the theory of surreal numbers onto the
firm ground of "pure" set theory. 
This approach is closely related to Gonshor's one by "sign expansions", but
appears to be significantly simpler and clearer, and hopefully may contribute to
realizing that "surreal" numbers are by no means surrealistic, goofy or wacky. 
 They could, and probably should, play a central
role in mathematics.
We discuss the interplay between the various approaches to surreal numbers,  
and analyze  the link with Conway's original approach via Combinatorial
Game Theory (CGT).
 To clarify this, we propose to call {\em pure set theory} the algebraic theory of pure sets,
 or in other terms, of the algebraic structures of the {\em von Neumann universe}.
This topic may be interesting in its own right:
it puts CGT into a broad context which has a strong "quantum flavor", and  where
Conway's numbers (as well as their analogue, the nimbers) arise naturally.
\end{abstract}

{\small

\bigskip
\nin \textbf{Keywords.}

\nin
{\em
surreal numbers (Conway numbers), ordinal numbers, nimbers, 
(Conway) reals, binary trees, 
pure set theory, (graded) von Neumann universe, combinatorial game theory, 
order theory
}

\bigskip
\nin \textbf{MSC2020 classification.}

\nin
03A05, 
03E05,
05C05,
06F25,
12-02,
91A46,
97F50,
97H50

}


\vfill \eject

\setcounter{tocdepth}{2}

 \tableofcontents

\setcounter{chapter}{-1} 
\chapter{Introduction}

\section{The place of Conway numbers in mathematics}

In the present work, I would like to contribute to  Dedekind's question
{\em Was sind und was sollen die Zahlen?}\footnote{What are numbers and what should they be? \cite{Ded}}
Dedekind's idea to realize the real numbers by "cuts" of the rationals has been taken up in an ingenious way by John Horton Conway 
to construct,  "out of nothing", an ordered Field, the Conway numbers, 
 later called (following Donald Knuth's book \cite{K}) {\em surreal numbers}, see \cite{Co01}.\footnote{I prefer the term {\em Conway number} because the word 
 \href{https://dictionary.cambridge.org/fr/dictionnaire/anglais/surreal}{"surreal"} is biased as established term in history of arts, or as meaning "surrealistic". Since Conway is no longer among the living, it may be time to honor him by remembering him as the creator of a new concept of "number".} 
  In Conway's own words, this Field 
 comprises "All Numbers Great and Small" -- besides the real numbers, it contains infinitely big and small numbers, such as all ordinal numbers, 
 and  it rightly deserves to be called the {\em absolute arithmetic continuum}, see \cite{E12, E20}.

I first learned about Conway numbers as an undergraduate through a chapter (\cite{He}) in the beautiful 
book "Zahlen" \cite{Numbers}, which  follows the original presentation \cite{Co01} of surreal numbers as 
a particular instance of Games. According to his own writings,
 Conway got his amazing intuition from studying Games.  
Several authors, most notably Alling \cite{A}, and Gonshor \cite{Go}, admitted to have
difficulties with Conway's approach, and proposed alternative ones.\footnote{E.g., \cite{A}, p.\ 14: "...the author found
Conway's basic construction of his surreal numbers [..] hard to follow -- so much 
that he gave another construction of the surreals within a more conventional set
theory."}
Conway himself discusses such issues in an "Appendix to Part Zero" in \cite{Co01}, and  I will come back to this at several places.

Following the ideas of Dedekind and Cantor, the natural, rational, and real numbers can be constructed on the basis of
set theory in the form later given by \href{https://en.wikipedia.org/wiki/Zermelo–Fraenkel_set_theory}{Zermelo and Fraenkel (ZF)}, or others, where
mathematical objects are realized as {\em pure sets}. A particularly clear model of the universe of pure sets is given by the
\href{https://en.wikipedia.org/wiki/Von_Neumann_universe}{\em cumulative von Neumann hierarchy}, also called the {\em von Neumann
universe}.
Therefore, if Conway numbers are full citizens of the mathematical universe, one must ask: {\em How can we realize the
Conway numbers in the universe of pure sets, that is, in the von Neumann hierarchy?} 
The present work gives an answer to this question -- an answer which is stunningly simple, and which, hopefully, contributes to putting
Conway numbers into the center of mathematics instead of considering them as an exotic and surrealist 
fresco on its border. 
In the following, I give a concise description of this construction.

\section{Numbers as sets of ordinals}

The {\em ordinal numbers} are the prolongation of the natural numbers $\N$ into the infinite range.
Their theory  is one of Georg Cantor's greatest achievements, and it belongs to the central issues of modern  set theory. 
John von Neumann discovered a system realizing not only the ordinals, but the whole universe of sets,
 "out of nothing", starting from the empty set,
$\vN_0 = \eset$, by a transfinite induction procedure:
$\vN_{\alpha + 1} = \cP(\vN_\alpha)$ is the power set of the preceding stage, and
$\vN_\beta = \cup_{\alpha < \beta} \vN_\alpha$ if $\beta$ is a {\em limit ordinal}.
The ordinal $\alpha$  itself is an element of $\vN_{\alpha+1}$, and the class of ordinals $\On$ thus is a subclass of $\vN$. Now we define:

\begin{definition}
A {\em (Conway) number} is a {\em set of ordinals having a maximal element}. 
The maximal element $\Beta(x)$ of such a set $x$ is called the {\em birthday of $x$}.
\end{definition}

For instance, every finite and non-empty subset of $\N$ is a number, but
 the first infinite ordinal $\omega  = \{ 0,1,2,3,\ldots \}=\N$ is a set of ordinals having no maximal element,
and thus is not a number. Likewise, $\eset$ is not a number.
The maximal element $\Beta(x)$ of $x$ plays the r\^ole of a  "full stop sign" at the end
of a phrase $x$: it is a necessary part of the phrase, though usually not pronounced. 
As a consequence,
an ordinal $\beta$ is a number if, and only if, it has a maximal element, i.e.,
if it is a {\em successor ordinal} in von Neumann's system. 
 To remain in keeping with Conway's notation, we therefore must
 define:
{\em for every von Neumann ordinal $\alpha$, 
the corresponding \textbf{Conway ordinal} $\alpha_\Co$  is given by}
\begin{equation}\label{eqn:Conway-ordinal}
\Alpha := \alpha_\Co := \alpha + 1 .
\end{equation}
Just like the von Neumann hierarchy, numbers are organized in stages: for every ordinal $\alpha$,
we let $\No_\alpha$ be the set of numbers $x$ such that $\Beta(x)<\alpha$.
The union of all $\No_\alpha$ forms a proper class $\No$ in $\vN$, called the {\em class of (Conway) numbers}, and containing in turn the class of (Conway) ordinals.
In some textbooks on set theory (cf.\ \cite{D}, p. 88), the von Neumann stages  $\vN_\alpha$
 are represented like
an icecream cone -- their size exploses as a function of $\alpha$ ("tetration").
Now, there is another cone $\No$ inside, whose size growths roughly by $2^\alpha$, still much faster than
the linear growth of $\On$, but much gentler than "tetration".

\section{The number tree}

As emphasized by Ehrlich (\cite{E12, E20}), 
the structure of {\em binary tree} on $\No$ is a salient feature: 
this tree is the Hasse diagram of a natural partial order $x \prec y$, meaning
"$y$ is a descendant of $x$", or "$x$ is an initial segment of $y$".
 Every number $x$ has {\em exactly  two} immediate
descendants (children), denoted by $x_+$ and $x_-$. 
We start by the
root $0_\Co$ and represent right-descendants $x_+$ by a right branch
and left-descendants $x_-$ by a left branch.  
Here is a figure of the number tree $\No_4$:
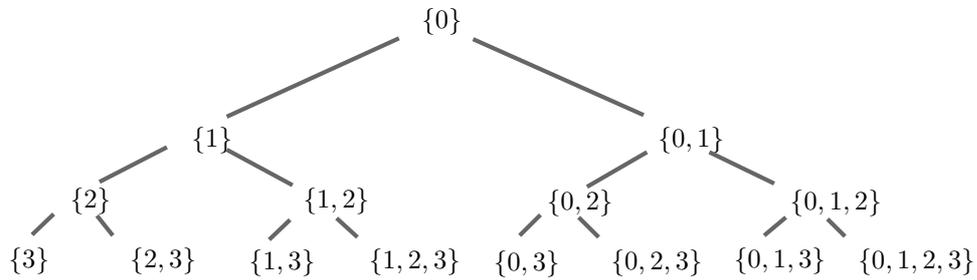
\begin{center}
\newrgbcolor{wwwwww}{0.4 0.4 0.4}
\psset{xunit=0.4cm,yunit=0.4cm,algebraic=true,dimen=middle,dotstyle=o,dotsize=5pt 0,linewidth=2pt,arrowsize=3pt 2,arrowinset=0.25}
\begin{pspicture*}(-14.633691567548487,-8.805684356179417)(19.64254665899164,0.2)
\rput[tl](-0.536254913965393,0.22999587024050655){$\{ 0\}$}
\rput[tl](-8.09576584037922,-3.7091727687234033){$\{1\}$}
\rput[tl](7.2396079870496,-3.7091727687234033){$\{0,1\}$}
\rput[tl](-12.089645154884296,-5.733467763746523){$\{2\}$}
\rput[tl](-4.429241863644729,-5.733467763746523){$\{1,2\}$}
\rput[tl](3.58721031360037,-5.788178439287688){$\{0,2\}$}
\rput[tl](11.583052083228484,-5.788178439287688){$\{0,1,2\}$}
\psline[linewidth=1.6pt,linecolor=wwwwww](-1.2842867355041276,-0.7821516272710536)(-6.946841654014747,-3.380908715476411)
\psline[linewidth=1.6pt,linecolor=wwwwww](-11.132208332913901,-5.541980399352444)(-8.916425973496702,-4.393056212987971)
\psline[linewidth=1.6pt,linecolor=wwwwww](-6.946841654014747,-4.475122226299719)(-4.7857699701387135,-5.6514017504347756)
\psline[linewidth=1.6pt,linecolor=wwwwww](1.1503383260777333,-0.8095069650416362)(6.758182569047188,-3.326198039935245)
\psline[linewidth=1.6pt,linecolor=wwwwww](8.919254252923222,-4.557188239611467)(11.052970599028672,-5.596691074893609)
\psline[linewidth=1.6pt,linecolor=wwwwww](4.898019600647564,-5.706112425975941)(6.867603920129519,-4.557188239611467)
\rput[tl](-14.113940149907416,-7.730407420999061){$\{3\}$}
\rput[tl](-10.120060835402342,-7.730407420999061){$\{2,3\}$}
\rput[tl](-6.208247534209015,-7.785118096540226){$\{1,3\}$}
\rput[tl](-2.269078895245105,-7.757762758769643){$\{1,2,3\}$}
\rput[tl](1.834221770342301,-7.812473434310809){$\{0,3\}$}
\rput[tl](5.691324395994462,-7.785118096540226){$\{0,2,3\}$}
\rput[tl](9.739914386040702,-7.703052083228478){$\{ 0,1,3\}$}
\rput[tl](13.815859713857526,-7.757762758769643){$\{0,1,2,3\}$}
\psline[linewidth=1.6pt,linecolor=wwwwww](-13.3479906923311,-7.292722016669738)(-12.636751910295951,-6.60883857240517)
\psline[linewidth=1.6pt,linecolor=wwwwww](-11.21427434622565,-6.663549247946335)(-10.694522928584579,-7.32007735444032)
\psline[linewidth=1.6pt,linecolor=wwwwww](-4.867835983450462,-6.800325936799249)(-5.55171942771503,-7.456854043293234)
\psline[linewidth=1.6pt,linecolor=wwwwww](-3.3359370682978304,-6.7456152612580835)(-2.652053624033263,-7.374788029981485)
\psline[linewidth=1.6pt,linecolor=wwwwww](3.3661206854949324,-6.636193910175753)(2.6822372412303648,-7.32007735444032)
\psline[linewidth=1.6pt,linecolor=wwwwww](4.624466222941737,-6.7182599234875005)(5.280994329435722,-7.374788029981485)
\psline[linewidth=1.6pt,linecolor=wwwwww](11.43594532781683,-6.7182599234875005)(10.72470654578168,-7.32007735444032)
\psline[linewidth=1.6pt,linecolor=wwwwww](12.803712216345966,-6.772970599028666)(13.378174309528204,-7.347432692210903)
\end{pspicture*}
\end{center}
The innocent-looking definition (\ref{eqn:Conway-ordinal}) is of utmost importance for all the following,
and it reflects some rather subtle issues to which
Conway alludes by an amusing anecdote told in the nice book 
\cite{CG}, beginning of Chapter 10, under the headline 
"Sierpi\'nski's luggage".
\footnote{{\em Wac\l aw Sierpi\'nski, the great Polish mathematician, was very interested
in infinite numbers. The story, presumably apocryphal, is that
once when he was travelling, he was worried that he'd lost one piece
of his luggage. "No, dear!" said his wife, "All six pieces are here."
"That can't be true," said Sierpi\'nski, "I've counted them several
times: zero, one, two, three, four, five."} (He  forgot to say:
"full stop: six.")}

\section{Sign expansions: "Sierpi\'nski's luggage"} 

Our setting of surreal numbers is close to the approach via
 {\em sign-expansions} used by Gonshor  \cite{Go}.
 Let us define:

 \begin{definition}\label{def:sign-expansion}
 The {\em sign-expansion $s= s_x$ of a number $x$} is defined by:
 $$
 s_x(\alpha) := \Bigg\{ \begin{matrix} 
 + & \mbox{ if } &  \alpha \mbox  {  \em  is an element in } x, \mbox{ i.e.: } 
 \alpha < \Beta(x) \mbox{ and } \alpha \in x,
 \\
 - & \mbox{ if }  &  \alpha \mbox{ \em  is a hole in } x, \mbox{ i.e.: } 
 \alpha < \Beta(x) \mbox{ and } \alpha \notin x,
 \\
 0 & \mbox{ if } &  \alpha \mbox{ \em  is not in } x, \mbox{i.e.: } 
  \mbox{ if }  \alpha \geq \Beta(x).
 \end{matrix}
$$
\end{definition}

 \nin
  Here is the binary tree $\No_4$, written in terms of sign-expansions (you may adjoin as many  zeroes to the right of each entry as you like):
 \begin{center}
\newrgbcolor{wwwwww}{0.4 0.4 0.4}
\psset{xunit=0.4cm,yunit=0.4cm,algebraic=true,dimen=middle,dotstyle=o,dotsize=5pt 0,linewidth=2pt,arrowsize=3pt 2,arrowinset=0.25}
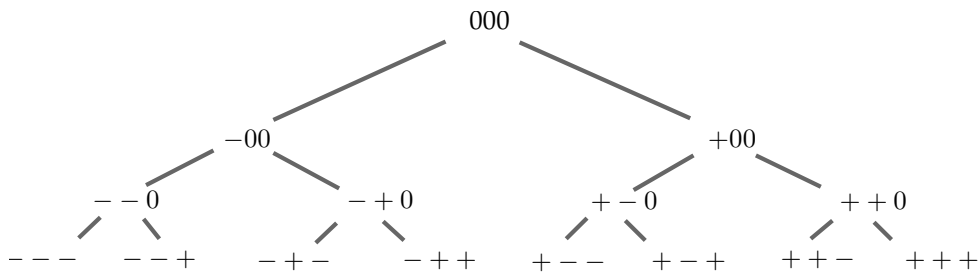
\begin{pspicture*}(-15.633691567548487,-8.75684356179417)(19.64254665899164,0.2)
\rput[tl](-0.5,0.22999587024050655){000}
\rput[tl](-8.59576584037922,-3.7091727687234033){$-00$}
\rput[tl](7.352396079870496,-3.7091727687234033){$+00$}
\rput[tl](-12.889645154884296,-5.733467763746523){$--0$}
\rput[tl](-4.529241863644729,-5.733467763746523){$-+0$}
\rput[tl](3.4858721031360037,-5.788178439287688){$+-0$}
\rput[tl](11.683052083228484,-5.788178439287688){$++0$}
\psline[linewidth=1.6pt,linecolor=wwwwww](-1.2842867355041276,-0.7821516272710536)(-6.946841654014747,-3.380908715476411)
\psline[linewidth=1.6pt,linecolor=wwwwww](-11.132208332913901,-5.541980399352444)(-8.916425973496702,-4.393056212987971)
\psline[linewidth=1.6pt,linecolor=wwwwww](-6.946841654014747,-4.475122226299719)(-4.7857699701387135,-5.6514017504347756)
\psline[linewidth=1.6pt,linecolor=wwwwww](1.1503383260777333,-0.8095069650416362)(6.758182569047188,-3.326198039935245)
\psline[linewidth=1.6pt,linecolor=wwwwww](8.919254252923222,-4.557188239611467)(11.052970599028672,-5.596691074893609)
\psline[linewidth=1.6pt,linecolor=wwwwww](4.898019600647564,-5.706112425975941)(6.867603920129519,-4.557188239611467)
\rput[tl](-15.813940149907416,-7.730407420999061){$---$}
\rput[tl](-11.920060835402342,-7.730407420999061){$--+$}
\rput[tl](-7.4908247534209015,-7.785118096540226){$-+-$}
\rput[tl](-2.7,-7.757762758769643){$-++$}
\rput[tl](1.534221770342301,-7.812473434310809){$+--$}
\rput[tl](5.491324395994462,-7.785118096540226){$+-+$}
\rput[tl](9.739914386040702,-7.703052083228478){$++-$}
\rput[tl](13.815859713857526,-7.757762758769643){$+++$}
\psline[linewidth=1.6pt,linecolor=wwwwww](-13.3479906923311,-7.292722016669738)(-12.636751910295951,-6.60883857240517)
\psline[linewidth=1.6pt,linecolor=wwwwww](-11.21427434622565,-6.663549247946335)(-10.694522928584579,-7.32007735444032)
\psline[linewidth=1.6pt,linecolor=wwwwww](-4.867835983450462,-6.800325936799249)(-5.55171942771503,-7.456854043293234)
\psline[linewidth=1.6pt,linecolor=wwwwww](-3.3359370682978304,-6.7456152612580835)(-2.652053624033263,-7.374788029981485)
\psline[linewidth=1.6pt,linecolor=wwwwww](3.3661206854949324,-6.636193910175753)(2.6822372412303648,-7.32007735444032)
\psline[linewidth=1.6pt,linecolor=wwwwww](4.624466222941737,-6.7182599234875005)(5.280994329435722,-7.374788029981485)
\psline[linewidth=1.6pt,linecolor=wwwwww](11.43594532781683,-6.7182599234875005)(10.72470654578168,-7.32007735444032)
\psline[linewidth=1.6pt,linecolor=wwwwww](12.803712216345966,-6.772970599028666)(13.378174309528204,-7.347432692210903)
\end{pspicture*}
\end{center}
Compare this to the  basic definition adopted in \cite{Go}, p.3:
{\em A surreal number is a function from an initial segment of the ordinals into the set $\{  + , -  \}$, i.e.,
informally, an ordinal sequence consisting of pluses and minuses which terminate. The empty sequence
is included as possibility.}  
 It is not quite clear to me what the words "which terminate" shall mean:
Does the empty sequence terminate? 
Does a sequence of $\omega$ pluses "terminate"?
Gonshor adds: 
{\em For stylistic reasons I shall occasionally say that $a(\alpha)=0$ if $a$ is undefined at $\alpha$.
This should be regarded as an abuse of notation since we do not want the domain of $\alpha$ to be
the proper class of all ordinals.}
It is quite unsatisfying to start the basic definition of a theory by an "abuse of notation", which, as often,
reflects a fundamental problem -- in our case, the missing distinction between Conway and von Neumann
ordinals.

\section{The Conway reals}

$\No$ contains a canonical copy of $\R$ which we call the {\em Conway reals}, $\R_\Co$.
The set $\R_\Co$ is a
 a disjoint union of the set of  {\em short numbers} $\No_\omega$ (the finite subsets of $\omega = \N$)
with the set of {\em long reals} -- those of the form $x = X \cup \{ \omega \}$,
with $X \subset \N$ both infinite and co-infinite. 
The short numbers correspond to {\em dyadic rationals}: $\No_\omega \cong \Z[\frac{1}{2}]$.
Under this correspondence, the number tree $\No_4$ is represented like this,
where all Conway reals ought to carry an index $\Co$ omitted here:
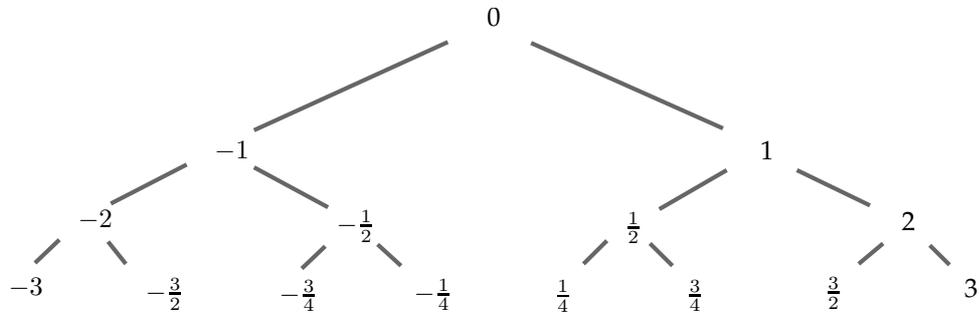
\begin{center}
\newrgbcolor{wwwwww}{0.4 0.4 0.4}
\psset{xunit=0.45cm,yunit=0.45cm,algebraic=true,dimen=middle,dotstyle=o,dotsize=5pt 0,linewidth=2pt,arrowsize=3pt 2,arrowinset=0.25}
\begin{pspicture*}(-14.633691567548487,-8.805684356179417)(19.64254665899164,0.2)
\rput[tl](-0.13536254913965393,0.22999587024050655){0}
\rput[tl](-8.09576584037922,-3.7091727687234033){$-1$}
\rput[tl](7.852396079870496,-3.7091727687234033){1}
\rput[tl](-12.089645154884296,-5.733467763746523){$-2$}
\rput[tl](-4.529241863644729,-5.733467763746523){$-\frac{1}{2}$}
\rput[tl](3.8858721031360037,-5.788178439287688){$\frac{1}{2}$}
\rput[tl](11.983052083228484,-5.788178439287688){2}
\psline[linewidth=1.6pt,linecolor=wwwwww](-1.2842867355041276,-0.7821516272710536)(-6.946841654014747,-3.380908715476411)
\psline[linewidth=1.6pt,linecolor=wwwwww](-11.132208332913901,-5.541980399352444)(-8.916425973496702,-4.393056212987971)
\psline[linewidth=1.6pt,linecolor=wwwwww](-6.946841654014747,-4.475122226299719)(-4.7857699701387135,-5.6514017504347756)
\psline[linewidth=1.6pt,linecolor=wwwwww](1.1503383260777333,-0.8095069650416362)(6.758182569047188,-3.326198039935245)
\psline[linewidth=1.6pt,linecolor=wwwwww](8.919254252923222,-4.557188239611467)(11.052970599028672,-5.596691074893609)
\psline[linewidth=1.6pt,linecolor=wwwwww](4.898019600647564,-5.706112425975941)(6.867603920129519,-4.557188239611467)
\rput[tl](-14.113940149907416,-7.730407420999061){$-3$}
\rput[tl](-10.120060835402342,-7.730407420999061){$-\frac{3}{2}$}
\rput[tl](-6.208247534209015,-7.785118096540226){$-\frac{3}{4}$}
\rput[tl](-2.269078895245105,-7.757762758769643){$-\frac{1}{4}$}
\rput[tl](1.834221770342301,-7.812473434310809){$\frac{1}{4}$}
\rput[tl](5.691324395994462,-7.785118096540226){$\frac{3}{4}$}
\rput[tl](9.739914386040702,-7.703052083228478){$\frac{3}{2}$}
\rput[tl](13.815859713857526,-7.757762758769643){3}
\psline[linewidth=1.6pt,linecolor=wwwwww](-13.3479906923311,-7.292722016669738)(-12.636751910295951,-6.60883857240517)
\psline[linewidth=1.6pt,linecolor=wwwwww](-11.21427434622565,-6.663549247946335)(-10.694522928584579,-7.32007735444032)
\psline[linewidth=1.6pt,linecolor=wwwwww](-4.867835983450462,-6.800325936799249)(-5.55171942771503,-7.456854043293234)
\psline[linewidth=1.6pt,linecolor=wwwwww](-3.3359370682978304,-6.7456152612580835)(-2.652053624033263,-7.374788029981485)
\psline[linewidth=1.6pt,linecolor=wwwwww](3.3661206854949324,-6.636193910175753)(2.6822372412303648,-7.32007735444032)
\psline[linewidth=1.6pt,linecolor=wwwwww](4.624466222941737,-6.7182599234875005)(5.280994329435722,-7.374788029981485)
\psline[linewidth=1.6pt,linecolor=wwwwww](11.43594532781683,-6.7182599234875005)(10.72470654578168,-7.32007735444032)
\psline[linewidth=1.6pt,linecolor=wwwwww](12.803712216345966,-6.772970599028666)(13.378174309528204,-7.347432692210903)
\end{pspicture*}
\end{center}
We call {\em mirror Berlekamp algorithm} (Theorem \ref{th:Berlekamp}) the algorithm translating
sets (numbers $x$) into real numbers, and vice versa -- indeed, it is a version of Berlekamp's
Rule (\cite{Co01}, p. 31), simpler than the original one because it is performed using the
{\em left} part of the tree, and not the positive (right) part. 
The Conway reals $\R_\Co$ form a model of $\R$ in the von Neumann universe.
As far as the order structure is concerned, this is easy: 

\section{Total order}\label{sec:total}

Order structures of $\No$ can all be described in terms of the natural order of ordinals.
There are partial orders, like $\preceq$, or "older than", and there is a total order $\leq$
(Theorem \ref{th:order}). 
To highlight  the order theoretic viewpoint, one may consider for any  well-ordered set
$(M,<)$ the set $\No_M$ of {\em numbers in $M$}, 
\begin{equation}\label{eqn:NoM}
\No_M := \{ x \in \cP(M) \mid \, x \mbox{ has a maximum, denoted by } \Beta(x) \}.
\end{equation}
All our order-theoretic definitions can be formulated for $\No_M$.
Total order $\leq$ and tree-order $\preceq$ have close relations among each other : in the figures of the tree,
the total order increases from left to right, the birthday from top to bottom, and $\preceq$ according to the
edges of the tree.
In our setting, this is the easy part.

\section{Conway cuts, and arithmetic operations}\label{sec:ONAG}

Extending the arithmetic structure of $\R_\Co$ to all of $\No$ is an entirely different problem. 
For the time being, I do not know a 
"direct" or "simple" algorithm  how to define and compute 
sum $x+y$ and product $xy$ of $x,y \in \No$,  in terms of sets or of sign-sequences. 
It needed Conway's genius to find a procedure of defining the arithmetic structure, described in a very
concise way p.4-5 \cite{Co01}.
Because of its importance we quote the whole paragraph:

{\em
\textbf{Construction.}
If $L,R$ are any two sets of numbers, and no member of $L$ is $\geq$ any member of of $R$,
then there is a number $\{ L \vert R \}$. All numbers are constructed this way.

\ssk
\textbf{Convention.}
If $x = \{ L \vert R\}$ we write $x^L$ for the typical member of $L$, and $x^R$ for the typical
member of of $R$. For $x$ itself we then write $\{ x^L\vert x^R\}$.

$x=\{ a,b,c,\ldots \vert d,e,f,\ldots\}$ means that $x=\{ L \vert R\}$, where $a,b,c,\ldots$ are
typical members of $L$, and $d,e,f,\ldots$ the typical members of $R$.

\ssk
\textbf{Definitions.}

{\em Definition} of $x \geq y, x\leq y$.

\nin We say $x \geq y$ if (no $x^R \leq y$ and $x \leq$ no $y^L$), and $x \leq y$ iff $y \geq x$.
We write $x \not\leq y$ to mean $x \leq y$ does not hold.

{\em Definition} of $x = y, x > y, x<y$.

\nin
$x=y$ iff ($x\geq y$ and $y \geq x$). 
$x>y$ iff ($x \geq y$ and $y \not\geq x$).
$x<y$ iff $y>x$.

{\em Definition} of $x+y$.

\nin
$x+y = \{ x^L + y, x + y^L \vert x^R + y, x + y^R \}$

{\em Definition} of $-x$.

\nin
$-x = \{ - x^R \vert - x^L \}$

{\em Definition} of $xy$.

\nin
$
xy  = \{ x^L y+ xy^L - x^L y^L, \, x^r y + xy^R -x^R y^R \mid$

$
{ }  \qquad \qquad \mid x^l y  + x y^R - x^L y^R, x^R y + xy^L - x^R y^L \} .
$
}
 
\ssk
Conway's  framework is quite different from the one proposed here, and we follow the same
strategy as Gonshor to translate between the settings: 
 the "Fundamental Existence Theorem" 
(\cite{Go} Thm.\ 2.1, our Theorem \ref{th:FundamentalExistence}) ensures that every number $x$ can be represented by a
{\em Conway cut} $(L,R)$ in $\No$.
This representation is not unique, but it can be used to define the arithmetic operations
by the formulae given above.
The  "Fundamental Existence Theorem" (Theorem \ref{th:FundamentalExistence})
 essentially corresponds to
{\em connectedness and completeness of the binary number tree} (Theorems \ref{th:connected},
\ref{th:completeness}).
This also entails that Alling's axioms of surreal number systems (recalled in Def.\ \ref{def:Alling}) are satisfied, and establishes
 equivalence of our approach with all other known approaches to 
surreal numbers.
Thus a first goal of the present work is achieved: we have given a simple and
rigorous definition of $\No$ based on classical,  "pure",  set theory. 
Let me  explain what is meant by this term.

\section{Pure set theory, or: "What are the elements of $\pi$"?}\label{sec:puresettheory}

We start this paper with a Chapter \ref{chap:vNu}  on the von Neumann universe. 
The von Neumann universe plays a paradoxical  role in mathematics:  
on the one hand, it is the canonical model of set theory -- there is just one empty set,
and by induction, all other sets, too, are unique like individual persons, defined by their place
in the von Neumann hierarchy. On the other hand, everybody agrees that 
this model has "nothing to do with the real mathematical life", and  "nobody considers mathematical 
objects as sets of sets of sets...",  in other words, that it is practically irrelevant.\footnote{ Cf.\ the detailed discussion p.125/26 in \cite{D}, which is 
very well-written reference on general set-theory.}
I would like to propose, be it just by curiosity to see where it leads, as a sort of game or experiment, 
to take temporarily
the opposite position: to do as if 
 the von Neumann hierarchy perfectly encoded "the rules of the game of 
the mathematical universe",  and do as if mathematical objects were indeed pure sets.
I have been playing this game for some time, and the longer I do so, the more it appears to make sense.
To be more specific, it seems to shed light onto a fundamental issue of modern mathematics,
{\em  "classical", versus "quantum"}:
\begin{enumerate}
\item
"Quantum" is discreteness and the tree-like structure of a world organized in hierarchies, molecules, atoms, particles called 
"elementary" on some level, but appearing less elementary on the next higher level ;
these levels are organised in "stages", indexed by  quantum numbers, e.g., the one called "birthday" by 
Conway, and "rank" by von Neumann,
\item
"Classical" mathematics thinks the world as a "continuum",
 ignoring the notion of "rank" or "birthday".  Sets are underlying to "spaces", whose
 elements are considered to be
"points": points
 have no internal structure and no "individuality". 
 Classically, "equality" always means "equivalence under isomorphy", referring to some category;
 it makes no sense to ask if two isomorphic objects "are the same, or not".
  A global    "equality relation" on the mathematical universe appears to be a fiction.
 \end{enumerate}
 "Quantum" seems to be paradoxical and disturbing --
 traditionally, we all think "classically", and tend to believe that this point of view is the only correct and mathematically
reasonable one. 
To express this, some authors point out that
it makes no sense to ask questions like, for instance:  "what are the elements
of $\pi$?"\footnote{ E.g., \cite{Lei}, Introduction.}
But why does it not make sense? 
The problem comes from the notion of {\em ordered pair}. Let me explain: 

\subsubsection{The trouble with Kuratowski's definition}

The "elements of $\pi$"
would depend on the {\em model} of $\R$ that we use; and these models use
standard constructions of mathematics:  {\em  Cartesian products,
equivalence relations, functions}, and so on.  
Most of it is based on the notion of {\em ordered pair}. In set theory,   the
most commonly accepted definition of ordered pair  is the
\href{https://en.wikipedia.org/wiki/Ordered_pair#Kuratowski's_definition}{\em Kuratowski definition}
(see Eqn.\  (\ref{eqn:Kuratowski})).
However, there exist
\href{https://en.wikipedia.org/wiki/Ordered_pair#Defining_the_ordered_pair_using_set_theory}{several
variants of this definition}, and none of them is "natural" in any way. 
Halmos  \cite{Ha}, p. 24/25, remarks:  "The explicit definition of the ordered pair (by
(\ref{eqn:Kuratowski})) is frequently relegated to pathological set theory," leading to
  "mistrust and suspicion that many mathematicians feel towards
(this) definition".
Indeed, every definition of the "elements of $\pi$" that would use the Kuratowski definition is 
biased by this arbitrary choice, and would not have much sense.
Thus let us call {\em natural (in the sense of pure set theory)} every construction of the field
$\R$ in ZF-set theory that does {\em not} refer, explicitly or implicitly, to the Kuratowski
construction of ordered pairs.
As far as I see, all of the constructions described in the comprehensive overview \cite{We}
are not natural in this sense (recall that already the usual constructions of $\Z$ and $\mathbb Q$ out of $\N$
do use ordered pairs).  
Now, it is remarkable that, 
 {\em The construction of the Conway reals $\R_\Co$ 
is natural.} It even appears to be the {\em only} natural construction of $\R$ I know of.
(Note that this construction can easily be given without mentioning the general class $\No$
surrounding it.)\footnote{ Weiss, \cite{We},  completely 
underestimates the novelty of Conway's approach: he writes,
{\em
...one may view the
surreal number system as providing yet another construction of the real
numbers. However, when one distills just the real numbers from the
entire array of surreal ones, the construction basically collapses to the
Dedekind cuts construction.} This is not true.}
Summing up, the question
"what are the elements of $\pi$?" does make sense in pure set theory,
and it does have an answer (Example \ref{ex:numbers}).

\section{Approach by combinatorial game theory}\label{sec:CGT}

My aim is to analyze, as completely as possible, the relations between standard set theory
and foundations of surreal numbers. 
Of course, already
Conway himself discusses such foundational issus, including the trouble with the
Kuratowski definition, at various places,
e.g.,  in the Appendix to Part Zero (p.\ 64--68 in \cite{Co01}), p.25--27 ("The logical 
theory of real numbers"),
and in the
Epilogue, p.\ 225--228 loc.\ cit.
He admits his dissatisfaction both with Gonshor's and Alling's approaches, on the grounds that
the true nature of his constructions belongs to  the more general context of 
\href{https://en.wikipedia.org/wiki/Combinatorial_game_theory}{\em combinatorial game theory}
(see Chapter VIII of \cite{S} for a modern account). 
He writes (\cite{Co01}, p.65/66): 
\begin{enumerate}
\item[ ]
{\em 
Plainly the proper set theory in which to perform a formalisation would be one with two kinds of
membership, and would in fact be very much like the abstract theory of games that underlies
the next part of this book.} 
\end{enumerate}

\nin
In Chapter \ref{chap:Games}, I  try to carry out this idea.
Implicitly, this is also done in \cite{S}, however, the relevance for general pure set theory 
remains invisible there.
Historically and logically, combinatorial game theory developed in two steps:
\begin{enumerate}
\item[(A)]
The theory of {\em impartial games} (Sprague, Grundy). 
H.W.\ Lenstra was one among the first to notice that this theory corresponds exactly to what
we call "pure set theory": 
 \cite{Le}, 
 p.1,  "Definition. {\em A game is a set.}"
 \item[(B)]
 The theory of {\em partizan games}, started by Conway, with Berlekamp and Guy.
 The quote given above is a very precise definition: {\em it is set theory with two kinds
 of membership}, say $\in_L$ and $\in_R$, replacing $\in$.
 \end{enumerate}
 
\nin
The "universe" for (A) is the von Neumann universe $\vN$.
We define the "universe" for (B), 
which we call the {\em graded von Neumann universe} $\gvN$. 
Our definition is close to the "universe of games" $\widetilde{\mathbf{PG}}$
defined in \cite{S}, p.398,
but we insist on the fact that it really concerns set theory, and not so much a certain 
and quite specialised application of mathematics. 
Compressed into a simple slogan:
{\em Von Neumann created the "ungraded (impartial) universe" out of nothing, and
Conway created the "graded (partizan) universe" out of nothing.}
However, there is also a "diagonal imbedding" of the ungraded into the graded universe,
so the partizan theory appears to be strictly more general than the impartial one.

There exists an analog of the surreal numbers in the ungraded context: this is the
{\em Field of nimbers}, another beautiful invention of Conway's (\cite{Co01}, Chapter 6).
We recall  the main results (Section \ref{sec:impartial}).
The structural analogy
$$
\begin{matrix}
\frac{ \mbox{ ungraded von Neumann universe }}{ \mbox{ Field of nimbers }} & = &
\frac{ \mbox{ graded von Neumann universe }}{ \mbox{ Field of surreal numbers }} 
\end{matrix}
$$
is strong, but it becomes clear that the technical realization is much more complicated in
the graded case, and that certain objects are not yet well identified in this context.
In my opinion, this is an important topic for further research.

This analysis of the logical and set-theoretic foundations of surreal numbers leads me to
the conclusion that Conway was wrong when objecting 
 against  Gonshor's definition of surreal numbers that it
{\em requires a prior construction of the ordinals, which are in ONAG produced as particular
cases of the surreals} (\cite{Co01}, p. 226).
Indeed, the ordinals are the backbone of the von Neumann universe, and the purpose of
axiomatic set theory is to ensure their existence
and consistency. If it were true that in ONAG they are "produced", then this would mean
that Conway's "axioms" (quoted in Section \ref{sec:ONAG}) are at least as strong as those of axiomatic set theory.
But Conway's "axioms" are too weak and not formal enough -- they not even imply the 
\href{https://en.wikipedia.org/wiki/Axiom_of_infinity}{axiom of
infinity}! (The stage $\No_\omega$ of short numbers would be a valid "model of Conway's axioms",
which thus could be realized in a finitary structure.) 
Or maybe by using the  word "set" in his definition, Conway implicitly assumed that
axiomatic set theory underlies also his theory -- but then his objection against Gonshor would
be self-contradictory.
Of course, these remarks by no means prevent recognizing
 the genius and the profoundness of Conway's work, but they confirm the impression that
 Philip Ehrlich once expressed by saying that
 {\em Conway was his own worst enemy in promoting the surreals}.
 \footnote{See  \url{https://mathoverflow.net/questions/29300/whats-wrong-with-the-surreals}.}

\section{Everything is number}

The final chapter \ref{chap:No-theory} is both a summary of the preceding ones and a
(personal) outlook on further research  topics. 
 
In his overview \cite{E20}, Philip Ehrlich puts the Conway numbers into the wide context
of "infinitesimalist  theories of continua", and I'm grateful to him for mentioning in this context also
my own contribution (cf.\ \cite{Be08}), under the hashtag TDC (topological differential calculus).
Indeed, my motivation and interest in Conway numbers come from my work on
foundations of differential calculus and number systems containing infinitesimals, and I do believe
that the present approach to Numbers will, ultimately, make it possible to see most, if not all, of the approaches
discussed in \cite{E20} in a common language and framework.

Both "TDC" and the present approach to Conway numbers reflect the Pytha\-gorian idea that
{\em everything is number}, and "pure set theory" as defined above is its prolongation even further.
In the last chapter, we put forward some more arguments which might support such a point of view.
As a sort of preliminary conclusion, this leads me to answer  the question
"what are Conway numbers and what should they be?"
They should, eventually, become to us as familiar as the real numbers;
as long as this is not the case, they "are" given by their simplest model, and to my
feeling this is the model proposed in Chapter \ref{chap:Conway}.

I acknowledge that all original and deep ideas in this realm are due to John Horton Conway, and 
 my contribution 
simply is a modest, and possibly quite naive, try to clarify the place that these ideas may deserve to take in the 
mainstream mathematical universe.
If the reader, and future mathematicians, may retain something new it is probably the idea that ordinals, 
Conway numbers, and the von Neumann universe should be seen as a kind of trinity, created  on
the same grounds and enriching the understanding of each other.

\bigskip
\nin
\textbf{\large Notation.}

\nin
We use standard notation of set theory, and write
$\vN$ for the class of the von Neumann universe, and
$\On$ for the class of ordinals.
The only deviation to standard notation is 
that we denote ordinal addition by $\alpha \oplus \beta$, to distinguish it from the
{\em Hessenberg sum} $\alpha+\beta$ (which is also the Conway-sum), and similarly
$\alpha \otimes \beta$ for the usual ordinal product, reserving
$\alpha \cdot \beta$ for their Hessenberg-Conway product. 
In our approach to Conway numbers, we will use several (partial) order relations simultaneously,
denoted by $\preceq, \leq , \sqsubseteq, \subseteq$,... 
When $a<b$ in some ordered set, 
we use  "French notation" for open, resp.\ closed intervals:
\begin{equation}
\begin{matrix}
]a,b[ & =  &\{ x \mid a<x<b\},\\
[a,b] & =  & \{ x \mid a \leq x \leq b\}, \\
[a,b[ & = & \{ x \mid a \leq x < b \}.
\end{matrix}
\end{equation}
Recall that a subset $I$ is called {\em convex} if, whenever $a,b \in I$, then also
the interval $[a,b]$ belongs to $I$.
We say that {\em $b$ is an immediate successor of $a$}, or {\em $a$ is an
immediate predecessor of $b$},  if
$a<b$ and $]a,b[ = \eset$. 
 {\em Ordinal intervals}, that is, intervals in $\On$, are denoted by double brackets, e.g., 
\begin{equation}
\begin{matrix}
\llbrack  \beta , \gamma \rrbrack &=  & \{ \alpha \mid \, \beta \leq \alpha \leq \gamma \}, \\
\rrbrack  \beta , \gamma \llbrack & = & \{ \alpha \mid \, \beta  <  \alpha  < \gamma \} .
\end{matrix}
\end{equation}

\chapter{The von Neumann universe}\label{chap:vNu} 

\section{The Beginning of Infinity}\label{sec:basic}

For convenience of the reader, we recall the construction of the {\em von Neumann universe}
$\vN$, in the spirit of
"naive set theory" (\cite{Ha}), not going into details of axiomatic set
theory. 
The von Neumann universe is the "universe of all pure sets", that is, the universe
 of all sets all of whose elements are sets.
To construct it,  start from (0),  and build up "stages" using
(1) and (2): 
\begin{enumerate}
\item[(0)]
(Empty set.)
There exists a set $\eset$ having no elements.
\item[(1)]
(Power set.)
For every set $M$, there is a set $\cP(M)$ whose elements are all the subsets of $M$.
Note that  $A \subset B$ implies $\cP(A) \subset \cP(B)$. 
\item[(2)]
(Ascending unions.)
Assume $M$ is a {\em chain} (i.e., a set of sets which is totally ordered for inclusion:
 for all $a,b \in M$, we have
$a \subset b$ or $b \subset a$).
Then there is  a set $S = \bigcup_{a \in M} a$ which is the union of all sets $a \in M$.\footnote{This 
 seems more intuitive than the stronger
\href{https://en.wikipedia.org/wiki/Axiom_of_union}{\em axiom of union} that $\cup M$ exists for any
set of sets.}
\end{enumerate}
Thus, the beginning of  this hierarchy is defined as follows:  let $\vN_0 := \eset$, and
\begin{align*}
\vN_1 & = \cP(\eset) =  \{ \eset \},
\\
\vN_2 & = \cP (\cP(\eset)) = 
\cP( \{ \emptyset \}) = \{ \eset , \{ \eset \} \} ,
\\
 \vN_3 &= \cP(\{ \emptyset , \{ \eset \}  \}) = 
 \{ \eset , \{ \eset \} , \{ \{ \eset \} \} , \{ \eset , \{ \eset \} \} \} ,
 \\
 \vN_{n+1} & = \cP(\vN_n) .
 \end{align*} 
Since $\vN_0 = \eset  \subset \vN_1$, we get by complete induction that
$\vN_n \subset \vN_{n+1}$, for all $n\in \N$. Therefore, by Principle (2), the union
$\bigcup_{n\in \N} \vN_n$ should exist. 
An element $x$ of this union is called {\em of rank $n$} if
$n$ is the smallest integer such that $x \in \vN_{n+1}$.
 To get a more intelligible notation, we let 
\begin{align}
\I &  :=  \{ \eset \} =  \{ \{ \} \}, 
\\
\II & := \{ \I \} = \{ \{ \eset \} \} = \{   \{ \{ \} \} \} ,
\\
\III & :=  \{ \eset , \I \} = \{ \eset , \{ \eset \} \} =\{ \{ \}, \{ \{ \} \} \} ,
\end{align} 
so $\vN_1 = \{ \eset \}$, $\vN_2 = \{ \eset , \I \}$, 
$\vN_3 = \{ \eset,\I,\II,\III\}$.
The only pure set of rank $0$  is $\emptyset$,  the only one 
of rank $1$ is $\I$, there are $2=4-2$ sets
of rank $2$, namely $\II$ and $\III$,  and  there are
$16 - 4 = 12$ sets of
rank $3$.  Let us give their list, with order number in 
\href{https://en.wikipedia.org/wiki/Lexicographic_order#Colexicographic_order}{\em
colex-order}, as follows. 
By "depth 1" we mean the description of a set by a plain list of its elements,
by "depth 2 we replace again each of those elements by its list of elements, and we leave
to reader to give a "depth 3"-description, replacing $\I$ by $\{ \eset \}$. 
As symbols for these sets we use roman letters. Horizontal lines indicate stratification by rank.

\begin{table}[!h] \caption{The beginning of infinity}\label{table:beginning}
\begin{center}
\begin{tabular}{|*{10}{c|}}
\hline
order number & symbol & rank  & depth $1$ & depth $2$ 
\\
\hline
$0$ & $\eset$ & $0$ & $\eset$ & $\eset$ 
\\
\hline
$1$ & $\I$ & $1$ & $\{ \eset \}$ &  $\{ \eset \} $ 
\\
\hline
$2$ & $\II$ & $2$ & $\{ \I \}$ & $\{ \{ \eset\}\}$ 
\\
$3$ & $\III$ & $2$ & $\{ \I,\eset \}$  & $\{ \eset , \{ \eset \}\}$ 
\\
    \hline
$4$ & IV & $3$ & $\{ \II \}$   & $\{ \{ \I \} \}$ 
\\
 $5$ & V & $3$ &  $\{ \II,\eset \}$ & $\{ \{ \I \}, \eset\}$ 
 \\
 $6$ & VI &  $3$ &  $\{ \II,\I \}$ & $\{ \{ \I\}, \eset \}$ 
 \\
 $7$ &VII  &  $3$ &  $\{ \II,\I ,\eset\}$ & $\{ \{ \I \}, \{ \eset  \}, \eset \}$  
 \\
 $8$ &VIII  & $3$ & $\{ \III \}$   & $\{ \{ \I,\eset \} \}$
\\
 $9$ & IX & $3$ & $\{ \III , \eset \}$   & $\{ \{ \I,\eset \} ,\eset \}$
\\
 $10$ & X  & $3$ & $\{ \III , \I \}$   & $\{ \{ \I,\eset \},\{ \eset \}  \}$
\\
 $11$ & XI & $3$ & $\{ \III , \I, \eset  \}$   & $\{ \{ \I,\eset \} ,\{ \eset \} , \eset \}$
\\
 $12$ & XII  & $3$ & $\{ \III , \II  \}$   & $\{ \{ \I,\eset \} , \{ \I \}  \}$
\\
 $13$ &  XIII & $3$ & $\{ \III , \II ,\eset \}$   & $\{ \{ \I,\eset \} , \{ \I \} , \eset \}$
\\
 $14$ & XIV  & $3$ & $\{ \III , \II ,\I \}$   & $\{ \{ \I,\eset \} , \{ \I \} , \{ \eset \}   \}$
\\
 $15$ & XV  & $3$ & $\{ \III , \II ,\I ,\eset \}$   & $\{ \{ \I,\eset \} , \{ \I \} , \{ \eset \}, \eset   \}$
\\
\hline
$16$ & XVI & $4$ & $\{ \mbox{ IV } \}$ &  $\{ \{ \II \}\}$ 
\\ 
    \end{tabular}
\end{center}
\end{table}
\nin 
The elements with order number $0,1,3,11,\ldots$, that is,
$N_{n+1} = 2^{N_n} + N_n$, are the {\em von Neumann ordinals}, and
those with order number $0,1,2,4,16,\ldots$, that is,
$Z_{n+1} = 2^{Z_n}$, are the {\em Zermelo ordinals}, see below, 
 Remark \ref{rk:Zermelo}.

\section{Von Neumann ordinals and von Neumann stages}\label{sec:vNo}\label{sec:vN-hierarchy}

The general von Neumann stages are indexed by {\em ordinal numbers},
which form a proper class $\On$ generalizing  the natural numbers $\N$. 
Like the natural numbers, ordinals are {\em well-ordered}: Every non-empty set of ordinals 
contains a minimal element.
Von Neumann's insight was that there exists a natural "model" of the ordinals given by pure sets,
such that the order relation gets modelized by the element-relation. We shall use this model. 
However, since there exist other models of ordinals by pure sets (see
Remark \ref{rk:Zermelo}), we will call them  the
{\em von Neumann ordinals}, and will write $\alpha_\vn$ instead of $\alpha$,
if necessary. 
We use the same three principles as in the construction of $\vN$: start with
$0_\vn =\eset$, and define the
{\em (von Neumann) successor} of an ordinal $\alpha_\vn$ by
\begin{equation}
(\alpha+1)_\vn := \alpha_\vn \cup \{ \alpha_\vn\} .
\end{equation}
Thus,
\begin{equation*}
1_\vn = \{ \eset \} = \{ 0_\vn \}, \quad
2_\vn = \{ 0_\vn, 1_\vn \} = \{ \eset, \{ \eset \} \} , \quad
3_\vn = \{ 0_\vn,1_\vn,2_\vn\} \ldots
\end{equation*}
where we have already met $1_\vn = \I$ and $2_\vn = \III$ and 
 $3_\vn = $XI (the set with order number 11 in Table \ref{table:beginning}).
An ordinal $\beta$ is a successor ordinal,
$\beta = \alpha + 1$, if and only if it has maximal element (namely, $\alpha$).
An ordinal $\beta$ that does not have a maximal element 
is called a {\em limit ordinal}. 
Every limit ordinal $\beta$ is the union (supremum)
of its preceding ordinals:
\begin{equation}
\beta = \bigcup_{\alpha \in \beta} \alpha = \cup \beta .
\end{equation}
In particular,
we consider $\eset$ as limit ordinal, and 
 the smallest infinite ordinal $\omega$ is the next limit ordinal:
\begin{equation}
\omega  = \omega_0 = \bigcup_{n \in \N } \, n_\vn = \{ 0_\vn,1_\vn,2_\vn,\ldots\} = \N .
\end{equation}
Existence and consistency of these constructions is ensured by 
axiomatic set theory, which we will not recall here.\footnote{ In particular, the  "NBG-versus-ZF discussion" (cf.\ \cite{A}, p.15) is unimportant for the present 
work:
all  "proper classes" we deal with reduce to honest sets, by intersecting them with
$\alpha$-stages.  }
 As a result, the class $\On$ of von
Neumann ordinals is a subclass of $\vN$, such that each rank in $\vN$ is represented by exactly
one ordinal, and it satisfies:
\begin{enumerate}
\item
$\eset$ is an ordinal,
\item
for two ordinals $\alpha,\beta$, either
$\alpha = \beta$, or $\alpha \in \beta$, or $\beta \in \alpha$,
\item
if $\alpha \in \beta$ and $\beta \in \gamma$, then $\alpha \in \gamma$,
so $\in$ defines a total order on $\On$,
\item
every non-empty set of ordinals contains a smallest element.
\end{enumerate}
It follows that $\alpha \leq \beta$ iff $\alpha \subset \beta$,
and $\alpha \cup \beta$ is  the ordinal $\max(\alpha,\beta)$ and
$\alpha \cap \beta$ the ordinal $\min(\alpha,\beta)$.
Not only the chain of ordinals has no end, but also the chain of limit ordinals has no end:
by induction, the ordinals $\omega + n$ for $n\in \N$ are defined, etc., and we get, e.g.,
\begin{align*}
 \omega \cdot 2 & := \cup_{n\in \N} (\omega + n),
 \\
 \omega^2 & := \cup_{n\in \N} (\omega \cdot n), \ldots , 
 \\
 \omega^\omega &:= \cup_{n\in \N} (\omega^n), \ldots,
 \\
 \eps_0 & := \omega \cup \omega^\omega \cup \omega^{\omega^\omega} \cup \ldots .
 \end{align*}
 There are ever and ever larger limit ordinals, and beyond $\eps_0$ it becomes increasingly
 difficult to define them: there is no "general algorithm" to do this. 
Now we can more formally define the construction of $\vN$:

\begin{definition}\label{def:vNu}
The {\em von Neumann hierarchy} is the collection of sets 
$\vN_\alpha$, for each ordinal $\alpha$, given by:
for successor ordinals,
\begin{equation*}
\vN_{\alpha + 1} := \cP (\vN_\alpha) ,
\end{equation*}
and for limit ordinals: 
\begin{equation*}
\vN_\alpha := \bigcup_{\beta \mbox{  \tiny ordinal}, \beta  < \alpha} \vN_\beta .
\end{equation*}
\end{definition}

Set theory tells us that
we then  have the principle of
\href{https://en.wikipedia.org/wiki/Transfinite_induction}{\em transfinite induction}, or
\href{https://en.wikipedia.org/wiki/Epsilon-induction}{\em principle of $\epsilon$-induction}:
{\em assume a property $P(\alpha)$ 
 holds for an ordinal $\alpha$ whenever $P(\beta)$ holds for all
$\beta \in \alpha$; then $P(\alpha)$ holds for all ordinals}. 
Or:
{\em Assume a property holds for a pure set $x$ whenever it holds for all elements of $x$.
Then it holds for all pure sets.} 

\ssk
We conclude this section  by some general remarks.
By {\em (naive) pure set theory} we mean a study of things one can do with pure sets by taking seriously
the definition of the von Neumann hierarchy. That is,  retain and study 
features of sets seen as elements of $\vN$, for instance,
the {\em rank} of a set, or the theory of $\On$ and of  $\No$.
We'll see in Chapter \ref{chap:Games} that pure set theory is just another way of looking
at the {\em theory of impartial combinatorial games}, and then propose a setting of 
 "pure set theory with two different kinds of membership
relation", to cover the
{\em theory of partizan games}. 
Results can be interpreted as being about "games", or about "pure sets",
whatever the reader may prefer.

\begin{remark}[Zermelo-ordinals]\label{rk:Zermelo}
 \href{https://mathoverflow.net/questions/273292/where-did-zermelo-first-model-the-natural-numbers-by-iterates-of-the-singleton-o}{Another model of the ordinals by pure sets is attributed to E.\ Zermelo}: 
start with $0_Z := \eset$, and define successors by
\begin{equation}
(\alpha + 1)_Z := \{ \alpha_Z \} .
\end{equation}
The literature is not clear about Zermelo's definition of limit ordinals $\lambda_Z$.
One possibility would be
\begin{equation}
\lambda_Z := \cup_{\alpha < \lambda} \alpha_Z .
\end{equation}
This system is less satisfying than von Neumann's, but would lead to the same
hierarchy $\vN$, with essentially the same rank function.
Another, quite radical, choice (indexed by $X$) would be to choose 
$\vN_\alpha$ itself as ordinal
$\alpha_X$. The order number of $n_X$ in Table \ref{table:beginning} is 
$w_{n+1}-1$ with $w_n$ as follows:
\end{remark}

\begin{remark}
The cardinality $w_n$ of the stages $\vN_n$ for $n \in \N$ growths by
\href{https://en.wikipedia.org/wiki/Tetration}{\em tetration} (see Section
\ref{app:A}): it
 is given inductively by
$w_0 = \vert \vN_0 \vert = 0$, and
$w_{n+1} = 2^{w_n}$, so 
$w_n=H_4(2,n)$, 
starting with the values 
$$
\begin{matrix}
n : &  & 0 & 1 & 2 & 3 & 4 & 5  & 6\\
w_n :& & 0 &1& 2 & 4& 16 & 
2^{16} = 65536    & 2^{65536} \simeq 10^{20.000} 
\end{matrix}
$$
As has been remarked,  $w_6$ exceeds by far the 
\href{https://en.wikipedia.org/wiki/Eddington_number}{number of atoms in the physical universe},  
so the von Neumann hierarchy certainly has "room enough" to modelize any physical system one
may imagine.
My point of view is that the stages $\vN_n$  look like 
interesting mathematical objects in itself, and that it may be worth studying them as
algebraic structures in their own right (Chapter \ref{chap:No-theory}).
\end{remark}

\section{Ordinal arithmetic, and hyperoperations}\label{sec:ordinalarithmetic}

By induction,
the successor operation $\alpha \mapsto \alpha + 1$ on ordinals gives rise to 
several  binary  operations on ordinals:
 \begin{enumerate}
 \item
Usual {\em ordinal arithmetic}, going back to Cantor himself: we will denote 
\href{https://en.wikipedia.org/wiki/Ordinal_arithmetic#Addition}{\em ordinal addition} by $\oplus$ and
\href{https://en.wikipedia.org/wiki/Ordinal_arithmetic#Multiplication}{\em ordinal multiplication} by
$\otimes$. These operations are associative and
 "continuous in the second argument", but fail to be commutative.
\item
The so-called \href{https://en.wikipedia.org/wiki/Ordinal_arithmetic#Natural_operations}{\em natural
sum and natural product}, going back to Hessenberg, and which we will denote by $+$ and
$\cdot$. They are commutative, associative, and distributive, but fail to be continuous in one of 
the arguments.
\end{enumerate}
Our choice of notation is dictated by the link with surreal numbers:
the Field operations there will correspond to $+$ and $\cdot$, and not to  $\oplus$ and $\otimes$.

\subsection{Ordinal arithmetic}\label{sec:ordinalarithmetic}

\begin{definition}
By transfinite induction, for a pair of ordinals $(\alpha,\beta)$, we define

\ssk
{\em ordinal addition} $\alpha \oplus \beta$ by $\alpha \oplus 0 = \alpha$, and

$\alpha \oplus (\beta + 1) := (\alpha \oplus \beta) + 1$

$\alpha \oplus \beta := \bigcup_{\beta' < \beta} (\alpha\oplus \beta')$ if $\beta$ is a limit
ordinal,

\ssk
{\em ordinal multiplication} $\alpha \otimes \beta$ by
$\alpha \otimes 0 = 0$,

$\alpha \otimes (\beta + 1) := (\alpha \otimes \beta) \oplus \alpha$

$\alpha \otimes \beta := \bigcup_{\beta' < \beta} (\alpha\otimes \beta')$ if $\beta$ is a limit
ordinal,

\ssk
{\em ordinal exponentiation} $\alpha \ootimes \beta$ by 
$\alpha \ootimes 0 = 1$, and 

$\alpha \ootimes (\beta + 1) := (\alpha \ootimes \beta) \otimes \alpha$

$\alpha \ootimes \beta := \bigcup_{\beta' < \beta} (\alpha\ootimes \beta')$ if $\beta$ is a limit
ordinal. 
\end{definition}

\begin{theorem}
When restricted to $\N$,  these operations give the  usual
operations $+,\cdot,\alpha^\beta$ of natural numbers.
For general ordinals, these operations are no longer commutative,
but $\oplus$ and $\otimes$ are still associative, and they are
{\em left distributive}, i.e.,:

$\alpha \otimes  (\gamma \oplus \beta) = (\alpha \otimes \gamma) \oplus (\alpha \otimes \beta)$,

$\alpha \ootimes  (\gamma \oplus \beta) = (\alpha \ootimes \gamma) \otimes (\alpha \ootimes \beta)$.
\end{theorem}

\begin{proof}
Textbooks on set theory.
(See, e.g.,
\url{https://www.math.uni-bonn.de/ag/logik/teaching/2012WS/Set%20theory/oa.pdf},
for full details.)
\end{proof}

\nin To illustrate non-commutativity:

$1 \oplus \omega =\bigcup_{n\in \N} (1+n) = \omega$,

$\omega \oplus 1 =\omega +1$.

 $
 2 \otimes \omega = \cup_{n\in \N}  (2n) = \omega$,
 
 $
 \omega \otimes  2  = \omega \oplus \omega > \omega$.

\nin
Some more properties:
ordinal addition admits {\em left subtraction}:
when $\alpha \leq \beta$, there exists a unique $\gamma$ such that
$\beta = \alpha \oplus  \gamma$. 
It follows that every ordinal $\alpha$ has a unique decomposition
\begin{equation}\label{eqn:decomposition}
\alpha = \lambda \oplus  n = \lambda(\alpha) \oplus  n(\alpha), \quad
\lambda \mbox{ limit ordinal, }  n \in \N .
\end{equation}
Also, there is a {\em left division with remainder}: given $\alpha$, and
 $\beta \not= 0$, we can write
\begin{equation}
\alpha = \beta \otimes u \oplus v, \mbox{  with unique }  u \mbox{  and } v < \beta.
\end{equation}
Taking $\beta = \omega$, we get 
$\alpha = \omega \otimes u  \oplus v$, with $v \in \N$.
Decomposing $u$ again in this way, etc., we get after finitely many steps 
 the \href{https://en.wikipedia.org/wiki/Ordinal_arithmetic#Cantor_normal_form}{\em Cantor normal form} of an ordinal
\begin{equation} \label{eqn:CantorNoFo}
\alpha = \bigoplus_{i=1}^N \omega^{u_i} \otimes v_i,
\mbox{ where } u_1 > \ldots > u_N \geq 0, \,  v_i \in \N \setminus \{ 0 \}.
 \end{equation}
Continuity (\ref{eqn:lim}) has the good effect that "infinite sums" make no problem:

\begin{definition}\label{def:infiniteoplus}
Assume given ordinals $n_\alpha$ for all ordinals $\alpha < \gamma$. Then the (generically) infinite sum
$\oplus_{\alpha < \gamma} n_\alpha$ is defined by transfinite induction via:
$$
\bigoplus_{\alpha < \gamma} n_\alpha := \Bigl\{
\begin{matrix}
\oplus_{\alpha < \beta} \, n_\alpha \oplus n_\beta & \mbox{ if } & \gamma = \beta + 1
\\
\cup_{\beta < \gamma} (\oplus_{\alpha < \beta} n_\alpha )& \mbox{ if } & \gamma \mbox{ is a limit ordinal.}
\end{matrix}
$$
\end{definition}

\subsection{The hyperoperations}\label{app:A}\label{app:hyperoperations}


The sequence of operations $\oplus,\otimes,\ootimes$ can be extended, to finite or 
infinite rank. 
For  $a,b,n \in \N$, the \href{https://en.wikipedia.org/wiki/Hyperoperation}{\em classical hyperoperations}
$H_n(a,b) \in \N$ are defined by:

\nin (1)
initial conditions
$$
\begin{matrix}
H_0 (a,b) = b+ 1, &  H_1(a,0) = a, &
 H_2(a,0)=0,
 &H_n(a,0) = 1 \mbox{ if  }n  \geq 3, &
 \end{matrix}
 $$
\nin (2)  recursion formula 
\begin{equation}\label{eqn:recursion}
H_{n+1}(a,b+1) = H_n(a,H_{n+1}(a,b)) .
\end{equation}
Then, by complete induction, it follows, among other things, that

$H_1(a,b) = a+b$ (sum),

$H_2 (a,b) = ab $ (product),

$H_3(a,b)= a^b$ (exponentiation),

$H_4 (a,n) =  a^{(a^{(a^{... })})}$ with $n$ terms: this operation is called {\em tetration},

$H_n(a,1) = a$ if $n \geq 1$, 

$H_{n+1}(a,2) = H_n(a,a)$; for instance $a^a = H_3(a,a)=H_4(a,2)$.

\begin{remark} 
Following  Doner and  Tarski \cite{DT}, the hyperoperations can be extended to
ordinal numbers $a,b,n$.  
We shall keep $n$ finite,
and define $H_n(\alpha,\beta)$ for ordinals $\alpha,\beta$, as follows:
$H_0$ is just the ordinal successor function.
The recursion formula 
(\ref{eqn:recursion}) remains the same for successor ordinals $\beta+1$.
To define $H_n(a,b)$ for limit ordinals $b$, we  demand {\em continuity in $b$}: 
\begin{equation}\label{eqn:lim}
H_n(a,b) = 
H_{n} (a , \bigcup_{b' < b} b') = \bigcup_{b'<b} H_n(a,b') .
\end{equation}
By transfinite induction, $H_n(a,b)$ is then defined for all ordinals $a,b$ ($n$ finite).

Then $H_1 (a,b) = a \oplus b$ is ordinal addition, but
$H_2(a,b)$ would fail to define $a \otimes b$, because of non-commutativity of $\oplus$.
For this reason,
Doner and Tarski use a slightly different definition:
their operations $O_n(a,b)$ are defined by the recursion
 $$
 O_{n+1}(a+1,b)= O_n(O_{n+1} (a,b),a)
 $$
 and the initial datum $O_0(a,b) = a \oplus b$. 
 Then $O_1$ is indeed ordinal multiplication, and
 $O_2$ is ordinal exponentiation
 (but $O_3$ is not ordinal tetration). 
 
The first of  Cantor's 
\href{https://en.wikipedia.org/wiki/Epsilon_number}{\em epsilon-numbers} 
can be defined by $\eps_0 = H_5(\omega,2)$. 
It cannot be finitely expressed by addition, multiplication and exponentiation of lower ordinals.
And there are ever and ever 
\href{https://en.wikipedia.org/wiki/Large_countable_ordinal}{larger ordinals}
that cannot be attained from below by the sequence of hyperoperations.
\end{remark}

\subsection{Hessenberg  sum and product of ordinals}\label{ssec:natural}

Besides the (Cantor) operations  of ordinal calculus, there are also
 the {\em Hessenberg}, or {\em natural}, operations on ordinals (sum and product), which are
 associative and {\em commutative}.
 There are several different definitions of these operations, see e.g., \cite{Alt, Ca}.

\subsubsection{Definition in terms of surreal numbers}

The Hessenberg operations are the restriction of sum and product from surreal
numbers $\No$ to the (Conway) ordinals. 
For this reason we denote them by $+$ and $\cdot \quad$ . 
Conway's construction does not need a prior definition of the Hessenberg operations, so
it can indeed be used to define them, and the reader may skip the following.
However,  
 Gonshor \cite{Go}, p.\ 13, and Alling \cite{A}, p. 133, use the natural sum 
 in the  formalisation of their inductive proofs, so they assume it to be defined beforehand.
 (From a logical viewpoint, this seems to be unnecessary, and it is merely a technical
 question of organizing inductive proofs on several variables.)

\subsubsection{Definition in terms of the Cantor normal form}

 \href{https://en.wikipedia.org/wiki/Ordinal_arithmetic#Natural_operations}{This seems to be the most common definition}: 
using the Cantor normal form
(\ref{eqn:CantorNoFo}), ordinals are added and multiplied like "polynomials in $\omega$'', using
associativity and commutativity:
let $\alpha = \oplus_{i=1}^k  \omega^{\alpha_i}$ and
$\beta = \oplus_{j=1}^\ell  \omega^{\beta_j}$, with
$\alpha_k \geq \ldots \geq \alpha_1 \geq 0$ and
$\beta_\ell \geq \ldots \geq \beta_1 \geq 0$.   Then
\begin{equation}
\alpha + \beta : = \oplus_{i=1}^{k+\ell}
\omega^\gamma_i 
\end{equation}
where $\gamma_1,\ldots,\gamma_{k+\ell}$ are the exponents 
$\alpha_1,\ldots,\alpha_k,\beta_1,\ldots,\beta_\ell$ sorted in nonincreasing order. 
Using the Hessenberg sum, the Hessenberg product is defined by
\begin{equation}
\alpha \cdot \beta :=
\sum_{1\leq i \leq k\atop 1 \leq j \leq \ell}
\omega^{\alpha_i + \beta_j}.
\end{equation}
The operations $+$ and $\cdot$ are commutative, associative and distributive. 

\subsubsection{Order theoretic definition}

Carruth (\cite{Ca}) shows that $\alpha + \beta$ is a {\em mimal natural
well-order} on the disjoint union of $\alpha$ and $\beta$, and
$\alpha \cdot \beta$ a {\em minimal natural well-order} on the Cartesian
product $\alpha \times \beta$.
For our purposes, this approach is unsuitable, since the notions of disjoint union
and Cartesian product are more problematic in pure set theory than the operation 
we wish to define.

\subsubsection{Definition by transfinite induction}

Avoiding the use of the Cantor normal form, one may define the Hessenberg operations inductively,
by transfinite $\in$-induction.
According to \cite{Alt}, Theorem 2.4, we have
\begin{align}\label{eqn:O1}
\alpha + \beta & = \min \{ \gamma \in \On \mid \,
\forall \beta' < \beta, \forall \alpha' < \alpha :
\, \alpha + \beta' < \gamma, 
\, \alpha' + \beta < \gamma \} ,
\\
\alpha \cdot \beta & =
\min \{ \gamma \in \On  \mid  \forall \beta' < \beta, \forall \alpha' < \alpha :
\, \alpha \cdot \beta' + \alpha' \cdot \beta < \gamma + \alpha' \cdot \beta' \} .
\end{align}
These properties can be used to define first the operation $+$ by $\in$-induction
-- indeed,  this is the definition mentioned at the  wikipedia page :
{\em
 We can also define the natural sum of $\alpha$ and $\beta$  inductively (by simultaneous induction on 
 $\alpha$ and $\beta$) as the smallest ordinal greater than the natural sum of $\alpha$ and 
 $\gamma$  for all $\gamma < \beta$  and of $\gamma$ and $\beta$ for all 
 $\gamma < \alpha$ }. 
In a second step, using the sum, this defines the product, again by $\in$-induction. 
 In \cite{Alt}, loc.cit., both formulae are attributed to Conway, \cite{Co01}, but they do not appear there in 
 an explicit way. 
 Certainly, the idea behind these formulae is the one explained in \cite{Co01}:
 these are the "simplest" conditions making true the desired monotonicity
 $\alpha + \beta' < \alpha + \beta$, $\alpha' +\beta < \alpha + \beta$, respectively,
 $(\alpha - \alpha') \cdot (\beta - \beta') > 0$. 
 However, Conway does not state these formulae, and
 a formal proof not involving surreal numbers would be quite technical
 (cf.\  \href{https://math.stackexchange.com/questions/4852808/hessenberg-sum-natural-sum-of-ordinals-definition}{question 4852808 on math.stackexchange}).
This definition should be compared with the definition of the nimber-sum on the ordinals, which follows
the same ideas, but by dropping the positivity condition (\cite{Co01}, Chapter 6, see our Section 
\ref{sec:Nimbers}) . 
By the way, the Cantor sum is described by a similar formula:
 \begin{align}\label{eqn:O2}
\alpha \oplus  \beta & = \min \{ \gamma   \mid \,  \,   \, \alpha \leq \gamma, \,
\forall \beta' < \beta: \, \alpha + \beta' < \gamma
\, \} ,
\end{align}
the "simplest" conditions making true 
$\alpha + \beta' < \alpha + \beta$ and  $\alpha \leq \alpha + \beta$ (we state this last condition 
explicitly in order to include the case $\beta = \eset$,  to get
$\alpha\oplus \eset = \alpha$, and not "$=\eset$"), 
but it seems that there is no such simple formula describing the Cantor product (the reason being
that one-sided distributivity does not give strong enough inequalities). 
Finally, let us note that Altmann in \cite{Alt} investigates other operations on ordinals, 
called the {\em  Jacobsthal (super) operations}, such as exponentiation
 -- this could be even further extended to higher
hyperoperations. See also \cite{Go2}.

\subsubsection{Definition of "sums of pure sets", Conway-style}

Remarkably, the operations $+$ and $\oplus$ extend to associative operations on the whole
von Neumann universe.
Such  "Conway-style" definitions really belong already to Chapter \ref{chap:Games},
 see Theorems \ref{th:Conway-sum} and \ref{th:Conway-style}.
However, concerning $\cdot$ and $\otimes$, the situation is  more complicated:
"Conway-style" definitions still work, but the operations thus defined tend to have less nice
properties.

\chapter{A construction of Conway numbers}\label{chap:Conway}\label{chap:numbers}

\section{Numbers and their quanta}\label{sec:No}

We recall  basic definitions and examples  already given in the introduction:

\begin{definition}\label{def:Number}
A {\em (Conway) number},  {\em surreal number}, or just {\em number},
 is a set $x$ of (von Neumann) ordinals having a maximal element 
$\max(x)$. (In particular, $x$ is non-empty.)
This maximal element  is  called the {\em birthday of $x$} and denoted by 
$\Beta(x)$.
We say that 
{\em $x$ is older than $y$},  if $\Beta(x) < \Beta(y)$.\footnote{Conway says "simpler than" 
instead of "older than", but other authors  use the term "simpler than" for the
relation $\preceq$, see Def.\ \ref{def:descendant}; for this reason we rather avoid it.}
We say that an ordinal $\alpha$

--  is an {\em element in $x$} if
$\alpha \in x$ and $\alpha < \Beta(x)$; 
we then let $s_x(\alpha) = +$;

-- is a
{\em hole in $x$} if 
$\alpha \notin x$ and $\alpha <\Beta(x)$ ;
we then let $s_x(\alpha) = -$;

-- is
{\em not in $x$} if $\alpha \geq \Beta(x)$; then we let
$s_x(\alpha) = 0$.

\nin
The Function $s_x: \alpha \mapsto s_x(\alpha)$ is called the
{\em sign-expansion of $x$}
(cf.\ Def.\ \ref{def:sign-expansion}).
\end{definition}

The birthday $\Beta(x)$ is an element  "of $x$" , but not  "in $x$": one may think of
it as the  "boundary of $x$", being neither "in $x$" nor "outside of $x$". 

\begin{definition}\label{def:sharp}
The {\em opposite $x^\sharp$} of a number $x$ is the number obtained by exchanging
the signs $+$ and $-$ in the sign-expansion; that is, exchanging  "elements in $x$" and 
"holes in $x$":
$$
x^\sharp := \{ \alpha \mid \, \alpha < \Beta(x), \alpha \notin x \} \cup \{ \Beta(x) \} .
$$
\end{definition}

\begin{example}
The empty set $\eset$ is not a number.
The set $\N=\omega$ is not a number: in fact, no limit ordinal is a number since it has no
maximum.
No infinite subset of $\N$ is a number, since it has no maximum. 
\end{example}

\begin{definition}\label{def:CoReal} \label{def:Conway-real} 
A {\em short number} is a non-empty and finite subset of $\N$.

A {\em long Conway real} is a number of the form
$x = X \cup \{ \omega \}$, where
$X \subset \N$ is a subset that is both infinite and
co-infinite (i.e., $\N\setminus X$ is infinite).
The birthday of such a number is $\omega$.

A {\em Conway real} is either a short number, or a long Conway real.
\end{definition}

For instance, the set $P$ of prime numbers in $\N$ defines a long Conway
real $P \cup \{ \omega \}$.
We'll see that there is a structure-preserving bijection between
"usual"  reals $\R$, and the set $\R_\Co$ of Conway reals (Section \ref{sec:Conwayreals}).

\begin{example}
Every successor ordinal $\alpha + 1$ has a maximum, namely $ \alpha $, hence is a number. 
In particular, $n+1 = \{ 0,\ldots,n \}$ is a (short) number, and
$\I = \{ 0  \}$ is the oldest of all numbers.
There exist numbers that are finite sets but are not short: for instance,
$x = \{ \omega \}$ is such a number (it is not a real number). 
\end{example}

\begin{definition}\label{def:Conway-ordinal}
For every (von Neumann) ordinal $\alpha$, the corresponding {\em Conway ordinal} is
the number defined by
$\Alpha := \alpha_\Co := (\alpha + 1)_\vn$, see Equation (\ref{eqn:Conway-ordinal}).
Its birthday is $\alpha$.
\end{definition}

\begin{definition}\label{def:children}
To every number $x$, we associate two other numbers, its {\em right child} $x_+$ and
its {\em left child} $x_-$ :
\begin{align*}
x_+ & := x \cup \{ \Beta(x) + 1 \},
\\
x_- &:=  x \cup \{ \Beta(x) + 1 \} \setminus \{ \Beta(x) \} .
\end{align*}
The maximal element is $\Beta(x)+1$, whence, in both cases:
$\Beta(x_{\pm}) = \Beta(x)+1$.
\end{definition}

The preceding definition leads to the structure of the {\em binary tree of numbers}
-- see Section \ref{sec:number-tree}.

\subsubsection{Quanta}
To every number $x$ we associate certain ordinals or integers, which we call 
{\em quanta of $x$}. The most fundamental quantum of $x$ is its birthday $\Beta(x)$.
Another fundamental quantum of $x$ is its {\em tipping element (tipping point)},
the first place where a sign change occurs in the sign expansion:  

\begin{definition}[Tipping Point]\label{def:quanta}\label{def:tippingpoint}
The {\em tipping element} of a number $x$ is the ordinal
\begin{equation}
\tip(x) :=
\min \{ \alpha \in x \mid \,
\alpha + 1 \notin x, \mbox{ or }
\exists \beta \notin x : \alpha = \beta + 1 \} .
\end{equation}
By definition, $\tip(x)$ is an ordinal belonging to $x$,
and since $\Beta(x)$ satisfies the condition inside the set description,
such an ordinal exists, and
 we always have
 $\tip(x) \leq \Beta(x)$.
\end{definition}

\begin{definition}[Omnific integer]\label{def:integer}
The {\em (omnific)  integer part} of a number $x$ is the number
$$
[x]:= \{ \alpha \in x \mid \alpha \leq \tip(x) \}.
$$
An {\em omnific integer} is a number $x$ such that
$\tip(x)=\Beta(x)$, equivalently, $x = [x]$.
\end{definition}

\begin{definition}[Absolute) Sign]\label{def:sign}
The {\em (absolute) sign of  a number $x$} is
\begin{equation}
\sgn(x) := +1 \mbox{ if } 0 \in x, \quad
\sgn(x) := -1 \mbox{ if } 0 \notin x.
\end{equation}
\end{definition}


We'll give more explanations and define other quanta in due course.

\section{The class $\No$: the Conway hierarchy}\label{sec:Ccubes} 

Just like the ordinals $\On$, the Conway numbers form a class,
$\No$:

\begin{definition}[Limit and successors, stages, boundary, the class $\No$]\label{def:stages} $ $
\begin{itemize}
\item
A number $x$ is called a {\em successor number} if its birthday $\Beta(x)$ is a successor ordinal, 
and a {\em limit number} if $\Beta(x)$ is a limit ordinal.
\item
For every ordinal $\alpha$, the  {\em $\alpha$-stage of numbers} is the set
$$
\No_\alpha := \{ x \mid \, x \mbox{ is a number with } \Beta(x) <  \alpha \}
$$
Stages form a hierarchy:
$\alpha < \beta \Rightarrow \No_\alpha \subset \No_\beta$. Note  that
$\No_\alpha \subset \vN_{\alpha +1}$.
\item
For every ordinal $\alpha$, we define  the {\em boundary of $\No_\alpha$}
$$
\partial (\No_\alpha)   := \No_{\alpha + 1}\setminus \No_\alpha= 
  \{ x \mid \, x \mbox{ is a number with } \Beta(x) = \alpha \} .
$$
\item
By $\No$ we denote the (proper) class of all Conway numbers, formally:
$$
\No = \bigcup_{\alpha \in \On} \No_\alpha \subset \vN .
$$
\end{itemize}
\end{definition}

\begin{example}[Short numbers]\label{ex:short}
A non-empty subset $x \subset \N$ is a number iff it is finite, iff it is short.
The set of all short numbers is $\No_\omega$. 
Every short number, except $0_\Co$, is a successor number. 
For small values of $\alpha$, we have (with notation as in  Table \ref{table:beginning}): 
$\No_0=\eset$,
\begin{align*}
\No_1 &=  
  \{ \{ \eset \} \} = \{ \I \}  = \II , \\
\No_2 & =  \
\{ \{ 0 \}, \{ 1\}, \{ 0,1 \}\} =
\{ \I , \II, \III \} = {\mathrm{XIV}} ,\\
\No_3 &   =  \{ \{ 0 \}, \{ 1\}, \{ 0,1 \}, \{ 2 \}, \{ 0,2\}, \{ 1,2\}, \{ 0,1,2\} \} =
 \{ \I, \II, \III, {\mathrm{ VII, VIII, IX, X}} \} .
\end{align*}
One observes that the share of $\No_\alpha \subset \vN_{\alpha+1}$ is $\frac{1}{2}$ for
$\alpha = 1$, it is $\frac{3}{4}$ for $\alpha = 2$, it is 
$\frac{7}{16}$ for $\alpha = 3$, and then quickly tends to $0$.
Besides $\eset$, the "simpleset" pure sets which are {\em not} numbers, are
IV, V, VI, VII, XII, XIII, XIV, and  XV (those in Table \ref{table:beginning} 
 containing a $\II$ at depth $1$). 
\end{example}

\begin{remark}
The class of all Conway ordinals is the class of all von Neumann successor ordinals. 
There is a Bijection between Conway ordinals and von Neumann ordinals, just as there
is one between all natural numbers and all strictly positive natural numbers. In other terms,
the Birthday Map
$
\Beta : \No \to \On$
has a Section $\alpha \mapsto \alpha_\Co$.
Note also that the rank of $x$ in the von Neumann universe is  the successor of $\Beta(x)$, i.e.,
\begin{equation}
\rk(x) = \Beta(x) + 1.
\end{equation}
\end{remark}

\begin{remark}[Numbers in a well-ordered set]\label{rk:well-ordered}
For every well-ordered set $M$, one can define similarly the
set of {\em numbers in $M$},
$$
\No_M := \{ x \in  \cP(M) \mid  x \mbox{ has a maximum } \Beta(x) := \max(x) \ \}.
$$
Since $M$ is order isomorphic to some ordinal $\alpha$, this is just the definition of
$\No_\alpha$ in another form. 
This definition stresses the point of view that the theory of surreal numbers can be considered
as a topic in order theory, notably via the {\em  interplay of several partial orders}
that can all be defined on $\No_M$.
\end{remark}

\section{Conway reals}\label{sec:Conwayreals}

An advantage of our approach is that the imbedding of the  "usual"  reals $\R$ into the
Conway-hierarchy can be defined in a simple and algorithmic way.
Recall Definition \ref{def:Conway-real} of the set $\R_\Co$ of Conway reals.  We will
define two bijections, inverse to each other, 
\begin{equation}
\R_\Co \to \R, \, x \mapsto r(x), 
\qquad
\R \to \R_\Co, \, r \mapsto x(r),
\end{equation}
satisfying a relation saying that $x^\sharp$ corresponds to the usual negative:
\begin{equation}
r(x^\sharp) = - r(x), \qquad
x(-r) = (x(r))^\sharp .
\end{equation}
Therefore it suffices to define $r(x)$ only for $x$ with negative sign (i.e., $0 \notin x$),
or only for $x$ with positive sign (i.e., $0 \in x$).
The latter is the algorithm 
attributed at p.\ 31 \cite{Co01} to Elwyn Berlekamp, defined  for {\em positive} reals.
However, it turns out that for negative $x$ the formalisation is somewhat simpler.
The reason is that, in this case, the tipping element $\tip(x)$ (Def.\ \ref{def:quanta})
is equal to the minimal element $\min(x) \in x$, which is not the case for positive numbers.
The following algorithm starts by finding the integer part $[x]$ of a number $x$, and 
this in turn corresponds to determining the tipping element of $x$.

\begin{example}\label{ex:tip}
Let $x = \{ 2,3,5\}$. Then $\min(x) = 2 = \tip(x)$, and
$[x]=\{ 2 \}$.

\nin
Let $y=x^\sharp=\{ 0,1,4,5\}$.
Then $\tip(y) = 1 = \min(y^\sharp) - 1$, and $[y]= \{ 0,1\} = 1_\Co$.
\end{example}

\begin{theorem}[The  mirror Berlekamp algorithm]\label{th:Berlekamp}
There is a bijection from the (usual) reals $\R$ onto
the set $\R_\Co$ of Conway-reals, given as follows:
\begin{enumerate}
\item
Let $x$ be a Conway real with negative sign, i.e., $0 \notin x$, i.e., $\min(x)\in \N$, $\min(x)>0$.
Let $X = x\cap\N$. 
Define the "usual" real number
$$
r:= r(x) :=  - \min(x)   + \sum_{m \in X, m > \min(x) } \, (  \frac{1}{2} )^{m- \min(x)}  .
$$
If $0 \in x$, then the corresponding usual real is defined to be
$r(x) = - r(x^\sharp) $.
\item
In the other direction, 
let $r \in \R$ be negative, and expand
$$
r = - r_0 + \sum_{k=1}^{\infty} a_k 2^{-k}, \quad
r_0 \in \N^*, \, a_k \in \{ 0,1\} .
$$
Avoid ambiguity of such representation by replacing
 "long ends" ($\exists m: \forall k >m: a_k=1$) always by the corresponding
finite sum.
If the expansion is finite, the corresponding Conway real is the short number
$$
x = x(r) = \{ r_0 \} \cup \{ r_0 + k \mid k \in \N , a_k = 1 \} ,
$$
and if the expansion is infinite, the corresponding Conway real is the long real
$$
x = x(r) = \{ r_0 \} \cup \{ r_0 + k \mid k\in \N, a_k = 1 \} \cup \{ \omega \} .
$$
If $r >0$, then let  $x(r) = (x(-r))^\sharp$.
\end{enumerate}
Under this correspondence, short numbers correspond to the ring $\Z[\frac{1}{2}]$ of
dyadic rationals, the finite Conway ordinals correspond to $\N$, and 
$\{ n \}$ corresponds to $-n$, for all $n\in\N$.
Ordinal arithmetics on $\N_\Co$ corresponds  to usual arithmetics on $\N$.
\end{theorem}

\begin{proof}
Both maps $x \mapsto r(x)$ and $r \mapsto x(r)$ are well-defined and inverse to each
other: the main point is that, since $X$ in Def.\ \ref{def:Conway-real} is co-infinite,
 "long ends" do not appear, and hence the expansion of $r(x)$ is finite if, and only if,
 the set $x$ is finite (cf.\ also \cite{Go}, p.33). 
It is now obvious from the definitions that both constructions define inverse bijections.
The algorithm shows that $x = \{ n \}$ corresponds to  $r=-n$ (there is no fractional part),
so $n_\Co = \{0,\ldots,n\} = \{n\}^\sharp$ corresponds to $r=n$.
The statement about ordinal arithmetics is obvious from this correspondence.
\end{proof}

The following is an equivalent version of the preceding algorithm.
For positive reals, it corresponds to the original ``Berlekamp algorithm''.

\begin{corollary}\label{cor:Berlekamp}
The bijection from the preceding theorem can also be described as follows:
let $x \in \R_\Co$, $X= x \cap \N$.
If $x$ has positive sign, 
then
$[x]= \tip(x)$ is its integer part, and
$$
r (x) = [x] + \sum_{m \in X, m > \tip(x) }  \frac{1}{2^{m - 1 - \tip(x)}}.
$$
If $x$ has negative sign, then
$[x]=-\tip(x)$ is its integer part, and
$$
r (x) = [x] + \sum_{m \in x, m > \tip(x) } \frac{1}{2^{m - \tip(x)}} . 
$$
\end{corollary}

\begin{proof}
Notice that the definition of the tipping element matches exactly the behaviour of the usual
 integer part:
if $x\in \R_\Co$ is an (omnific) integer, then $\tip(x) = \Beta(x)$, corresponding to the
fact that $[-r] = - [r]$ for a ``usual'' integer.

If $r>0$ is not a usual integer, then $[-r] = - ([r] + 1)$, which corresponds to the fact that,
if $x$ is not an omnific integer, and $0 \in x$, then
$\tip(x) + 1 = \min(x^\sharp) $  (cf.\ Example \ref{ex:tip}; and Examples \ref{ex:numbers}).
\end{proof}


\begin{example}\label{ex:numbers}
The following Conway-reals $x$ give "usual" reals $r = r(x)$:

\ssk
$x = \{ 1,2\}$ gives $r = -1 + \frac{1}{2}= - \frac{1}{2}$,

$x = \{ 0,2\}$, with $\tip(x)=0$,  gives  $r = 0 + \frac{1}{2^{2-1}} = \frac{1}{2}$,

$x = \{ 1,2,\ldots,n\}$ gives $r = -1 + \sum_{m=2}^n 2^{m-1} = - \frac{1}{2^{n-1}}$

$x = \{ 0,n \}$, with $\tip(x)=0$,  gives $r = \frac{1}{2^{n-1}}$

$x = \{ 2,3,5, 7,11  \}$ gives
$r = - 2 + \frac{1}{2} + \frac{1}{2^3} + \frac{1}{2^5} + \frac{1}{2^9}  = - 1.34375$

$x = P \cup \{ \omega\}$, $P=$ set of prime numbers, gives
$r = -2 + \sum_{p\in P\atop p>2} 2^{2-p}$

$x = \{1,3,5,7,\ldots,\omega\}$ gives
$r = - 1 + \sum_{k=1}^\infty  \frac{1}{4^k} =-  \frac{2}{3}$

$x =\{ 0,2,4,\ldots,\omega \}$ gives $r=\sum_{k=0}^\infty \frac{1}{2^{1 + 2k}} = 
\frac{2}{3}$

\ssk\nin
Conversely, the following usual reals $r$ give Conway-reals $x$:

\ssk
$\sqrt{2} = 1.4142...$ gives
$\{ 0,1,4,5,7,9,\ldots , \omega \}$

$e = 2.71828...$ 
gives $\{ 0,1,2,4,6,7,9,\ldots , \omega \}$
\end{example}

\begin{example}["What are the elements of $\pi$?"]\label{ex:pi}
In his paper \cite{Lei}, Tom Leinster writes:
{\em...
accost a mathematician at random and ask them

{\em ‘what are the elements of $\pi$?’}, 

\nin
and they will probably assume they misheard you, or ask you what you’re talking about, or else tell you that 
{\em your question makes no sense}. If forced to answer, they might reply that real numbers have no elements. But this too is in conflict with ZFC’s usage of ‘set’: if all elements of $\R$ are sets, and they all have no elements, then they are all the empty set, from which it follows that all real numbers are equal...}

So, here is an answer to that question:
$
\pi = 3,1415926535897932...$
corresponds to the sign expansion given in \cite{WW}, or
\cite{CG}, p. 288, Figure 10.8, starting with

$\pi \simeq +++(+-)--+--+----++++++ - ++ - + ... $ 

\nin which gives (the first four elements just say that the integer part is $\tip(\pi) = 3$):

$
\pi_\Co = 
 \{ 0,1,2,3, 7 , 10, 14,15,16,17,18, 19 ,21,22,24   \ldots, \omega \}
$
\end{example}
 
\nin \nin
These examples suggest several number theoretic questions :

(1) How to characterize $x=x(r)$ corresponding to {\em rational} reals $r$?

(2) How to characterize $x=x(r)$ corresponding to {\em algebraic} reals $r$?

(3) How to characterize $x=x(r)$ corresponding to {\em transcendental} reals $r$?

(4) What can we say about $r=r(x)$ for a "random" sequence $x$?

\nin
Here, we'll just give an answer to question (1):
a Conway real is {\em rational} if, and only if,  it is short, or has a finite pattern repeated $\omega$ times.
This has already been proved by Moritz Schick (\cite{Sch}, 
Bemerkung 6.19).\footnote{For this and related questions, see also
\url{https://mathoverflow.net/questions/214354/nice-sign-expansions-of-special-surreal-numbers }.}
Before stating the result, let us consider some more examples of rational Conway numbers:

\begin{example}
Consider long Conway reals 
$x = \{ m+ 4 n \mid n \in \N \} \cup \{ 0,\omega \}$.
Using the Berlekamp algorithm, with $m = 2,3,4,5$, 

$x = \{ 0,2,6,10 ,\ldots,\omega \}$
gives
$r = \frac{1}{2} + \frac{1}{2^5} + \frac{1}{2^9} + \ldots =
\frac{1}{2} \sum_{k=0}^\infty \frac{1}{(2^4)^k} = \frac{8}{15}$

$x = \{ 0,3,7,11 ,\ldots,\omega \}$
gives
$r = \frac{1}{2^2} + \frac{1}{2^6} + \frac{1}{2^{10}} + \ldots =
\frac{1}{2} \frac{8}{15} = \frac{4}{15}$

$x = \{ 0,4,8,12 ,\ldots,\omega \}$
gives
$r = \frac{1}{2^3} + \frac{1}{2^7} + \frac{1}{2^{11}} + \ldots =
\frac{1}{2} \frac{4}{15}  = \frac{2}{15}$

$x = \{ 0,5,9,13 ,\ldots,\omega \}$
gives
$r = \frac{1}{2^4} + \frac{1}{2^8} + \frac{1}{2^{12}} + \ldots 
=\frac{1}{2} \frac{2}{15} 
= \frac{1}{15}$

\nin
To represent a rational $r = \frac{k}{15}$ with $0 < k < 15$ by a set $x$, 
write $k = k_0 2^0 + k_1 2^1 + k_2 2^2 + k_3 2^3$ in base $2$, with
$k_i = 0$ or $1$;
then take the union of the sets shown above, for all $i$ with  $k_i = 1$. 
For instance, to represent
$  \frac{1}{5}  =  \frac{2}{15} +  \frac{1}{15}$, we take the union of the last two
numbers shown, that is:

$x = \{ 0,4,5,8,9,12 ,13, \ldots,\omega \}$
corresponds to
$r = 
 \frac{2}{15} +  \frac{1}{15} =  \frac{1}{5} $.
 
\nin
To represent, e.g., $\frac{1}{10}$, shift the non-zero part of $x(\frac{1}{5})$ one step ``to the right''.
To represent, e.g., $\frac{3}{10} = \frac{1}{2} \cdot \frac{9}{15}$, take the representation of
$\frac{1+8}{15}$, and then shift it again ``one step to the right'':

$x = \{ 0,3 , 6, 7 , 10, 11 ,\ldots , \omega\} $ corresponds to to $\frac{3}{10}$.
\end{example}

\begin{theorem}\label{th:rationals}
A Conway real $x$ is rational if, and only if, its sign expansion starts with a finite sequence, 
and then repeats another finite sequence $\omega$ times, followed by zeroes;
equivalently, with notation to be introduced in Section \ref{sec:CoCa}, if and only if,
$$
x = y \oplus z \otimes \omega , \quad \mbox{ with short numbers } y, z .
$$
This representation becomes unique if we require $\Beta(y)$ and $\Beta(z)$ to be minimal 
(then call $y$ the "prefix", and $z$ the "repetitor" of $x$). 
\end{theorem}


\begin{proof}
Let $x$ be a Conway real having a representation as in the theorem.
By the Berlekamp algorithm, the corresponding real $r=r(x)$ will have a fractional part that
can be expressed by certain integer combinations of geometric series with ratio a power of
$\frac{1}{2}$, and so will be a rational number.

Conversely, every  "usual" rational
$r = \frac{a}{b}$ is thus obtained. To see this,
we may assume that $r>0$, since the statement is obviously invariant under taking negatives
(opposites), and moreover that
 $0 < r < 1$, by changing the  "prefix", if necessary.
If $b$ is a power of $2$, then $r$ is a short number, and conversely every short number in the unit interval
is obtained this way. 
If $b$ is odd, apply the following

\begin{lemma}
Every rational number $r = \frac{a}{b}$ with $0 <  r < 1$ and odd $b$ 
can be represented in the
form $r = \frac{c}{2^d - 1}$ with $0<c<2^d - 1$, for some $d \in \N$.
\end{lemma}

\begin{proof}
Since $b$ is odd, the element $2$ is invertible in $\Z/ b \Z$.
Thus it is of finite order in the group $(\Z/b \Z)^\times$, so
 there is $d \in \N$ with $2^d \equiv 1 \mod b$, so there is $k \in \N$ with
$2^d - 1 = k b$, whence
$\frac{a}{b} = \frac{ka}{kb} = \frac{ka}{2^d - 1}$.
\end{proof} 

In this case we proceed as in the example above (where $2^d - 1 = 2^4 - 1 = 15$): 
we use geometric
series with ratio $\frac{1}{2^\ell}$ for some $\ell \geq 1$.
If $\ell \in \N$ is the length of a periodic pattern $M$, then the sum
$\sum_{k=0}^\infty  2^{- \ell k + m} = \frac{1}{2^m} \frac{ 1}{1 - 2^{-\ell}} =
2^{\ell - m} \frac{1}{2^\ell - 1}
$
will appear.
Various $M_i$ with length $\ell$ give rise to various $m_i$, 
so the corresponding real number 
has the form of a finite sum
$\sum_i \frac{ 2^{\ell - m_i}}{2^\ell -1 } = \frac{ u}{ 2^\ell - 1}$, and all such numbers
are represented in the form given by the theorem.

When $b = 2^k b'$ with odd $b'$, then the Conway real corresponding
$\frac{a}{b'}$ can be realized as in the preceding step, and the Conway real corresponding to
$\frac{a}{b} = \frac{1}{2^k} \cdot \frac{a}{b'}$ can be deduced, as in the example.

Finally, uniqueness of an expression $x = y \oplus z \otimes \omega$ is best explained
by using the notation related to concatenation, see Section \ref{sec:CoCa}.
  \end{proof}

\section{The total order on $\No$}\label{sec:totalorder}

The bijection between $\R_\Co$ and $\R$ will be compatible with all relevant 
structures. Let us start with {\em total order}.

 \begin{definition}\label{def:discriminant}
 Recall that the {\em symmetric difference} of two sets $x$ and $y$ is
 $$
 x \Delta y = (x \setminus y) \cup (y \setminus x) =
 (x \cup y) \setminus (x \cap y) ,
 $$
 so $x=y$ iff $x \Delta y= \eset$.
 We say that an ordinal $\alpha$ {\em discriminates between two numbers $x$ and $y$}
 if $\alpha$ belongs either to $x$ or to $y$, i.e., if
 $\alpha \in x \Delta y$.    The {\em discriminant}\footnote{ We take this terminology from \cite{Ba}, 13.5:  "The smallest ordinal to
 discriminate between two Numbers is called their {\em discriminant}."}
 of two numbers $x \not= y$ is the ordinal
 $$
 \delta := \delta (x,y) := \min (x\Delta y).
 $$
 \end{definition}

\nin
The discriminant belongs to exactly one of $x$ or $y$. Let $x<y$ if it belongs to $y$: 

\begin{theorem}\label{th:order1}
The following prescription defines a  total order on $\No$: 
$x <  y$ iff 
$$
\boxed{
x\not=y,  \mbox{ and:  } \quad   \min(x \Delta y) \in y  }.
$$
Equivalently, this order can be characterized as the {\em lexicographic order on sign-expansions}
(cf.\ Def.\ \ref{def:sign-expansion}),
where the set $S = \{ +,-,0 \}$ is ordered by
$- < 0 < +$.
\end{theorem}

\begin{proof}
If $x\not=y$, 
the non-empty set $x \Delta y$ admits a minimum $\delta$ (principle of well-order).
This minimum must belong to $x \setminus y$ or to $y \setminus x$,
but cannot belong to both since these sets are disjoint. 
In particular, it belongs either to $x$ or to $y$.
Therefore, if $x \not= y$,
then either $x<y$ or $y<x$, whence antisymmetry and totality.
To prove transitivity,
assume $\eta := \min(x\Delta y) \in y$ and $\zeta := \min(y \Delta z) \in z$, and
distinguish cases: 
\begin{itemize}
\item
$\eta \notin z$.
Then $\zeta \leq \eta$, and
$\forall \xi \in x \setminus z$: 
(either $\xi \notin y$, then $\xi \geq \eta \geq \zeta$)(or
$\xi \in y$, then $y \geq \zeta$), so in this case 
$ \min(x \Delta z) = \zeta \in z$.
\item
$\eta \in z$.  We distinguish again:
\begin{itemize}
\item
If $\zeta \notin x$, then $\mu := \min(\zeta,\eta) = \min(x \Delta z) \in z$.
\item
If $\zeta \in x$, then $\zeta \geq \eta$.
It follows that $\forall \xi  \in x \setminus z: \xi \geq \eta$
(by distinguishing the two cases $\xi \in y ; \xi \notin y$), whence
$\eta = \min (x \Delta z) \in z$.
\end{itemize}
\end{itemize}

Recall the lexicographic order of sign-expansions:
$s <_\ell s'$ if there is an ordinal $\alpha$ such that
$s(\beta) = s'(\beta)$ for all $\beta <\alpha$, and $s(\alpha) < s'(\alpha)$.
When $s = s_x$ and $s'=s_y$ with $x \not= y$, then
($\beta < \alpha \Rightarrow \beta \notin x \Delta y$), so
$\min (x \Delta y) \geq \alpha$.
The condition $s(\alpha) <_\ell  s'(\alpha)$ gives three possible values for the pair
 $(s(\alpha),s'(\alpha))$:
\begin{enumerate}
\item
$(-,0)$, i.e., ($\alpha \notin x$ and $\max(x) > \alpha)$)
and $\alpha = \max(y)$, so $\min(x \Delta y)=\alpha =\max(y) \in y$.
\item
$(0,+)$, i.e., $\max(x)=\alpha$ and ($\alpha \in y$, $\max(y) >\alpha$),
so again $\min(x \Delta y) \in y$.
\item
$(-,+)$, i.e., $\alpha \notin x$, $\alpha \in y$, and
$\max(x),\max(y)>\alpha$, so $\min(x \Delta y) = \alpha \in y$.
\end{enumerate}
Therefore $x \leq_\ell  y$ implies $x \leq y$, and since both are total orders,
the converse then also holds, and both orders agree. 
\end{proof}

\nin
Remark.
The property defining the order can be rewritten as:
$x\leq y$ iff
\begin{equation}
x \subset y \,
\mbox{ or } \,  [(x \setminus y) \not= \eset \mbox{ and } (y \setminus x) \not= \eset \mbox{ and }
\min (x \setminus y) > \min(y \setminus x) ]   .
\end{equation}

\begin{definition}
We define the {\em absolute value} of a number $x$ by
$$
\vert x \vert := \Bigl\{ \begin{matrix} x & \mbox{ if } & x \geq 0 \\
x^\sharp & \mbox{ if } & x < 0 .\end{matrix} .
$$
\end{definition}

\begin{theorem}\label{th:order}
The total order $\leq$  has the following properties:
\begin{enumerate}
\item[(1)]
it  extends the partial order given by inclusion of subsets of $\On$,
\item[(2)]
the pseudo-complement map $\sharp$ is an order-reversing bijection fixing $0_\Co$, 
\item[(3)]
for each ordinal $\alpha$ and each surreal number $x$ with $\Beta(x)=\alpha$:

the 
right child $x_+$ is an immediate successor of $x$ for $\leq$ in $\No_{\alpha+1}$, and 

the left child
$x_-$ is an immediate predecessor of $x$ for $\leq$  in $\No_{\alpha+1}$.
  \item[(4)] 
The following inequalities are expressed by set-membership:

$
0_\Co \leq x \, \mbox{ if and only if } \, 0 \in x,
$
iff $\sgn(x) = 1$.

$
0_\Co \leq x < 1_\Co \, \mbox{ if and only if } \, (0 \in x)  \land ( 1 \notin x);
$
 iff: $\tip(x) = 0$.

$
1_\Co \leq x \, \mbox{ if and only if } \, \{ 0,1 \} \subset x. 
$

$\alpha_\Co \leq x$ if and only if $\alpha_\Co \subset x$.
\item[(5)] 
$\alpha_\Co$ is a maximal element in
$\No_\alpha$, and $\{ \alpha \}$ is a minimal element in $\No_\alpha$.
\end{enumerate}
\end{theorem}

\begin{proof}
(1):
If $x \subset y$, then $x \Delta y = y \setminus x$, and
$\min(x \Delta y) \in y$.

(2):
The sign-expansion of $x^\sharp$ is related to the one of $x$ by exchanging
$+$ and $-$, that is, by reversing the order on the set $S =\{ -,0,+\}$,
and thus $\sharp$ reverses the order of the lexicographic order.

(3) follows again from the definition of $\leq$:
clearly, $x_- < x < x_+$; and  if
$x_- < y < x$, we distinguish 
the cases $\Beta(y) \leq \alpha$ and $\Beta(y) =\alpha + 1$.
Both cases can be ruled out by the definition of $\leq$, so the interval
$]x_-,x[ \cap \No_{\alpha + 1}$ is empty; similarly for $]x,x_+[$.

(4) 
Since $0_\Co = \{ \eset \}$, the definition of the order shows that $0_\Co \leq x$  iff
$\eset \in x$, iff $\{ 0 \} = 0_\Co \subset x$. 
Similarly, for every Conway ordinal $\alpha_\Co$, we have
$\alpha_\Co \leq y$ iff 
$\alpha_\Co = \{ 0,\ldots, \alpha \}  \subset y$. 

By Definition \ref{def:tippingpoint} of $\tip(x)$, the condition
$\tip(x)=0$ amounts to $0\in x$, $1 \notin x$.

(5)  
$x \in \No_\alpha$ implies $x \subset \alpha_\Co$; now use (1) to get
$x \leq \alpha_\Co$. 

For the other statement, apply (2). 
\end{proof}

\begin{theorem}
The bijection $\R \to \R_\Co$ from Theorem \ref{th:Berlekamp} is order-preserving.
\end{theorem}

\begin{proof}
Since both $\sharp$ and $r\mapsto -r$ are order-reversing,
it suffices to prove that the restriction to negative reals is order-preserving.
Let $x , y$ be Conway reals, and $X = x \cap \N$, $ Y = y \cap \N$.
Assume $x < y < 0$, so $\min(x \Delta y) \in y$.
Then $\min(x) \geq \min(y)$ (for $\min(x) < \min(y)$ implies 
$\min(x \Delta y) = \min(x) \in x$).
If $\min(x) < \min(y)$, then certainly $r(x) < r(y)$, since the fractional part of the
sum belongs to the real interval $[0,1[$.
Assume that $\min(x) = \min(y) = M$ and consider the difference 
$$
r(y) - r(x) =  \sum_{m \in Y, m > M } \, (  \frac{1}{2} )^{m- M} -
 \sum_{n \in X, n > M } \, (  \frac{1}{2} )^{n- M} .
$$
It is positive if, and only if, the first term which does not cancel out belongs 
to an index $m \in Y$, and this index is precisely $\min(x \Delta y)$; thus
$r(y) > r(x)$ iff $y > x$. 
\end{proof}

\subsubsection{Infinite and infinitesimal numbers}

Now we know that $\No$ is totally ordered and that the continuum $\R \cong \R_\Co$ is an ordered
subset of $\No$, it is obvious that $\No$ must contain many "infinite" and many "infinitesimal"
numbers.

\begin{definition}\label{def:infinitesimal}
A number $x$ is said

-- positive  infinite, $x\gg 1$,  if $\forall n \in \N : x >n$,

-- negative infinite, $x\ll -1$, if $\forall n\in \N : x< - n$,

-- infinitesimal, $x \approx 0$,  if $\forall n\in \N$ : $\{ 1,2,\ldots,n \} < x < \{ 0 , n \}$.

\nin
We denote by $\bI$ the (proper) class of all infinitesimal numbers.
\end{definition}

\nin
Directly from Theorem \ref{th:order}, it follows that $x$

--  is negative infinite iff $x \cap \N = \eset$,

--  is positive infinite iff $x \cap \N = \N$,

-- infinitesimal  and positive iff  $x \cap \N = \{ 0 \}$, 

-- infinitesimal and negative iff $x \cap \N = \N^*$, where
$\N^* = \N \setminus \{ 0 \}$.

\begin{example}
Consider numbers  $x = X \cup \{ \omega \}$ where
$X = x \cap \N \subset \N$.

-- For $X = \eset$, we get $x = - \omega_\Co$,

-- for $X =\N$, we get $x =\omega_\Co$,

-- for $X = \{ 0 \}$, we get $\eps := \{ 0 , \omega \} \approx 0_\Co$, with $\eps > 0$,

-- for $X = \N^*$, we get $x = \eps^\sharp = - \eps \approx 0_\Co$,

-- for $X = \{ 0,1\}$, we get $x = \{ 0,1,\omega \}$, which will be $1_\Co + \eps$,

-- for $X = \{ 0,2,3,4,\ldots \}$, we get a number which will be $1_\Co - \eps$,

-- for $X = \{ 2,3,4,\ldots \}$, we get a number which will be $-1_\Co - \eps$,

-- for $X = \{ 1 \}$, we get a number which will be $-1_\Co + \eps$.

\nin
In general, if $X$ is finite (so is a short number), then $x = X +\eps$, and 
if $X$ is cofinite in $\N$, then $x = X^\sharp - \eps$.
Moreover, it will turn out that $\eps = \frac{1}{\omega}$.
\end{example}

\section{Descendance:  the binary tree of numbers}\label{sec:number-tree}

Besides the total order $\leq$, there exist several partial orders on $\No$,
for instance,
inclusion of numbers, as sets, $x \subset y$, or  "older than". 
The most important partial order is "descendance" $x \preceq y$; some authors also use the
terms "simpler than", meaning that the sign expansion of $x$ is an "initial segment"
(\cite{Go}) of the one of $y$. 
Its Hasse diagram is a {\em complete binary tree}.

\begin{definition}[Descendance]\label{def:descendant}
We say that a number $y$ is a {\em descendant} of $x$, and $x$ an {\em ancestor}
of $y$, and we write $x \preceq y$, if
$$
\Beta(x) \leq \Beta(y), \mbox{ and } 
\llbrack  0, \Beta(x) \llbrack  \cap \, x = \llbrack 0,\Beta(x) \llbrack  \cap \, y,
$$
This  means that  "below the birthday of $x$, both numbers coincide". 
\end{definition}

\begin{definition}[Initial inclusion]
 We define {\em initial  inclusion}  of a number $x$ in $y$, and 
 write  $x \sqsubseteq y$, iff $x$ coincides with $y$ up to, and including,
 $\Beta(x)$:
 $$
 x = \llbrack 0, \Beta(x) \rrbrack \cap y .
 $$
This implies that $x \subset y$, and $\Beta(x) \leq \Beta(y)$.
\end{definition}

\begin{theorem}[Elementary properties of descendance]\label{la:immediate}\label{th:immediate}
$ $
\begin{enumerate}
\item
Both relations $\preceq$ and $\sqsubseteq$ are partial orders.

Moreover, $x \sqsubseteq y$ if, and only if, $[ x \subset y$ and $x \preceq y ]$.
\item
If   $y \prec z$, then:  $z$ is an {\em immediate} successor of $y$ for $\prec$
  iff $\Beta(z)=\Beta(y)+1$. 
 
Every number $x$ has exactly two immediate successors ("children") for $\prec$, namely
  its {\em left and right children} $x_\pm$, defined in Def.\
\ref{def:children}.
\item
A number $x$ is a successor number (Def.\ \ref{def:stages}) if, and only if,

it is of the form $x = y_+$ or  $x = y_-$ for some number $y$, i.e., iff

  it has an immediate predecessor ("parent")  for $\preceq$.
  \item
  if $x \sqsubset y$, then $y$ is an immediate successor of $x$ for initial inclusion, if and only if :
 $x \not= y$ and $y = x \cup \{ \Beta(y)\}$.
 In particular, every number has infinitely many immediate successors for initial inclusion.
\end{enumerate}
\end{theorem}

\begin{proof}
1. is immediate from the definitions.

2. 
Assume $y$  is an immediate successor of $x$ for $\prec$.
If $\Beta(y) = \Beta(x)$, then clearly $x=y$, so this is impossible.
If $\Beta(y) = \Beta(x)+1$,
then $y$  is determined up to choice of the element $\Beta(x)$, leading
to the $2$ cases $y=x_+$ and $y=x_-$, which are
 immediate successors.
If $\Beta(y) \geq \Beta(x) + 2$,
then distinguish two cases:
 either, $\Beta(x) \in y$, or $\Beta(x) \notin y$.
 In the first case, 
 $x \preceq x_+ \preceq y$, and in the second,
 $x \preceq x_- \preceq y$, so $y$ is not an {\em immediate} successor.
 Therefore we must have $\Beta(y) = \Beta(x)+1$, and we end up with the only
 possibilities $y = x_+$ or $y = x_-$. 

3. 
If $x = y_\pm$, then $\Beta(x) = \Beta(y)+1$ is a successor ordinal.
Conversely, if $\Beta (x) = \alpha + 1$, then let
$y = (x \setminus \{ \Beta(x) \}) \cup \{ \alpha \}$; then
$x= y_+$ if $\alpha \in x$ and $x=y_-$ if $\alpha \notin x$, so $x$ has a "parent"  for $\preceq$. 

4. 
Immediate from definitions.
\end{proof}

The {\em binary tree of numbers} is the Hasse diagram of the partial order
$\preceq$: it starts  with top vertex $0_\Co = \{ 0_\vn \}$,
linking each vertex $x$ with vertices $x_-$ and $x_+$, and so on:  
for $\mu,\sigma \in \{ \pm\}$,
define $x_{ \mu \sigma} := (x_\mu)_\sigma$, and present the part of the number tree descending from $x$ in the form
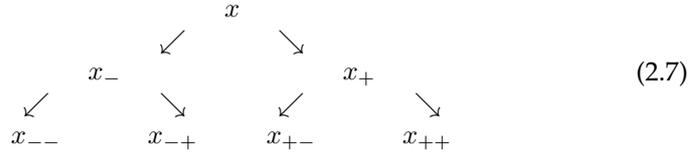
\begin{equation}\label{eqn:signexpansion}
 \begin{matrix}
 & & & x & & & \\
 & & \swarrow & & \searrow & & \\
 &  x_-    & & & &  x_+   & \\
 \swarrow & &  \searrow &  & \swarrow & & \searrow 
 \\
 x_{--}    & & x_{-+} & & x_{+-}  & &  x_{++}
 \end{matrix}
\end{equation}
and so on for sequences of $+$'s and $-$'s. E.g., for $\No_3$, 
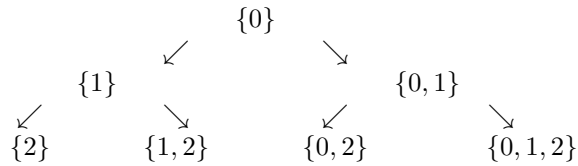
$$
{ } \qquad
\begin{matrix}
 & & & \{ 0\} & & & \\
 & & \swarrow & & \searrow & & \\
 & \{ 1\}  & & & &  \{ 0,1\}  & \\
 \swarrow & &  \searrow &  & \swarrow & & \searrow \phantom{\qquad}
 \\
 \{ 2 \} & & \{1,2\} & & \{0,2\} & & \{0,1,2\} 
 \end{matrix}
 $$
  For $\No_4$ see the figures given in the Introduction,
 and the internet for \href{https://upload.wikimedia.org/wikipedia/commons/4/49/Surreal_number_tree.svg}{figures of infinite stages}. 
 Note, however, that numbers whose birthday is a limit ordinal never have immediate
 predecessors.
 The horizontal stractification is given by birthday;
 the boundary $\partial(\No_\alpha)$ of a stage is a horizontal line.
 The  partial order
 $\sqsubset$ amounts to moving in the tree {\em to the right}:
 in fact, $x \sqsubset y$ holds if, and only if,
 $x\preceq y$ and $x_+ \preceq y$. 
 For $\preceq$,
 it is obvious that $0_\Co$ is a "common ancestor" of all numbers, and it is the
"oldest ancestor" of any two numbers $x$ and $y$.
There is also a youngest one, which is a meet $x\land y$ of $x$ and $y$ in the sense of lattice
theory. This may seem "obvious" when looking at figures of the number tree, but of course
needs a rigorous proof:
recall Definition \ref{def:discriminant}.

 \begin{definition}\label{def:yca}
For  two numbers $x \not= y$, let
 $
 \delta := \delta (x,y) := \min (x\Delta y).
 $
 It follows that $x \cap y \cap \llbrack 0,\delta  \llbrack = x \cap  \llbrack 0,\delta  \llbrack
 = y \cap \llbrack 0,\delta \llbrack$.
 We define a number 
 $$
 \yca(x,y)  :=  (x \cap y \cap \llbrack 0,\delta \llbrack) \cup \{ \delta \}  .
$$
\end{definition}

 \begin{theorem}\label{th:yca}
 Any two numbers $x,y$ have a unique {\em youngest common ancestor} 
 $u$,
 that is, there is a unique number $u$ whose birthday is maximal, subject to the condition
 that $u \preceq y$ and $u \preceq x$.
 It is given by $u= \yca(x,y)$, and it satisfies 
 $$
 \mbox{ if } x \leq y, \mbox{ then } 
 x \leq \yca(x,y) \leq y .
 $$
 \end{theorem}
 
 \begin{proof} 
 Let $u =\yca(x,y)$. Clearly
  $u \preceq y$ and $u \preceq x$.
And there cannot be another such number with greater birthday,
since $\alpha > \delta$ implies that $\alpha$ cannot belong to both  $x$ and $y$. 

\nin
Let $x<y$. Then $\delta \notin x$, whence
$x < \yca(x,y)$. And $\delta \in y$, so $\yca(x,y) < y$. 
\end{proof}
  
 \subsubsection{Links between tree structure and  total order}
 
The last statement of Theorem \ref{th:yca} establishes an important link between
the partiel order $\preceq$ and the total order $\leq$. It has several consequences.
Recall that a subset or subclass $I$ of $\No$ is {\em convex} if
$x,z \in I, x \leq y \leq z$ implies $y \in I$.

\begin{lemma}\label{la:convex}
Every non-empty convex  subclass $I$ of $\No$ contains a unique element with minimal birthday.
\end{lemma} 

\begin{proof}
Existence: since $I$ is not empty, it contains some element with some birthday $\alpha$.
By well-order, the set
$\{ \Beta(x) \mid x \in I , \Beta(x) \leq \alpha \} $ has a minimal element $\mu$.
Uniqueness: 
Let $x,z \in I$ elements having this minimal element $\mu$ as birthday.
Theorem \ref{th:yca}
implies that $y:=\yca(x,z)$ belongs to $I$, since $I$ is convex.
But $\Beta(y) < \Beta(x)$ if $y \not= x$, so it follows that $x=y=z$.
\end{proof}

\begin{definition}\label{def:truncation}
For all numbers $x$ and ordinals $\alpha$, we define the
{\em $\alpha$-truncation of $x$} to be the number 
$$
[x]_\alpha := 
 \Big\{ 
\begin{matrix}
x & \mbox{ if } \Beta(x) \leq \alpha,
\\
(x \cap \,  \llbrack 0, \alpha \llbrack)\cup \{ \alpha \}  & \mbox{ if }  \Beta(x) > \alpha. 
\end{matrix}
$$
Note that, if $\Beta(x)>\alpha$ and $\alpha \in x$, then 
$[x]_\alpha  = x \cap \,  \llbrack 0, \alpha \rrbrack$.
In particular,
$$
[x]_{\tip(x)} = [x]
$$
is the omnific integer part of $x$ (Def.\ \ref{def:integer}).
\end{definition}

\begin{lemma}\label{la:la}
For two numbers $y,z$ we have
$[z]_{\Beta(y)} = y$ if, and only if,
$y \preceq z$.
\end{lemma}

\begin{proof}
Both conditions imply that $\Beta(y)\leq \Beta(z)$, and then
express that $y$ is an initial segment of $z$.
\end{proof}

\begin{theorem}[Monotonicity of truncation]\label{th:truncation}
If 
$x \leq y$, then $[x]_\alpha \leq  [y]_\alpha$.

As a consequence, for all numbers $y$, the class $I$ of all $x$ with
$y \prec x$, is convex.
\end{theorem}

\begin{proof}
Let $x < y$, i.e., $\delta(x,y) \in y$, so the discriminant does {\em not} belong to $x$.
Then the same holds for the truncations, whence
$\delta([x]_\alpha,[y]_\alpha)) \in [y]_\alpha$,
and so $[x]_\alpha  < [y]_\alpha$.

By Lemma \ref{la:la}, $x \in I$ iff $[x]_{\Beta(y)} = y$.
Let $x , z \in I$ and $x< u < z$.
Then $y=[x]_{\Beta(y)} \leq [u]_{\Beta(y)} \leq [z]_{\Beta(y)} = y$,
whence $[u]_{\Beta(y)} = y$, so $u \in I$.
\end{proof}

\begin{example}\label{ex:truncation}
The convex set $I$ may be an interval:

$[x]_1 = \{ 1 \} $ iff $0 \notin x$, iff $x<0$ ;

$[x]_1 = \{ 0 \}$ if $x = \{ 0 \} = 0_\Co$

$[x]_1 = \{ 0 , 1 \} = 1_\Co$ iff $0 \in x$ and $\Beta(x)>0$, iff $x > 0$.

\nin
One can show that $I$ is an interval
if $y$ is a successor number,
but in general not so if $y$ is a limit number.
\end{example}

\section{The number tree is connected and complete}

In some regards, the descendance relation $\preceq$ is more important than the
total order $\leq$:
the "topology of $\No$''  is not the "usual interval topology" for the total order
(open intervals in $\No$ are proper classes, not sets), but it is rather given
by {\em closed sets}, where a set of numbers $U$ is "closed" if it contains all
its limits under forming descending sequences from elements of $U$.
Thus the "topological" properties of the number tree are crucial for the theory:
the tree is {\em connected} and {\em complete}. 

\begin{definition}\label{def:convergent}\label{def:limit}
A {\em sequence of numbers} is given by an ordinal-indexed
 family $(x_\alpha)_{\alpha < \lambda}$ of
numbers such that
$\alpha < \beta < \lambda$ implies $\Beta(x_\alpha) \leq \Beta(x_\beta)$,
 where the ordinal $\lambda$ is the {\em length} of the sequence.
We say that such a sequence {\em converges} if there exists a number $x$
(its {\em limit}) such that

(a)  for all $\alpha < \lambda$, we have $x_\alpha \preceq x$,

(b)  $x$ is minimal for (a), i.e., if $x'$ satisfies (a), then $x \preceq x'$.
\end{definition}

\nin
In view of Thm.\ \ref{th:immediate}, (b) is equivalent to saying that $\Beta(x)$ is
minimal for (a).

 \begin{theorem}\label{th:completeness}
 A sequence of numbers $(x_\alpha)_{\alpha <\lambda}$ converges if, and only if,
 it is a {\em chain with respect to the partial order $\preceq$}, that is:
 $$
 \forall \alpha < \beta < \lambda: \quad x_\alpha \preceq x_\beta. 
 $$
 In this case, the number $x$ is unique, called the {\em limit of the sequence}, 
 denoted by $x = \lim_{\alpha \to \lambda} x_\alpha$, and it is given by:
 $$
 x=  \bigcup_{\alpha < \lambda} X_\alpha \cup \{ \sigma \}, 
 \quad \mbox{where }
 X_\alpha = x_\alpha \setminus \{ \Beta(x_\alpha)\}, \mbox{ and }
 \sigma = \sup \{ \Beta(x_\alpha) \mid \alpha < \lambda \} .
 $$
 \end{theorem}
 
 \begin{proof}
 The condition for convergence is necessary:
 if both $x_\alpha$ and $x_\beta$ are initial segments of $x$, and
 $x_\alpha$ is older than $x_\beta$, then $x_\alpha$ is also an initial segment of
 $x_\beta$.
 
 Let us show that it is sufficient:
 if the sequence is a chain, then the sets $X_\alpha $ satisfy
 $\alpha < \beta \Rightarrow X_\alpha \subset X_\beta$, so
 $X:=\bigcup_{\alpha < \lambda} X_\alpha $ is an ascending union.
 Like every set of ordinals, it has a supremum (minimum of ordinals bigger than all its
 elements), which by definition is $\sigma$. 
 Thus $x = X \cup \{ \sigma \}$ is a set of ordinals having a maximum, namely $\sigma$,
 hence is a number.
 By its definition, this number satisfies $x_\alpha \preceq x$, for all
 $\alpha < \lambda$, and its birthday $\sigma$ is indeed minimal for this condition, so
 $x$ is a limit.
 
If $x'$ is another limit, then $x'$ must contain the set $X$, and from minimality it follows
that $\Beta(x') = \sigma$, whence $x'=x$, so the limit is unique. 
 \end{proof}

 \begin{definition}\label{def:branch} 
 A {\em maximal chain}, or {\em path} in the number tree,
   is a chain $(x_\gamma)_{\gamma < \lambda}$ that admits no proper 
   refinement, i.e.: 
 
 if  $z$ satisfies $x_\gamma \preceq  z \preceq x_\beta$,
 then there exists
 $\delta \in \rrbrack \gamma,\beta\llbrack$ such that $z = x_\delta$
 
 (in particular, $x_{\gamma + 1}$ is a child of $x_\gamma$, so
$x_{\gamma+1} = (x_\gamma)_+$ or $(x_\gamma)_-$).
 \end{definition}

 \begin{example} Let $x = \{ 0,2,4 \}$. A path joining $0_\Co$ to $x$ is: 

$x_0 = \{ 0 \}$

$x_1 = \{ 0 , 1 \} = (x_0)_+$

$x_2 = \{ 0,2 \} = (x_1)_-$

$x_3 = \{ 0,2,3 \} = (x_2)_+$

$x_4 = \{ 0,2,4 \} = (x_3)_-$
\end{example}

\begin{example} Let $x = \{ 1 , \omega \}$. A path joining $0_\Co$ to $x$ is:

$x_0 = \{ 0 \}$

$x_1 = \{ 1 \} = (x_0)_-$

$x_2 = \{ 1 ,2 \} = (x_1)_+$

$x_3 = \{ 1,3 \} = (x_2)_-$

$x_4 = (x_3)_- , \ldots , x_{n+1} = (x_n)_-, \ldots, x_\omega = \{ 1, \omega \}$.
\end{example}

\nin
Visualizing these paths in the number tree exhibits the algorithm how to construct a path
joining $0_\Co$ to a number $x$:
 at level $\gamma$, 

-- if $\gamma \in x$, then turn right,

-- if $\gamma \notin x$, then turn left.

\begin{theorem}\label{th:connected}
The number tree is {\em connected}:
for every number $x$, there exists a unique  path
$(x_\gamma)_{\gamma < \lambda}$ such that
$x_0 = 0_\Co$ and $ x = \lim_{\alpha \to \lambda} x_\alpha$.
Explicitly, this path is given by the family of truncations of $x$ (Def.\ \ref{def:truncation}),
with $\lambda = \Beta(x)$ if $x$ is a limit number, and
$\lambda + 1 = \Beta(x)$ if $x$ is a successor number,
$$
\forall \gamma < \lambda : \qquad
x_\gamma = [x]_\gamma  = 
 (x \cap \llbrack 0, \gamma \llbrack ) \cup \{ \gamma \}  .
$$
\end{theorem}

\begin{proof}
The family $(x_\gamma)_{\gamma < \Beta(x)}$ defined in the theorem is a chain
starting at $0_\Co$. 
Indeed, if $\beta < \gamma$, then $x_\beta \prec x_\gamma$, since
$\Beta(x_\beta)=\beta$, so
$$
x_\beta \cap \llbrack 0, \Beta(x_\beta)\llbrack = 
(x \cap  \llbrack 0, \beta \llbrack ) =
(x_\gamma \cap  \llbrack 0, \beta \llbrack ) .
$$
This chain is {\em maximal}.
Indeed, it is seen by induction that {\em  any chain $(z_\gamma)_{\gamma \leq \alpha}$ for
$\prec$ is maximal if, and only if, $\Beta(z_\gamma) =\Beta(z_0)\oplus  \gamma$ for all $\gamma \leq \alpha$.}
(If it were not maximal, then one could refine it; but this is impossible by the birthday-condition.)
The condition $\Beta(x_\beta)=\beta$ holds for  the chain in question, which hence is maximal, that is,
it is a path.
Finally, since $\Beta(x) \in x$, it follows
that
$$
x_{\Beta(x)} =  (x \cap \llbrack 0, \Beta(x) \llbrack ) \cup \{ \Beta(x) \}  = x,
$$
so the path joins $x_0 = \{ 0 \} = 0_\Co$ to $x = x_{\Beta(x)}$.
\end{proof}

\begin{remark}
The preceding result allows to define properties or Functions on $\No$, by transfinite induction,
via

(1) successor rules, using  $x_+$ and $x_-$, and a

(2) limit rule, for $\lim_{\alpha \to \lambda} x_\alpha$.

\nin
Here is an example: we define another quantum of $x$.
\end{remark}

\begin{definition}\label{def:width}
The {\em width}, or {\em pseudo-cardinal} $\w(x)$,  of a number $x$ is the ordinal inductively defined by:

(1) $\w(x_+) = \w(x) \oplus 1$, $\w(x_-) = \w(x)$,

(2) $\w(0_\Co) = 0$ and
$\w(\lim_{\alpha \to \lambda} x_\lambda) = 
\sup_{\alpha < \lambda} \w(x_\lambda)$.
\end{definition}

For a short number $x$, the width is its cardinal minus one. 
For a long real number $x$, we always have $\w(x) = \omega$, which agrees with its cardinal.
Also $\w (\alpha_\Co ) = \alpha$ for every ordinal, so  $\w(x)$ measures the "size", or
"width", of a number, but it is not the same as its cardinal.
This quantum appears in the description of Conway's omega-map.


 \subsubsection{Left and right predecessors of a number} 

The path joining $0$ to $x$ splits naturally into two parts,
with edges "turning right" (giving rise to a $+$ in the sign expansion) and edges
"turning left" (giving rise to a sign $-$), see the examples above.
The numbers where the path turns right are smaller than $x$, they form the 
"left set", and those where the path turns left are greater than $x$, they form the
"right set":

\begin{definition}\label{def:CanCut}
To each number $x$, we associate two sets of numbers, 
called its {\em canonical left, resp. right, sets}, defined
in terms of the truncation (Def.\ \ref{def:truncation})

$L_x := 
\{ 
[x]_\gamma 
 \mid {\gamma \in x} , \gamma < \Beta(x) \}$,

$R_x := \{ 
[x]_\gamma 
\mid {\gamma \notin x,
\gamma < \Beta(x)} \}$
\end{definition}

 \begin{lemma}\label{la:CanCut}
The canonical left set of $x$ is a maximal chain with respect to initial inclusion
$\sqsubset$. 
Moreover,

$L_x = \{ z \mid z \sqsubset x \}
= \{ z \mid z \preceq x, z < x $\},

$R_x
= \{ z \mid z \preceq x, z >  x \}$.
  \end{lemma}
 
 \begin{proof}
 For $\gamma \in x$, let
  $x_\gamma = x \cap \llbrack 0,\gamma \rrbrack $.
Clearly,
$\beta < \gamma$ implies
$x_\beta \sqsubset x_\gamma$, so we have chain for initial inclusion,
and since the left set is indexed by {\em all} elements of $x$,
this chain is maximal.  
Moreover, the definition of the total order $<$ shows that
$a \sqsubset b$ iff ($a \prec b$ and $a < b$). 
Similarly for the right set.  
  \end{proof}

See Theorem \ref{th:CanCut} for a sort of converse of this construction.

\subsubsection{Topology of $\No$: closed sets}\label{sec:topology}

The notion of "limit in $\No$" (Def.\ \ref{def:limit}) calls for defining a "topology" on
$\No$. Since  $\No$ is a proper class, 
the usual duality between "open" and "closed" sets breaks down (the "complement of a set"
in $\No$ is not a set), and it will be more convenient to base the idea of "topology" on
{\em closed} sets, since the "open" sets would rather correspond to proper classes
 - e.g., intervals of the total order on $\No$
are proper classes, whereas we have many "closed" small  sets:

\begin{definition}
Let $U$ be a set (or class) of numbers.
\begin{enumerate}
\item
A {\em limit point of $U$} is a number $x$ such that
$x = \lim_{\gamma \to \alpha} x_\gamma$ for a chain $(x_\gamma)_{\gamma < \lambda}$
 such that
all $x_\gamma$ belong to $U$. 
\item
The {\em closure} $c(U)$ or $\overline U$ of $U$ is the union of $U$ with the set
(or class)  of its limit points.
\item
The set (or class) $U$ is {\em closed} if $U = \overline U$.
\end{enumerate} \end{definition}

This definition matches the one given in \cite{BH1}, Section 5.2, and
as shown there it has properties that one may expect from closure in a topological
setting. 


\section{Extending ordinal arithmetic to Conway numbers}\label{sec:Cantor-No}\label{sec:CoCa}

Recall that we denote by $\oplus, \otimes,\ootimes$ the three operations of usual (Cantor)
ordinal arithmetic on $\On$.
These extend to $\No$ in a fairly straightforward way, and it will be convenient to use
the same symbols. 
Concerning $\oplus$ and $\otimes$, such results can be found in \cite{BH1},
Section 3.2. In fact, the first operation, $\oplus$, called 
 {\em concatenation}, is heavily used in \cite{Go} without a suitable formalisation,
 and in Game Theory, the notation
  $x:y$ is also used, see \cite{S}, p. 89, \cite{Co01}, p.31,  but this notation appears to be dangerous,
 and hides the relationship with usual ordinal addition.

\begin{definition} \label{def:involution2}
The {\em pseudo-inverse} $x^*$ of a number $x$ is defined by:
$\alpha \in x^*$ iff

-- either $\alpha = \Beta(x)$,

-- or $0 < \alpha < \Beta(x)$ and $\alpha \notin x$,

-- or [ $\alpha = 0$, if $0 \in x$ ].

\nin
In a formula:
$x^* = (\{ 0 , \Beta(x) \} \cap x) \cup (x^\sharp \cap \llbrack 1, \Beta(x) \llbrack)  .
$

\nin
We refer to $\sharp$ and $*$ as the {\em first and second main involution of $\No$}.
\end{definition}

In the number tree, $*$ is visualized as the union of the "symmetry of the right part at the
vertical axis $x=1$",  with the "symmetry of the left part at the
vertical axis $x=-1$".

\begin{lemma}\label{ref:la+-}
The pseudo-inverse satisfies:
\begin{enumerate}
\item
$\Beta(x^*) = \Beta(x)$
\item
$(x^*)^* =x$
\item
$(x^*)^\sharp = (x^\sharp)^*$
\item
$0_\Co^*  = 0_\Co$, $1_\Co^*= 1_\Co$, $(1_\Co^\sharp)^* = 1_\Co^\sharp$, and 
\item
$
\forall x \not=0_\Co : \,
(x_+)^* = (x^*)_-, \, (x_-)^* = (x^*)_+ $.
\end{enumerate}
\end{lemma}

\begin{proof}
Immediate, except maybe for the third equality:
we compute
$$
(x^\sharp)^* = \{ \alpha \mid \, 
\alpha \in x \cap \llbrack 1, \Beta(x) \rrbrack   \mbox{ or }
[ \alpha = 0 \mbox{ if } 0 \notin x ] \}
= (x^*)^\sharp .
$$
In other terms, $(x^\sharp)^* = x \Delta \{ 0 \}$, where $a \Delta b$ 
is the symmetric difference. 
\end{proof}

\begin{theorem}
For each pair of numbers $(x,y)$, we can define three new numbers
$x \oplus y$, $x \otimes y$ and $x\ootimes y$, uniquely defined via transfinite induction by the following rules: 
initial conditions,

$x \oplus 0 = x$

$x \otimes 0 = 0$

$x\ootimes 0 = 1$.

\nin
Succession rule in the second argument:
\begin{enumerate}
\item
$x \oplus y_+ = (x \oplus y)_+$

$x \oplus y_- = (x \oplus y)_-$
\item
$x \otimes y_+ = x \otimes y \oplus x$

$x \otimes y_- = x \otimes y \oplus x^\sharp$ (where $x^\sharp$ is given by Def.\
\ref{def:stages})

\item
$x \ootimes y_+ = x \ootimes y \otimes x$

$x \ootimes y_- = x \ootimes y \otimes x^*$ (where $x^*$ is given by Def.\
\ref{def:involution2}),
\end{enumerate}
Continuity with respect to limits in the second argument
(see Def.\  \ref{def:convergent})
\begin{enumerate}
\item
$x \oplus (\lim_{\alpha \to \lambda} y_\alpha) = \lim_{\alpha \to \lambda} (x \oplus y_\alpha)$
\item
$x \otimes (\lim_{\alpha \to \lambda} y_\alpha) = \lim_{\alpha \to \lambda} (x \otimes y_\alpha)$
\item
$x \ootimes (\lim_{\alpha \to \lambda} y_\alpha) = \lim_{\alpha \to \lambda} (x \ootimes y_\alpha)$
\end{enumerate}
\nin
The operations $\oplus$ and $\otimes$ are associative, but non commutative, and we have the
following
"left distributive laws" (where, by convention, $\ootimes$ binds stronger than $\otimes$):

$x \otimes (y \oplus z)  = x \otimes y \oplus x \otimes z$

$x \ootimes (y \oplus z) = x \ootimes y \otimes x \ootimes z$.

$x \ootimes (a \otimes b) = (x \ootimes a) \ootimes b$

\nin
The Birthday Map $\Beta: \No \to \On$, and its Section $\alpha \mapsto \alpha_\Co$, are Morphisms:

$\Beta( x \oplus y) = \Beta(x) \oplus \Beta(y)$,

$\Beta( x \otimes y) = \Beta(x) \otimes \Beta(y)$,

$\Beta( x \ootimes y) = \Beta(x) \ootimes \Beta(y)$.

\nin
Compatibility with the main involutions is given by

$(x \oplus y)^\sharp = x^\sharp \oplus y^\sharp$

$(x \otimes y)^\sharp = x^\sharp \otimes y = x \otimes y^\sharp$

$(x \ootimes y)^* = x \ootimes y^\sharp$

$(1 \oplus x)^* = 1 \oplus x^\sharp$, $\quad ( (-1) \oplus x )^*=(-1) \oplus x^\sharp$

\end{theorem}

\begin{proof}
If $\bullet$ is one of the three operations, then for fixed $x$, the value of $x \bullet y$ 
is defined by transfinite induction on the birthday of $y$, 
using the first set of rules for induction on successor numbers $y$ (recall that every such number
is of the form $z_+$ or $z_-$, see Theorem \ref{th:immediate}), and
the continuity condition for limit numbers $y$. 
We have to define $\oplus$ first, since it is used to define $\otimes$, just like
$\otimes$ is needed for the definition of $\ootimes$.

In the same way, all other statements may be proved by transfinite induction, copying the corresponding
arguments proving the corresponding rules for ordinal arithmetic on $\On$. 
\end{proof}

Remark. See \cite{BH1}, Rk.3.2 for
 a definition of $\oplus$ and $\otimes$  a la Conway by cuts.

\begin{lemma}\label{la:concatenation}
 The sign sequence of $x \oplus y$ is obtained by juxtaposition of those
 from $x$ (first, written on the left) and $y$ (second, written on the right).
 More formally, 
  $$
 x \oplus  y = \{ \alpha \in \On \mid \, \alpha \in x\setminus \{\Beta(x) \}, \mbox{ or } \exists \beta \in y :
 \alpha = \Beta(x) \oplus  \beta \} .
 $$
 In other terms, $\alpha \in x \oplus y$ iff
 $\alpha = \alpha_0 \oplus \alpha_1$, where
 
 either [$\alpha_0 \in x, \alpha_0 < \Beta(x)$ and $\alpha_1=0$ ] or
 [ $\alpha_0 = \Beta(x)$ and $ \alpha_1 \in y$ ].
 \end{lemma}
 
\begin{proof}
By straightforward transfinite induction.
\end{proof}

 In the number tree, the graph of $x \oplus y$
  corresponds to "grafting the graph of $y$ onto the one
 of $x$". The two immediate successors of $x$ are obtained as in the theorem, and 
 by iteration 
    \begin{align*}
   x_{++} & = x \oplus 1_\Co \oplus 1_\Co = x \oplus \{ 0,1,2\} = x \oplus 2_\Co ,
   \\
 x_{+-} & = x \oplus 1_\Co \oplus (-1)_\Co = 
 x \oplus \{ 0,2\} = x \oplus (\frac{1}{2})_\Co , \\
 x_{-+} & =x \oplus (-1)_\Co \oplus 1_\Co = 
  x \oplus  \{ 1,2 \}  = x \oplus (-\frac{1}{2})_\Co ,\\
 x_{ - -} & = x \oplus \{ 2 \}  =  x \oplus 2_\Co^\sharp.
 \end{align*}
 This implies, e.g., that
 $0_{-+}=\{ 1,2  \} = (-1)_\Co \oplus 1_\Co$. 
 In general: 
 
 \begin{definition}
 For each number $x$ and each sign $\sigma \in \{ +, -  \}$, we let
$$
\sigma  x := \Bigl\{ \begin{matrix}
x& \mbox{ if } & \sigma  =+ \\
x^\sharp & \mbox{ if } & \sigma = - . 
\end{matrix} 
$$
\end{definition}

\begin{theorem}\label{th:concatenation2}
 Infinite ordered concatenations of numbers are well-defined: 
 given numbers $x_\alpha$ for  all $\alpha < \gamma$, for some ordinal $\gamma$,  the (generically)
 infinite concatenation 
$z= \bigoplus_{\alpha < \gamma} x_\alpha$ exists in $\No$.
The order of "summation" is the natural order of the indices. 
Its birthday is the ordinal
$\Beta(z) = \bigoplus_{\alpha < \gamma} \Beta(x_\alpha)$
(see Def.\ \ref{def:infiniteoplus}).
Every number $x$ can (uniquely) be represented in the form  
 \begin{equation}\label{eqn:x-rep}
x = \bigoplus_{\alpha < \Beta(x) } (s_x(\alpha) 1_\Co) .
\end{equation}
\end{theorem}
 
\begin{proof}
By induction.
If $\gamma$ is a successor ordinal, the sum is defined by adding a new term,
and if $\gamma$ is a limit ordinal, it is the limit of a chain. 
\end{proof}

\begin{theorem}
The Map "left translation by $a$",
$$
E_a : \No \to \desc(a) \subset \No, \quad x \mapsto a \oplus x
$$
is a Bijection onto the class of descendants of $a$. 
It preserves order and tree-structure.  In particular, for $a=1$, resp.\ $a=-1$,
we get Bijections
$$
\begin{matrix}
E_1 : &  \No \to \desc(1)=\{ x \in \No \mid x>0\}, & \quad
x \mapsto 1_\Co \oplus x,
\\
E_{-1} : &  \No \to \desc(1^\sharp)=\{ x \in \No \mid x<0\}, & \quad
x \mapsto (-1)_\Co \oplus x.
\end{matrix}
$$
They intertwine the main involutions:
$E_1 (x^\sharp) = (E_1x)^*$, and
$E_{-1} (x^\sharp)= (E_{-1}x)^*$
\end{theorem}

\begin{proof}
Clearly $a \preceq a \oplus x$.
Conversely, if $a \preceq y$, then 
$x$ is the "tail of $a$ in $y$" (\cite{Go}), p.3), given by
$$
x= \{ \alpha \ominus \Beta(a) \mid \alpha \in y, \alpha > \Beta(a) \}
$$
where $\alpha \ominus \beta = \gamma$ iff
$\alpha = \gamma \oplus \beta$ for ordinals, 
so $E_a$ is a Bijection.  
The intertwining property is direct from the definition of $\sharp$ and $*$.
\end{proof}


Computing the sign-expansions of
 $\otimes$, the following lemma states that the sign expansion of
 $x \otimes y$ is given 
by repeating the one of $x$ for each $+$ in the one of $y$,
repectively,  of $x^\sharp$ for each $-$ in the sign-expansion of $y$:

\begin{lemma}
For each pair of numbers $(x,y)$,
$$
x \otimes y = 
\bigoplus_{\alpha < \Beta(y)} (s_y(\alpha) x ).
$$
The sign-expansion of $x \otimes y$ is given by
\begin{equation}
s_{x \otimes y} ( \Beta(x) \otimes \alpha \oplus \beta) =
s_y(\alpha) \cdot s_x(\beta) .
\end{equation}
\end{lemma}

\begin{proof} The first claim is proved by induction, and the second is a direct consequence.
\end{proof}

\begin{example}

$x \otimes \OM$ is $\omega$ times juxtaposition of the sequence of $x$ with itself,

$\OM \otimes x$ is a juxtaposition of sequences of $\omega$ plusses or minusses.
\end{example}

Clearly the element $(-1)_\Co$ commutes with all numbers under $\otimes$, so that
$$
(x \otimes y)^\sharp = x \otimes y \otimes (-1)_\Co = x \otimes (-1)_\Co \otimes y =
x^\sharp \otimes y = x \otimes y^\sharp.
$$
(In fact, the only numbers that commute with all others under $\otimes$ are 
$0_\Co,1_\Co,(-1)_\Co$.)

\begin{lemma}
For every number $a$ and sign $\sigma \in \{ +,- \}$,  let
$$
a^\sigma := \Bigl\{ \begin{matrix} a & \mbox{ if } \sigma = + \\
a^* & \mbox{ if } \sigma = -  . 
\end{matrix}
$$
Then,  for every pair of numbers $(x,y)$,
$$
x \ootimes y := \bigotimes_{\alpha < \Beta(y)} x^{s_y(\alpha)}.
$$
In particular, it follows that $a \ootimes \alpha_\Co = \otimes_{\gamma \in \alpha} a$ is a concatenation
product of $\alpha$ factors $a$,  and $a^* = a \ootimes 1_\Co^\sharp$, and

$a \otimes a = a \ootimes 2_\Co = a \ootimes \{ 0,1,2 \}$, 

$a \otimes a^* = a \ootimes \{ 0, 2 \}$,

$a^* \otimes a = a \ootimes \{ 1,2 \}$

$a^* \otimes a^* = a \ootimes \{ 2 \}$,

$ a \otimes a \otimes a = a \ootimes 3_\Co = a \ootimes \{ 0,1,2,3 \}$, etc.
\end{lemma}

\begin{proof}
By induction.
\end{proof}

\begin{remark}
One could continue and define ordinal tetration, and other higher hyperoperations $H_n(a,x)$, on 
Conway numbers, following the pattern started above:
the only problem is to figure out the "correct definition" of
$H_n(a,-1)$, i.e., to define higher order "main involutions of $\No$";
 then  induction together with distributivity  entirely determines
$H_n(a,x)$ for all $x$.
\end{remark}

\begin{remark}
What other properties of ordinal calculus carry over to $\No$?
For instance, what about (left) division with remainder, and some analog of the
Cantor normal form?
How would  this be related to the "Conway normal form"?
This is a topic for subsequent work. See also Section \ref{sec:Grothendieck}. 
\end{remark}

\subsubsection{Purely infinite numbers}

\begin{definition}
A number $y$ is called {\em purely infinite} \cite{Go} (or {\em indivisible integer}, \cite{Co01}), or
{\em star}, \cite{Sim}), if there is a number $x$ such that  $y = \OM \otimes x$.
In other terms, the sign sequence is given by blocs of $\omega$ pluses or minuses, according to
the pattern of $x$.
The class of all purely infinite numbers is denoted by
$(\OM)$ or
${\mathbb J}$.
\end{definition}

\begin{example}
$\OM \otimes 1 = \OM$,

$\OM \otimes 2 = \OM \oplus \OM = \{ 0,1,2,\ldots, \omega \oplus \omega \}$

$\OM \otimes 2^\sharp = \OM^\sharp \oplus \OM^\sharp = \{ \omega \oplus \omega\}$

$\OM \otimes \{ 1,2 \} = \{ \omega,\omega+1,\ldots,\omega + \omega \}$. In general, 

\begin{equation}\label{eqn:string}
\OM \otimes x =
\OM  \otimes (\bigoplus_{\alpha < \Beta(x)} s_x(\alpha) 1) =
\bigoplus_{\alpha < \Beta(x)} s_x(\alpha) \omega 
\end{equation}
can be seen as the concatenation of strings which take the sign expansion of a number, and repeat each sign
$\omega$ times (cf.\  \cite{Sim}, 7.3, who adds the remark: 
 "this is another fractal view of $\No$.")  
The Map
 \begin{equation}\label{eqn:J-map}
 \No \to {\mathbb J}, \quad
 x \mapsto \omega \otimes x
 \end{equation}
  is a Bijection.
   From distributivity it follows that this map is an isomorphism of binary trees.
  (This an example of a "surreal substructure", cf.\ \cite{BH1}). 
   \end{example}



\subsubsection{Omnific integers and purely infinite numbers} \label{sec:OZ}

\begin{lemma}\label{la:omnific} 
Every omnific integer has a unique representation
$$
x = x_\bJ \oplus y,
$$
where $x_\bJ = \omega \otimes z$ is purely infinite, and $y \in \Z$ is a usual integer.
Conversely, every such sum defines an omnific integer. 
\end{lemma}

\begin{proof}
Decompose $\tip(x) = \lambda(x) \oplus n(x)$ into a limit ordinal plus a natural number 
and let $x_\bJ := x \cap \llbrack 0, \lambda(x) \rrbrack \cup \{ \lambda(x) \}$.
Two cases arise:

(a) if $\lambda(x) \in x$, then $y = \{ 0,1,\ldots,n(x) \}$ is such that
$x = x_\bJ \oplus y$,

(b) if if $\lambda(x) \notin x$, then $y = \{ n(x) \}$ is such that
$x = x_\bJ \oplus y$.

\nin
In both cases, we get a decomposition as claimed.
\end{proof}

\begin{theorem}\label{th:decomp1}
Every number $x$ has a unique decomposition
$$
x = z \oplus u,
$$
where $z = [x]\in \Oz$ and $0 \leq u < 1$.
In other terms, $\No / \Oz \cong  [0,1[$.
\end{theorem}

\begin{proof}
Existence:
let $z:= [x]$; clearly, $z \preceq x$ so there exists a number $u$ with
$x = z \oplus u$.
From the definition of the tipping element $\tip(x)$, it follows that $\tip(x) \in x$, so
$0 \in u$, but $1 \notin u$ since the sign expansion of $x$ changes sign at $\tip(x)$.
This means that $0 \leq u < 1$ (Theorem \ref{th:order1}).

Uniqueness:
all arguments can essentially be reversed, showing that $u$ must be defined by the
first sign change that occurs in the sign-sequence of $x$, leading to
$\Beta(z) = \tip(x)$.
\end{proof}


\section{Natural arithmetic on $\No$ (Conway arithmetics)}

Conway addition $+$ and multiplication $\cdot$ on $\No$ behaves with respect
to $\oplus$ and $\otimes$ just like natural (Hessenberg) operations on $\On$
behave with respect to Cantor ordinal arithmetic.
The fact that they are not continuous with respect to limits makes them
harder to define (and appear "less natural" from a point of view of pure set theory!), 
but one gains commutativity and the full field structure of $\No$.
We first state the fundamental results, due to Conway, and then turn to discussing
strategies of proof.

\begin{theorem}[Sum of numbers]\label{th:Group}
For each pair of numbers $(x,y)$, there is a unique number $z=x+y$ such that
$\Beta(z)$ is minimal, subject to the condition:

for all numbers $x',x'',y',y''$,
\begin{align*}
\Beta(x') < \Beta(x), \, \Beta(x'') < \Beta(x), \, 
x' < x < x'' & \quad  \Rightarrow \quad x' + y < x + y < x'' + y,
\\
\Beta(y') < \Beta(y), \, \Beta(y'') < \Beta(y), \, 
y' < y < y'' & \quad \Rightarrow \quad x + y' < x + y < x + y'' .
\end{align*}
\begin{itemize}
\item
Addition is associative,
  commutative, has neutral $0_\Co$ and inverses
$-x = x^\sharp$.
\item
Monotonicity holds throughout: $(\No,+,\leq)$ is an Ordered Group.
\item
The above definition of $x+y$ holds also if we quantify over
 
 all numbers
 $x',x'',y',y''$, such that
 
 $x' \prec x, x'' \prec x, x'<x<x''$, resp.
 
 $y' \prec y, y'' \prec y, y'<y< y''$. 
\end{itemize}
\end{theorem}

\begin{theorem}[Product of numbers]\label{th:number-product}
For each pair of numbers $(x,y)$, there is a unique number $xy$ such that
$b(xy)$ is minimal, 
subject to the condition:

for all numbers $x',x'',y',y''$,

$\Beta(x')< \Beta(x), \, \Beta(y') < \Beta(y), \, 
x' < x , \,   y' < y$  $\quad   \Rightarrow  \quad x'y + xy' - x' y' < xy$,

$\Beta(x'') < \Beta(x), \, \Beta(y'') < \Beta(x), \,
x < x'', \,  y < y''$  $\quad   \Rightarrow  \quad x''y + xy'' - x'' y'' < xy$,

$\Beta(x') < \Beta(x), \, \Beta(y'') < \Beta(y), \, 
x' < x , \,  y < y''$  $\quad   \Rightarrow  \quad xy < xy'' + x'y - x' y'' $,

$\Beta(x'')<\Beta(x), \, \Beta(y')< \Beta(y), \, 
x < x'' , \,   y' < y$  $\quad   \Rightarrow  \quad  xy < x'' y + xy' - x'' y'$.

\nin
The element $1_{\mathrm{Co}}:=\{ 0,1 \}$ is neutral. 
Moreover, 
 $(\No,+,\cdot,\leq)$ is a Commutative Ordered Ring.
On the class of Conway ordinals, the ring operations correspond to the  
 "natural (Hessenberg) operations".

In the definition of $xy$, we could also quantify over numbers 
$x',x'',y',y''$, where the above condition $\Beta(x') <\Beta(x)$ is replaced by
$x' \prec x$, etc. 
\end{theorem}

\begin{theorem}\label{th:Field} 
Every non-zero number $x$ has a multiplicative inverse in $\No$, and thus $\No$ is an 
Ordered Field.
\end{theorem}

\begin{theorem}\label{th:Reals}
The Conway reals $\R_\Co$ form a subfield of $\No$, and the bijection
between $\R $ and $\R_\Co$ from Theorem \ref{th:Berlekamp} is a field isomorphism.
\end{theorem}

\nin
{\em Ideas of proof.}
First of all, we show that every number $x$ can be represented (non-uniquely)
by a {\em Conway cut}: $x = [ \langle L_x,R_x \rangle ]$ (see next subsection).
Such a cut may be chosen {\em timely}, that is, elements $x',\ldots$ of
$L_x$ and $R_x$ have strictly lower birthday than $x$ itself (Definition \ref{def:timely}).
Then, assuming that $x+x'$, $x+x''$, etc., are already defined, Conway defines a
new cut by the formula
\begin{align*}
 \langle  \{ x + y' , x' + y \mid x'\in L_x,y' \in L_y \} , \{ x + y'' , x''+y \mid x'' \in R_x, y'' \in R_y \} \rangle .
 \end{align*}
He shows that this is indeed a cut, and it defines the number $z = x+y$,
by transfinite induction.\footnote{ 
Cf.\  \cite{Co01},  p. 5:
"such multiple inductions can be justified in the usual way in terms of repeated single
inductions".
However, Gonshor prefers to organize such multiple inductions differently, 
 \cite{Go}, p.\ 13: "We define addition by induction on the natural sum of the lengths on
 the addends." (Idem \cite{A}, p. 133, and Note p.\ 138.) 
 Possibly, this has some technical advantages, but I see no conceptual reason to do so.}    
 Moreover, the number $x+y$ thus defined {\em does not depend on the choice 
 of the cuts defining $x$ and $y$.}
 Our version of Theorem \ref{th:Group} simply states this formula in another way
 (if we quantify over all numbers with lower birthday, this means that our formulation
 corresponds to choosing the maximal possible timely cut, the Cuesta-Dutari (CD) cut, and if we quantify
 over numbers preceding for $\prec$, we choose the canonical cut, Def.\ \ref{def:canonicalcut}).
 Similarly for the product: our version of Theorem \ref{th:number-product} is equivalent
 to Conway's formula, \cite{Co01}, p.4, in terms of cuts,
 \begin{align*}
\langle  \{ x' y + xy' - x' y', \, x''y + xy'' - x'' y'' \vert  x' \in L_x, y' \in L_y, x'' \in R_x, y'' \in R_y  \} \mid 
\\
 { } \quad \, \, \, \, \,   \{ x' y + xy'' - x' y'', \, x''y + xy' - x'' y' \vert x' \in L_x,  y' \in L_y, x'' \in R_x, y'' \in R_y  \} \rangle,
\end{align*}
where we use quantifiers corresponding to the CD-cuts defining $x$ and $y$, resp.\ the
canonical cut if we use $\prec$.
Properties like associativity, distributivity, or those of the additive inverse,
 have to be checked, again by transfinite induction. 
Writing up all of these arguments in full detail takes some place  
(see \cite{Sim} for a full and  relatively compact presentation; see also the
following chapter \ref{chap:Games}). 

By induction, one sees that the sum of 
Conway ordinals is the same as their natural sum, which is also defined by induction
(see Section \ref{ssec:natural}, Equation (\ref{eqn:lim'})):
when $\beta$ is a successor ordinal, we get the same recursion as for ordinal addition
(hyperoperation $H_1$),
$$
\alpha_\Co + (\beta+1)_\Co = \langle \{ \alpha   \} \vert \eset \rangle
+ \langle  \{ \beta + 1  \} \vert \eset \rangle =
\langle \{ \alpha + \beta + 1 \} \vert \eset \rangle = (\alpha + (\beta + 1))_\Co,
$$
and if $\beta$ is a limit ordinal, then since commutativity is imposed by the Conway
definition, we get the modified recursion, Equation (\ref{eqn:lim'}),
which corresponds exactly  to the inductive definition of the natural sum of ordinals.
Similarly, the inductive definition of the modified hyperoperation $H_2$ corresponds exactly
to Conway's inductive definition of $\alpha_\Co \cdot \beta_\Co$. 

Finally, the culminating point of Conway's ingenious approach to the Field $\No$ is
the proof that every non-zero number $x$ admits a multiplicative inverse 
(Theorem \ref{th:Field}) -- again, a relatively compact presentation is given by
Simons \cite{Sim}.
Norman Alling 
 (\cite{A}, p. 160), having filled  all the technical details left out
by Conway, remarks: [this technical work]
 "does not reduce the author's admiration of Conway's insight. Indeed,
to see that the ring $\No$ is a field seems remarkable indeed. To prove it in the way
that Conway did seems to the author little short of inspired."


\subsection{Cut representation of a number}

Every number
can be represented by certain {\em cuts}. This idea, which originates in Dedekind's construction
of the real numbers, is basic in Conway's approach. First recall various definitions of {\em cuts}
(see, e.g., \cite{A}, 1.20 and 4.02):

\begin{definition}\label{def:cuts} 
Let $(N,<)$ be a totally ordered set (or class).
\begin{itemize}
\item
A {\em  Conway cut} in $(N,<)$  is given by an ordered pair $\langle L,R \rangle$ of
sets $L,R\subset N$ such that $L<R$, that is:
$\forall a \in L, \forall b \in R: a<b.
$
\item
If moreover $N = L \cup R$, then the cut  is a {\em  Cuesta-Dutari (CD) cut} in $N$.
(Then $N$ has to be a set itself!)
\item
If moreover $L\not= \eset,R\not=\emptyset$,  it is a {\em Cuesta-Dutari-Dedekind  (CDD) cut}.
\end{itemize}
The {\em Cuesta-Dutari completion} of $(N,<)$ is the disjoint union
$N \cup \CD(N)$ of $N$ with the set $\CD(N)$ of all its Cuesta-Dutari cuts.
\end{definition}

\begin{theorem}[Fundamental Existence Theorem, cf.\ \cite{Go}, Th.\ 2.1]\label{th:FundamentalExistence}
$ $
\begin{enumerate}
\item
Every Conway cut $\langle L,R\rangle$ in $\No$
 determines a unique number $c = [\langle L, R \rangle]$ such that:
\begin{enumerate}
\item[(1)]
$L<c<R$, i.e., $\forall a \in L, \forall b \in R : a < c < b$,
\item[(2)]
$\Beta(c)$ is minimal with respect to (1).
 \end{enumerate}
Moreover, if $y$ is a number such that  $L<y<R$, then $c \preceq y$. 
 
\item
Let $\alpha$ be some fixed ordinal. 
For every Cuesta-Dutari cut $\langle L,R\rangle$
 in $\No_\alpha$, the number $c$ from the preceding item
belongs to  the boundary $\partial(\No_\alpha)$.
Conversely, every number $c \in \partial(\No_\alpha)$ gives rise to a CD-cut in $\No_\alpha$: 
$$
L^c = \{ y \mid \Beta(y) < \alpha, \mbox{and }  y < c \}, \quad
R^c = \{ z \mid \Beta(z) < \alpha,\mbox{and }   z >c\}.
$$
Both constructions are inverse to each
other. Thus $\No_{\alpha+1} = \No_\alpha \cup \partial (\No_\alpha)$ is the CD-completion of $\No_\alpha$.
\end{enumerate}
\end{theorem}

\begin{proof}
2. 
Assume $(L,R)$ is a CD-cut in $\No_\alpha$.
We prove by transfinite induction on $\alpha$
 that a unique number $c$ satisfying (1) and (2) exists and  that
 $\Beta(c)=\alpha$.

When $\alpha=0$, so $\No_0=\eset$, Condition (1) is empty, so Condition (2) characterizes
$c$ as the number with minimal birthday, $c= 0_\Co$.
 
 Assume the statement holds for all ranks $\beta <\alpha$.
The pair $\langle L_\beta,R_\beta \rangle :=\langle
L\cap \No_\beta, R \cap \No_\beta\rangle $ forms a CD-cut of $\No_\beta$.
By induction, there exist  unique numbers $c_\beta$ with $\Beta(c_\beta) = \beta$ and
$$
 L_\beta  < c_\beta <  R_\beta .
$$
Then  the family $(c_\gamma)_{\gamma < \alpha}$ forms a maximal chain in the
number tree.
Indeed, this 
follows from Theorem \ref{th:truncation}: for all
$\beta < \gamma < \alpha$, we have
$[c_\gamma]_\beta = c_\beta$ (the projections $x \mapsto [x]_\beta$ are monotonic,
so if this property would not hold, then we would have a contradiction to the 
fact that $c_\beta$ defines a CD-cut), whence $c_\beta \preceq c_\gamma$.
And the chain is maximal since $\Beta(c_\beta)=\beta$.

According to Theorem \ref{th:completeness}, the maximal path $(c_\gamma)_{\gamma < \alpha}$
has a limit $c$.
Again, by monotonicity (Theorem \ref{th:truncation}),
this limit $c$ is the cut number sought for.

\ssk
1. 
It suffices to show that
the class $I := \{ x \in \No \mid L < x < R \}$ is not empty, for
clearly it is convex, and therefore by Lemma
\ref{la:convex} then admits a unique element $c$ with minimal birthday.

Since $L$ and $R$ are sets, they are included in $\No_\alpha$ for some ordinal $\alpha$.
Define
$\tilde L:= \{ x \in \No_\alpha \mid \exists a \in L: x \leq a \}$ and
$\tilde R:= \{ x \in \No_\alpha \mid \exists b \in R: x \geq b \}$.
Then $(\tilde L,\tilde R)$ is a Conway cut.
If the class $I$ were empty, then the corresponding class $\tilde I$ for this cut would be 
empty, too, and $(\tilde L,\tilde R)$ would be a CD-cut of $\No_\alpha$.
According to Part 2., it would admit a cut number, in contradiction with our assumption that
$\tilde I$ is empty. 

If $L<y<R$, then we have one of the following:
$y \leq c$, so
$y \leq  \yca (y,c) \leq c$, or
$c \leq y$, so
$c \leq \yca(y,c) \leq y$.
In both cases, minimality of $\Beta(c)$ implies that
$\yca(y,c) = c$, that is, $c \preceq y$.
\end{proof}

\subsection{The canonical cut} 

\begin{definition}
Two Conway cuts $\langle L,R \rangle$ and $\langle L',R'\rangle$ are called
{\em equivalent} if they generate the same number, i.e., if
$[\langle L,R \rangle] = [\langle L',R'\rangle]$.
(Conway says they are "equal", and uses the sign for equality,
 but this way of talking is not adequate in the present  context.) 
\end{definition}

\begin{definition}\label{def:timely}
Following \cite{A}, p.\ 125, we say that a Conway cut $\langle L,R\rangle$ is
{\em timely} if its cut number is "new", in the sense that
$$
\forall x \in L \cup R: \qquad \Beta(x) < \Beta ([\langle L,R \rangle ]).
$$
\end{definition}

\begin{example}
The number $0_\Co$ has only one timely cut
representation: $(L,R) = (\eset,\eset)$. All  other cuts with
$L<0$ and $R>0$ also represent $0_\Co$, but are not timely. 
\end{example}

By Theorem \ref{th:FundamentalExistence}, every number $x$
is represented by its timely CD-cut
$$
L = \{ y \mid y < x, \Beta(y)<\Beta(x) \}, \qquad
R = \{ y \mid y>x, \Beta(y) < \Beta(x) \}.
$$
The CD-cuts  are the most "saturated" timely cuts: the sets $L$ and $R$ are the biggest
possible.
On the other hand, one may wish to work with cuts such that the sets $L$ and $R$ are
"the smallest possible" -- in \cite{Go}, Theorem 2.8 such a choice is called
{\em canonical representation} of a surreal number:

\begin{definition}\label{def:canonicalcut}
The {\em canonical cut} of a number $a$ is the one given by the
left and right families of its branch, see Def.\ \ref{def:CanCut} and Theorem \ref{th:CanCut}:
$a=[\langle L_a\mid R_a\rangle ]$ with

$L_a = \{ x \mid x \preceq a, x < a \}$,

$R_a  = \{ x \mid x \preceq a, x>a \}$.
\end{definition}

\begin{theorem}[Canonical cut]\label{th:CanCut}
The pair $\langle L_a,R_a \rangle$ just defined is a timely Conway cut
defining $a$, i.e.,
$[\langle L_a,R_a \rangle]=a$.
\end{theorem}

\begin{proof}
From the very definition,
$x<a<y$ for all $x \in L_a$, $y \in R_b$.

Let $c:= [\langle L_a,R_a \rangle]$ be the number defined by this cut.
Then $c \preceq a$ (by Theorem \ref{th:FundamentalExistence}, 1),  
so if $\Beta(c) <\Beta(a)$, then the very definition of
$L_a$ and $R_a$ implies that $c \in L_a$ or $c_\in L_b$:
contradiction.  So $\Beta(c)=\Beta(a)$ and $c=a$.
\end{proof}

\section{Equivalence with other approaches}\label{sec:equivalence}

We end this chapter by discussing the equivalence of our approach with other 
known approaches to Conway numbers.
From the very beginning, it is clear that our basic definition 
(Definition \ref{def:Number}) is equivalent to Gonshor's one, and our approach
really is a transcription of  Gonshor's one (\cite{Go}) into a more rigorous
set-theoretic setting.
Next, let us discuss the relation with Alling's approach \cite{A}.

\subsection{Cuesta-Dutari completions}

For every totally ordered set $N$, its CD-completion $N \cup \CD(N)$ (cf.\ Def.\ \ref{def:cuts})
carries a natural total order:

\begin{lemma}
If $(N,<)$ is totally ordered, then
there is a total order on $N \cup \CD(N)$, given for elements $x,y \in N$ and CD-cuts $(A,B),(A',B')$
by the prescription: 
$$
\begin{matrix}
x<y & \mbox{  if } &  x<y \mbox{ in } (N,\leq),
\\
x < (A,B) & \mbox{ if } & x \in A,
\\
(A,B) < y  & \mbox{  if  } & y \in B,
\\
(A,B) < (A',B') &\mbox{ if }  &  A \subset A'.
\end{matrix}
$$
\end{lemma}

\begin{proof}
Straightforward by distinction of cases --
see \cite{A}, 4.02.
\end{proof}

As we have seen (Theorem  \ref{th:FundamentalExistence}, 2), 
$\No_{\alpha + 1}$ is the CD-completion of $\No_\alpha$.
Starting from $\No_0 = \{ 0_\Co\}$, one can thus construct, by transfinite induction,
all higher stages by CD-completions and limits of preceding ones.
This defines $\No$ as ordered class, and it can be used to "reconstruct"
$\No$ by starting from the empty set -- see \cite{A}. 
However, this approach does not simplify in any way the definition of the
arithmetic structure on $\No$.

\subsection{Allings axiomatic approach}

The following definition is given in  \cite{A}, Section 4.03:

\begin{definition}\label{def:Alling}
For every ordinal $\beta$,
a {\em class of surreal numbers of height $\beta$} is a triple consisting of

(a) an ordered class $(F,<)$,

(b) a function $\Beta: F \to [0,\beta[$, such that

(c) for every Conway cut $(L,R)$ in $F$ with $\Beta(L), \Beta(R) < \alpha <\beta$,

$\quad$ there is a unique $c \in F$ such that $L<c<R$ and
$\Beta(c)$ is minimal. 

\nin
The class is called {\em full} if, under assumptions as in (c), we  have
$\Beta(c) \leq \alpha$.
\end{definition}

\begin{theorem}\label{th:Alling}
Any two full classes of surreal numbers of height $\beta$ are isomorphic to each
other under a unique isomorphism (order and birthday-preserving bijection).
\end{theorem}

Using this result, Alling (\cite{A}, p.\ 131) concludes that, since all known
constructions of Conway numbers  satisfy these properties,
they give isomorphic results.

\subsection{The full binary tree of numbers}

We have shown  that $(\No,\preceq)$ forms a full (complete and connected) 
binary tree.
Since the binary tree contains all order theoretic aspects of $\No$, the whole
theory can be reconstructed on these grounds (see \cite{E11, E12, E20}).

\subsection{Badiou's Number and Numbers}

Given a number $x$ in our sense, let 

$W:=\Beta(x)$ its birthday and
$F:= x \cap W = x \setminus \{ \Beta(x)\}$, a subset of $W$.

Conversely, given an ordinal $W$ and subset $F\subset W$, we find
$x = F \cup \{ W \}$.

\nin
Badiou (\cite{Ba}, Section 12.2) defines a {\em Number}
 to be such a pair $N= (W,F) = (W(N),F(N))$, and calls $W$ the
{\em matter} of $N$, and $F$ the {\em form} of $N$.
He also considers the {\em residue} $R(N) = W \setminus F$, so that
$(W,R(N))$ corresponds 
to our number $x^\sharp$.

Badiou states and proves a series of results very similar to the ones presented here;
however, he often uses long phrases instead of mathematical formalism, 
and this makes the development
hard to read and to check. 
Nevertheless, the whole presentation is mathematically sound and contains 
several interesting ideas (e.g., emphasizing the role of the {\em discriminant}).
\footnote{ 
Badiou's main interest is, of course, in philosophical questions,
and this hides or even obscures the purely mathematical contents of his work.
Professional mathematicians generally dislike such mixture of philosophy and
mathematics -- see, e.g., the very negative review of \cite{Ba} by
Reuben Hersh in the Mathematical Intelligencer 31 (3), 2009, p. 67-69, which
should be compared with the more in-depth review by J. Kadvany,
Notre Dame Philosophical Reviews, 
\url{https://ndpr.nd.edu/reviews/number-and-numbers/}. }


\subsection{Conway's original approach: combinatorial game theory}

We'll discus the link with this approach in full detail in Chapter \ref{chap:Games}.

\chapter{Set theory and Game theory}\label{chap:Games}\label{chap:CGT}

The definition of surreal numbers given in the preceding chapter is close to Gonshor's sign-sequence
approach. Conway's original approach is to define numbers as special
instances of {\em games}. His definition of number (\cite{Co01}, p. 4) has the same structure 
as the one of a game  (\cite{Co01},  p.78/79):

{\textbf{Construction.}} {\em
If $L,R$ are any two sets of games, 
 there is a game $\{ L \vert R \}$. All games are constructed this way.}

\nin
Conway puts forward two main arguments why this approach should be the "good one" 
 (cf.\ \cite{Co01}, Epilogue, p. 225/26): 
\begin{enumerate}
\item[(a)]
{\em The greatest delight, and at the same time, the greatest mystery, of the Surreal numbers
is the amazing way that a few simple "genetic" definitions magically create a richly structured
Universe out of nothing.}
\item[(b)]
{\em 
The sign-sequence definition has also the failing that it requires a prior construction of the 
ordinals, which are in ONAG produced as particular cases of the surreals.
To my mind, this is another symptom of the same problem, because the definitions that work
universally should automatically render such prior constructions unnecessary.}
\end{enumerate}
Concerning (b), see remarks in the Introduction (Section \ref{sec:CGT}): I think this argument  cannot be maintained.
Concerning (a), it certainly expresses a deep feeling Conway had about his own creation, and it
has to be taken seriously.
However, Conway seems to overlook that the same "delight" and "mystery" already must have
been experienced by the founders of set theory, and in particular by von Neumann, when
he "created the richly structured universe $\vN$ out of nothing".
For this reason, in this chapter, I will try  to define a clear framework for the
approach by Combinatorial Game Theory (CGT), so that one can compare both approaches on firm grounds. 
I closely follow the presentation from \cite{S}, with some  deviations. 

Logically and historically, CGT developed in two steps:
\begin{itemize}
\item
the theory of {\em impartial games} (Sprague, Grundy, 1930ies),
\item
the theory of {\em partizan games} (\cite{Co01}, Chapter 1, \cite{WW}).
\end{itemize}
Although surreal numbers only arise in the latter, it is necessary to start with the former.
For me, the main point is that the theory of impartial games is what I call
"pure set theory" -- this has very clearly been remarked by Lenstra 
 \cite{Le}. Thus, it can be considered as a topic in pure, foundational mathematics, and 
in the following section we shall
 present it like this. 
 
 In the second section, we will define the analog setting for partizan games:
 following Conway, \cite{Co01}, p.66,
 
 \begin{itemize}
 \item[ ]
 {\em
 Plainly the proper set theory in which to perform a formalisation would be one with two kinds of
membership.}
\end{itemize}

\nin This is exactly what we will do.

\section{Pure set theory, and impartial games}\label{sec:pure}\label{sec:impartial}

By "pure set theory" we mean the algebraic theory of $\vN$, the universe of pure sets, 
by taking seriously its structure -- a set is a set of, set of... sets.  As said in
Section \ref{sec:puresettheory}, this way of looking at pure sets has a "quantum" taste, as 
opposed to the "classical" viewpoint on sets. In this sense, game theory is "quantum set theory".

\begin{definition}
An {\em (impartial) game} is just a pure set $G$. In this context $\vN$ is called the {\em universe of
impartial games}, and an element $g \in G$ is also called an {\em option of $G$}. 
An  {\em $n$-th order element} $g \in^n G$, or {\em position of $G$}, is inductively defined:

$\qquad g \in^1 G$ iff $g \in G$, and $g \in^{n+1} G$ if $\exists u \in^n G$: $g \in u$. 

\nin
A {\em complete chain, or run, in a game $G$} is given by a sequence of pure sets
$$
G = G_0 \ni G_1 \ni G_2 \ni \ldots \ni G_n = \eset ,
$$
where $n\in \N$.
Thus every position arises in some complete chain.
\end{definition}

All runs are of {\em finite} length:
there is no point in defining $x \in^\alpha A$ for infinite ordinals $\alpha$, because
by construction of the von Neumann universe, there are 
{\em no infinite descending chains of sets} (this can be seen as a 
\href{https://en.wikipedia.org/wiki/Axiom_of_regularity#No_infinite_descending_sequence_of_sets_exists}{consequence of
the axiom of foundation}).
Note that  {\em second order} elements appear quite frequently in "usual" maths
(via various \href{https://de.wikipedia.org/wiki/Mengensystem}{systems of sets}), but third and higher order
elements hardly ever.

\begin{theorem}
There are two types of (impartial) games:
every game $G$ is either 
 {\em fuzzy} or of  {\em zero type}, 
defined inductively as follows: $G$ is
\begin{enumerate}
\item
of {\em zero type} if all elements $G_1 \in G$ are fuzzy,
\item
 {\em  fuzzy}, if there exists an element $G_1 \in G$ of zero type.
\end{enumerate}
\end{theorem}

\begin{proof}
Clearly, $G = \eset$ is of zero type, and not fuzzy. 
By induction, assume the claim true below a certain rank $\alpha$, and
$G$ of rank $\alpha$. 
Then both conditions are well-defined propositions, and one is just the negation of the other,
so $G$ must be either of zero type or not (i.e., fuzzy).
\end{proof}

\nin {\textbf{Game interpretation.}}
"Playing $G$" means to fix a run 
$G = G_0 \ni G_1  \ni \ldots \ni G_n = \eset$
where $n\in \N$. 
The "first" player (W: White) chooses $G_1$, the "second player" (B: Black) $G_2$, etc.
In a fuzzy game, W has a "winning strategy" (she just has to choose a zero-type element
$G_1 \in G$, and then B is lost),  in a zero-type game, B has one.
To be precise, this is 
{\em normal-play convention}. There is also another convention, called
{\em  mis\`ere-play}, which is less natural: $G$ is
\begin{enumerate}
\item
 of {\em mis\`ere zero type} if either $G = \eset$, or there is
an element $G_1 \in G$ of mis\`ere fuzzy type,
\item
of {\em mis\`ere fuzzy type} if $G \not= \eset$, and
all  $G_1 \in G$ are of mis\`ere zero-type.
\end{enumerate}
Then an analog of the above theorem holds, with same proof; but the exception for $G = \eset$
has to be stated explicity, if we wish that $\eset$ be of mis\`ere  zero-type
(see \cite{S}, Chapter V, for mis\`ere-play theory).
This makes mis\`ere-play more cumbersome,
and we will not pursue this.

\begin{definition}
To every pure set $G$ we associate an ordinal $\Gamma(G)$, its {\em Grundy number}, as follows. 
For every set $A$ of ordinals, the {\em excluded minimum} is the ordinal
\begin{equation}  
\mex(A) := \min \{ \alpha \mid \alpha \mbox{ ordinal, } \alpha \notin A \}.
\end{equation}
Note that $\mex(A) = 0$ iff $0 \notin A$.
Now define, by induction,
\begin{equation}
\Gamma(G) := \mex \{ \Gamma(g) \mid g \in G \}   .
\end{equation}
\end{definition}

\nin Example.
We compute the Grundy numbers for the pure sets from Table \ref{table:beginning}:
\begin{table}[!h] \caption{The first Grundy numbers}\label{table:Grundy}
\begin{center}
\begin{tabular}{|*{10}{c|}}
\hline
order nr. & symbol & rank  & depth $1$ & Grundy nr. & outcome 
\\
\hline
$0$ & $\eset$ & $0$ & $\eset$ & $0$ & zero-type
\\
\hline
$1$ & $\I$ & $1$ & $\{ \eset \}$ &  $1$ & fuzzy
\\
\hline
$2$ & $\II$ & $2$ & $\{ \I \}$ & $0$ & zero-type 
\\
$3$ & $\III$ & $2$ & $\{ \I,\eset \}$  & $2$ & fuzzy 
\\
    \hline
$4$ & IV & $3$ & $\{ \II \}$   & $1$ & fuzzy
\\
 $5$ & V & $3$ &  $\{ \II,\eset \}$ & $1$ & fuzzy 
 \\
 $6$ & VI &  $3$ &  $\{ \II,\I \}$ & $2$ & fuzzy
 \\
 $7$ &VII  &  $3$ &  $\{ \II,\I ,\eset\}$ & $2$ & fuzzy 
 \\
 $8$ &VIII  & $3$ & $\{ \III \}$   & $0$ & zero-type 
\\
 $9$ & IX & $3$ & $\{ \III , \eset \}$   & $1$ & fuzzy 
\\
 $10$ & X  & $3$ & $\{ \III , \I \}$   & $0$ & zero-type
\\
 $11$ & XI & $3$ & $\{ \III , \I, \eset  \}$   & $3$ & fuzzy
\\
 $12$ & XII  & $3$ & $\{ \III , \II  \}$   & $1$ & fuzzy 
\\
 $13$ &  XIII & $3$ & $\{ \III , \II ,\eset \}$   & $1$ & fuzzy 
\\
 $14$ & XIV  & $3$ & $\{ \III , \II ,\I \}$   & $3$ & fuzzy 
\\
 $15$ & XV  & $3$ & $\{ \III , \II ,\I ,\eset \}$   & $3$ & fuzzy 
\\
\hline
$16$ & XVI & $4$ & $\{ \mbox{ IV } \}$ &  $0$ & zero-type
\\ 
    \end{tabular}
\end{center}
\end{table}

\nin By induction, always $\Gamma(x) \leq \rk(x)$, and,  $\Gamma(\alpha)=\alpha$ for all ordinals $\alpha$.

\begin{theorem}
A game
$G$ is of zero-type iff $\Gamma(G)=0$, and fuzzy
 iff $\Gamma(G) >0$. 
\end{theorem}

\begin{proof}
Since $\Gamma(G) = 0$ iff $0 \notin \{ \Gamma(g) \mid g \in G\}$,
the condition $\Gamma(G)=0$ means, by induction,  that no element $g \in G$ is of zero-type, hence
all elements of $G$ are fuzzy, hence $G$ is of zero type. 
Conversely, $\Gamma(G)>0$ means, by induction, that some element of $G$ is of zero-type, hence
$G$ is fuzzy.
\end{proof}

\begin{definition}
Let us call {\em Grundy-Map}, and  write
$\Gamma : \vN \to \On$,
for the collection of maps $\Gamma_{\vN_\alpha} : \vN_\alpha \to \alpha$.  
We write $G \equiv H$ iff $\Gamma(G)=\Gamma(H)$.
\end{definition}

The fiber over $0$ is the class of zero-type games.
The class of fuzzy games is characterized by non-zero Grundy numbers.
A key result of \cite{Co01}, Chapter 6, is that there are algebraic structures turning
$\Gamma$ into a morphism.

\subsection{Disjunctive, and other sums of pure sets}\label{sec:sum_impartial}

Conway realized that, via transfinite $\epsilon$-induction, one can define various
algebraic laws on universes such as $\vN$ or $\On$. 
In \cite{Co01}, Chapter 14, he distinguishes 3 basic "rules for moving in the compound game",
called {\em selective, conjunctive, disjunctive}; but in fact there are  more (cf. the table
p.41 in \cite{S}), 
and a systematic algebraic theory seems to be
 missing so far.  We do not pretend to fill this gap by the present
section, but just give some examples.  
Let us call "Conway-style" the way of defining such laws:

\begin{theorem}\label{th:Conway-sum}
The following  operations are well-defined by transfinite induction, for pure
sets $F$ and $G$, and all of them are associative: 
\begin{align*}
F \ast_1 G & := \{ F \ast_1 g , f \ast_1 g \mid f \in F, g\in G \},
\\
F \ast_2 G & : = \{ F \ast_2 g , f \ast_2 G,  f \ast_2 g \mid f \in F, g\in G \},
\\
F \ast_3 G & := \{ F \ast_3 g , f \ast_3 G \mid f \in F , g \in G \},
\\
F\ast_4 G & := \{ F \ast_4 g, f \mid f \in F, g \in G\} = F \cup \{ F \ast_4 g \mid g \in G \},
\\
F\ast_5 G & := \{ F \ast_5 g, f \ast_5 G, f',g' \mid f,f' \in G, g,g' \in G \} 
\\ &
= F \cup G \cup \{ F \ast_5 g \mid g\in G\} \cup \{ f \ast_5 G \mid f\in F \} .
\end{align*}
\end{theorem}

\begin{proof}
Note that $\eset \ast_i \eset = \eset$, for $i=1,\ldots,5$, since the conditions are empty.
Assume $f \ast \eset$ is defined for all pure sets $f$ of lower rank than the one of $F$.
Then by induction the given formulae define $F \ast \eset$, for all pure sets $F$.
Next,
assume $F \ast g$ is defined for all pure sets $F$ and all $g$ with lower rank than the one of $G$.
Then by induction the formulae define $F \ast G$. 
By construction, when $g \in G$, the rank of $F \ast g$ is lower than the rank of $F \ast G$.
Associativity is, in the words of Conway,  a "one-line proof", using the
definition and  the induction hypothesis,
e.g.,
\begin{align*}
(F \ast_1 G)\ast_1 H & = \{ (F\ast_1 G) \ast_1 h ,  (F \ast_1 g) \ast_1 h , 
(f \ast_1 g) \ast_2 h \mid f\in F, g\in G,h\in H  \} \\
& = \{ F\ast_1 ( G \ast_1 h ),  F \ast_1( g \ast_1 h ), 
f \ast_1 (g  \ast_2 h)  \mid f\in F, g\in G,h\in H  \} \\
&=
F \ast_1 (G  \ast_1 H),
\\
F\ast_4 (G \ast_4 H) & = \{ F \ast_4 G \ast_4 h, F \ast_4 g, f \mid f \in G,g\in G,h \in H\}
= (F\ast_4 G) \ast_4 H, 
\end{align*}
and similarly for the other laws (cf.  \cite{S}, p.40 : "all of these operations are associative").
\end{proof}

\begin{definition}\label{def:disjunctive-impartial}
We call the operation $\ast_3$ the
 {\em  disjunctive sum}  of two impartial games (pure sets), denoted in this context by $+$ or
 by $+_d$ :
$$
F+G := F+_d G :=  F\ast_3 G = 
 \{ f +_d G , F +_d  g \mid f \in F, g\in G \} .
$$
\end{definition}

\begin{example}
$\I+\I = \{ \I \} = \II$, 

\nin
$\I+\II=\{ \II\}=$IV, 

\nin $ \II + \II = \{ \I+\II\}=$XVI.

\nin Since $1+1 = \II$ is not an ordinal, 
 the class of ordinals is not stable under disjunctive sum.
 The following definition remedies this. 
\end{example}

\begin{definition}
For two ordinals $\alpha,\beta$, we define recursively their
\href{https://en.wikipedia.org/wiki/Nimber#Addition}{\em nimber sum}
$$
\alpha \diamond \beta := \mex \{ \alpha \diamond \beta' , \alpha' \diamond
\beta \mid \beta' \in \beta, \alpha' \in \alpha \}. 
$$
\end{definition}

\begin{theorem}
\begin{enumerate}
\item
The Grundy-map $\Gamma$ is a morphism from $+$ to $\diamond$:

for all pure sets $G,H$, we have
$\,  \Gamma(G +H) = \Gamma(G) \diamond \Gamma(H)$.
\item
Nimber-sum is associative, commutative, and
$\alpha \diamond \beta = 0 $ iff $\alpha = \beta$.

In particular, $\Gamma(G+H)=0$ iff $\Gamma(G)=\Gamma(H)$ (iff $G \equiv H$).
\item
$\vN$ with disjunctive sum forms a Monoid, and the Quotient $\vN / \equiv$
 forms  a Group, isomorphic to $(\On,\diamond)$.
\end{enumerate}
\end{theorem}

\begin{proof}
By the definitions, for ordinals $\alpha,\beta$,
$$
\Gamma(\alpha + \beta) = \mex \{ \Gamma(\alpha'+\beta) , \Gamma(\alpha+\beta') \mid
\alpha' <\alpha,\beta' <\beta \}.
$$
By an inductive argument, we see that this equals $\alpha \diamond \beta$. 
It follows that nimber sum is associative.
Similarly, it follows that $\Gamma(G +H) = \Gamma(G) \diamond \Gamma(H)$.
Finally, it is shown by induction that
$0 \in \alpha \diamond \beta$ iff $\alpha \not= \beta$. 
From this, it follows that
$\Gamma(G+H)\not= 0$ iff $\Gamma(G) \not= \Gamma(H)$.
In particular, $\alpha \diamond \alpha = 0$, so the ordinals with nimber sum form a 
Group, to which $\vN / \equiv$ is isomorphic. 
\end{proof}

Although this is not needed for the theory of impartial games, let us return to the
"sums" $\ast_4$ and $\ast_5$ :

\begin{theorem}\label{th:Conway-style}
The (von Neumann) ordinals are stable under $\ast_4$ and $\ast_5$, and:

$\alpha \ast_4 \beta = \alpha \oplus \beta$ is the ordinal sum of ordinals $\alpha$ and $\beta$, 

$\alpha \ast_5 \beta$ is the Hessenberg sum of $\alpha$ and $\beta$.
\end{theorem}

\begin{proof}
Compare with formulae (\ref{eqn:O1}) and (\ref{eqn:O2}): when $F = \alpha$ ad $G =\beta$
are (von Neumann) ordinals, then $\ast_4$ and $\ast_5$ coincide with the description of sets
given by these formulae, whence the claim.
(The "sum" $\ast_4$ is in \cite{Co01,S} denoted by $F:G$, and is called "ordinal sum";
concerning $\ast_5$, see e.g., \href{https://math.stackexchange.com/questions/3715343/explicit-formulas-for-natural-hessenberg-and-ordinal-sum-do-those-work}{question 3715343 on math.stackexchange}). 
\end{proof}

Note the similarity in logical structure
 between the operations $\ast_3$ and $\ast_5$, leading to nimbers, on the
one hand (see next section), and to ordinals with Hessenberg sum, on the other hand.

\subsection{Product, and the Field of nimbers}\label{sec:Nimbers}

The definition of multiplication is more delicate than the definition of addition.
For instance, usual ordinal, or Hessenberg multiplication does not extend to $\vN$ in
a way admitting all desired properties.
Concerning  "impartial pure set theory", we define: 

\begin{definition}
The {\em product} of two pure sets  $F,G$ is inductively defined by
$$
F \cdot G :=  \{ (f \cdot G) +  (F \cdot g) + (f \cdot g)
 \mid f \in F, g\in G \} ,
$$
where $+$ means $+_d$,
and the \href{https://en.wikipedia.org/wiki/Nimber#Multiplication}{\em nimber product}
of two ordinals inductively by
$$
\alpha \circ  \beta := \mex \{ (\alpha \circ  \beta' )\diamond   ( \alpha' \circ
\beta) \diamond  ( \alpha' \circ  \beta' ) \mid \beta' \in \beta, \alpha' \in \alpha \} .
$$
\end{definition}

By induction, just like for addition, we see that
$\Gamma(G \cdot H) = \Gamma(G) \circ \Gamma(H)$.
But whereas the product on $\vN$ fails to have the properties of a product in a Ring,
the product on $\On$ has excellent properties --
the following is the main result of \cite{Co01}, Chapter 6, cf.\  \cite{S}, Section VIII.4:

\begin{theorem}
The class of ordinals $\On$ together with nimber addition and nimber multiplication is a 
Field of characteristic $2$.
\end{theorem}

The Field of nimbers is the analog, in the impartial theory, of the Field of numbers
for partizan games. 
In loc.\, cit., more is proved about the nimbers: e.g., 
for every finite ordinal $n$, the Galois field $\F_{2^{2^n}}$ is a subfield of the nimbers, 
 and $(\N , \diamond, \circ)$ is a subfield,
isomorphic to the direct  limit of these.
Just as for numbers, there are two viewpoints:
either, view the nimbers as the ordinals together with special laws defining a Field structure,
or view them as $\vN$, quotiented by $\equiv$, together with a rather natural sum, and
a  less natural product.

\section{Partizan games, graded sets}\label{sec:partizan}

Impartial games correspond to usual, pure set theory, and partizan games correspond to a
similar theory, where now 
{\em the proper set theory in which to perform a formalisation would be one with two kinds of
membership}.
In principle, this is carried out in \cite{S}, Chapter VIII, albeit in a different language.
As stressed in the introduction, the main issue is about {\em ordered pairs}:
there seems to be no {\em natural} 
operation in pure set theory that modelizes "the" 
\href{https://en.wikipedia.org/wiki/Cartesian_product}{\em Cartesian product of sets},
(nor  "the"
\href{https://en.wikipedia.org/wiki/Disjoint_union#Set_theory_definition}{\em disjoint union}, that is, the {\em set theoretic coproduct}).
Every  such construction depends on choices, or on "contexts", specific to 
  a given situation.
 The usual convention of set theory is to define ordered pairs via the 
 \href{https://en.wikipedia.org/wiki/Ordered_pair#Kuratowski's_definition} {\em Kuratowski construction}
 \begin{equation}\label{eqn:Kuratowski}
 \langle a,b \rangle := \{ \{ a,b\}, \{ a \} \} 
 \end{equation} 
 and Cartesian products as the set of all ordered pairs,
 \begin{equation}\label{eqn:Cartesian} 
 a \times b := \{ \langle x,y  \rangle \mid x\in a, y \in b\} .
 \end{equation}
 This has the desired properties. But any other injective family of maps
 $f=f_\alpha:\vN_\alpha \times \vN_\alpha \to \vN_{\alpha + k}$ (for some $k \in \N$) could serve to define the notion of
 ordered pair, and the choice $f((a,b)) = \langle a,b \rangle$ (with $k=2$) appears to be arbitrary. 
 Note also that in general 
 $(a \times b) \times c \not= a \times (b \times c)$. 
  As far as I see, all known variants of the Kuratowski construction share 
these inconvenients.

\subsubsection{The graded von Neumann universe}

Our aim is to define a hierarchy having the same properties as the von Neumann hierarchy
$\vN$, but with two kinds of membership, $\in_L$ (left member) and $\in_R$ (right member) 
instead of just one type $\in$.
Technically, it will be defined as a subclass of $\vN$, by using one
 of the current definitions of {\em ordered pair} $\langle x, y \rangle$ of pure sets
$x,y$. For convenience, think of
the Kuratowski definition (\ref{eqn:Kuratowski}), but any other would be equally convenient.
The imbedding of our new universe in $\vN$ will depend on this choice, but its
"intrinsic" properties won't.

\begin{definition}\label{def:gvN}
Fix some definition of "ordered pair" in pure set theory, to be denoted by \begin{equation}\label{eqn:Kuratowski2}
 a \sqcup b  := \mbox{ ordered pair, often also denoted by }  (a,b).
 \end{equation} 
If $z = x \sqcup y$ is an ordered pair,
we call $x = z_L$ its {\em left component}, and $y = z_R$ its {\em right component}.
We define a subclass $\gvN$ of $\vN$, which we call the {\em graded von Neumann universe},
 as follows: let $\gvN_0 = \eset$,
$$
\gvN_{\alpha + 1} = \cP(\gvN_\alpha) \times \cP(\gvN_\alpha) = \{
y =  y_L \sqcup  y_R  \mid y_L, y_R \in \cP(\gvN_\alpha) \}
$$
for successor ordinals $\alpha + 1$, and if $\lambda$ is a limit ordinal, we let
$$
\gvN_\lambda = \bigcup_{\alpha < \lambda} \gvN_\alpha .
$$
Elements of $\gvN_\alpha$ are called {\em partizan games, {\em or}
graded sets (of rank at most $\alpha$)}.
Elements $a,b\ldots$ of $y_L$ are called {\em left elements (left options) of $y$}, and elements 
$v,w,\ldots$ of
$y_R$ {\em right elements (right options) of $y$}; we write
$$
a \in_L y \mbox{ for } a \in y_L, \qquad
v \in_R y \mbox{ for } v \in v_R.
$$
Thus $y = \{ a,b,\ldots\} \sqcup  \{ v,w,\ldots \}$ is, in our notation,
 the set which Conway often writes like this:
$y=\{ a,b,\ldots \mid v,w,\ldots \}$.
\end{definition}

\begin{remark}
Our notation $a \sqcup b$ for the ordered pair should remind the idea of 
"disjoint union with first set $a$ and second set $b$".
This also justifies Conway's notation, which would be an excellent way to describe
such disjoint unions, if it were not in conflict with set-builder notation -- we wish to 
use the vertical bar to indicate set-builder conditions, and hence do not use
Conway-notation.
\end{remark}

\begin{remark}\label{rk:Simons}
Following Simons, \cite{Sim}, it is sometimes useful to turn $\{ L , R \}$ into a group
with neutral element $R$, and to write $P$ or $O$ for an element of this group.
So, $\xi \in_P x$ means $\xi \in x_P$.
The letters $L,R$  look like   new "urelements" of set theory. 
One could also use $A,B$ for "Arthur" and "Bertha" (\cite{Co01}, p.\ 71).
Probably, in pure set theory, it would be better to use $\I, 0$, but we are going to follow the convention
from CGT.
\end{remark}

\begin{remark}
As a class, our $\gvN$ is the same as the class $\tilde{\mathbf{PG}}$ defined on pages 53/54 and
 398  in 
\cite{S}, but our labelling by ordinals $\alpha$, and hence the following rank definition, is slightly different:
e.g., we start with  $\gvN_0 = \eset$, whereas 
for $\alpha = 0$, the definition in loc.\ cit.\ yields $\tilde {\mathbf{PG}}_0 = \{ 0\} = \{ \langle \eset, \eset \rangle \}$,
and likewise for all limit ordinals $\lambda$, the definition from loc.\ cit.\ directly yields our
$\gvN_{\lambda + 1}$.
In other terms, the stages in loc.\ cit.\ are rather labelled by "Conway ordinals", than by "von Neumann ordinals".
\end{remark}

\begin{definition}
The {\em (graded) rank} of a graded set (partizan game) $G$ is the least $\alpha$
such that $G \in \gvN_{\alpha + 1}$.
\end{definition}

For $n \in \N$, we have
$\card(\gvN_{n+1}) = 2^{2 \cdot \card(\gvN_n)}$. 
For $n=1$, the only element of $\gvN_1$ is the "zero-game" ("graded zero")
\begin{equation}
0_G  : = \eset \sqcup  \eset .
\end{equation}
Here is $\gvN_2$ and some elements of $\gvN_3$
(our order-numbers are ad hoc): 

\begin{table}[!h] \caption{Beginning of the graded universe}\label{table:partizan}
\begin{center}
\begin{tabular}{|*{10}{c|}}
\hline
order nr. & CGT symbol & rank  & depth $1$ & Conway notation & type 
\\
\hline
$0$ & $0_G$ & $0$ &  $\eset  \sqcup \eset$ &   $\{ \mid \}$ & $\equiv 0$
\\
\hline
$1$ & $1_G$ & $1$ &  $ \{ 0_G   \} \sqcup  \eset $ & $  \{ 0 \mid \}$ & $>0$ 
\\
$2$ & $(-1)_G$ & $1$ &  $ \eset \sqcup  \{ 0_G   \} $ & $ \{  \mid 0 \}$ & $<0$
\\
$3$ &  $*$ & $1$ &  $ \{  0_G  \} \sqcup  \{  0_G   \}  $ & $ \{ 0 \mid 0 \} $ & $\Vert 0$ 
\\
\hline
$4$ & $2_G$ & $2$ & $\{ 1_G\} \sqcup \eset $ & $\{ 1 \vert \}$ & $>0$
\\
$5$ & $\uparrow$ & $2$ & $ \{ 0_G \} \sqcup \{ * \}$ &  $\{ 0 \vert *\}$ & $>0$ 
\\
$6$ & $\downarrow$ & $2$ & $ \{ * \} \sqcup \{ 0_G \}$ &  $\{ * \vert 0\}$ & $<0$ 
     \end{tabular}
\end{center}
\end{table}

\nin
We  leave to a machine the task of
giving  a list of the $2^{8} = 256$ elements of $\gvN_3$.

\begin{remark} 
By construction, $\gvN$ is included in $\vN$, but there are also converse inclusions.
By induction, we define three Imbeddings of $\vN$ into $\gvN$:
\begin{enumerate}
\item
$i_L$, Left Imbedding: $i_L(y) :=  \{ i_L(x) \mid x \in y \} \sqcup  \eset $,
\item
$i_R$, Right Imbedding: $i_R(y) :=  \eset \sqcup  \{ i_R(x) \mid x \in y \} $,
\item
$i_D$, Diagonal Imbedding: $i_D(y) :=  \{ i_D(x) \mid x \in y \} \sqcup
 \{ i_D(x) \mid x \in y \}$.
\end{enumerate}
The image of $i_L$ consists of "positive" elements, and the one of $i_R$ of "negative" elements
(see below).
The Diagonal Imbedding can be  (and is, in \cite{S})
  used to consider the theory of impartial games (preceding section)
 as a "subtheory" of the theory of partizan games. 
\end{remark}

\begin{definition}
For a single partizan game $G$, there is an {\em opposite game} $G^{\op}$ defined by:
$g \in_L G^{\op}$ iff $g \in_R G$, and 
$g \in_R G^{\op}$ iff $g \in_L G$.

More interestingly, by induction for all partizan games one can define 
 the {\em  Conway-opposite} $G^\sharp$ (which Conway denotes by $-G$) by
$$
G^\sharp :=  \{ g^\sharp \vert g \in_R G\} \sqcup  \{ g^\sharp \vert g \in_L G \}  .
$$
In Simons' convention (Remark \ref{rk:Simons}):
$g \in_P G^\sharp$ iff
$g^\sharp  \in_{P+L} G$.
\end{definition}

At lowest ranks, opposite and Conway-opposite coincide:
from the four elements of $\gvN_2$, $0_G$ and $*$ are fixed, and
$1_G$ and $(-1)_G$ are exchanged.

\begin{theorem}[The four outcome classes]
Every graded set (partizan game) $G$ is either 
 {\em positive} or {\em not positive}, 
 and either {\em negative} or {\em not negative},
defined inductively: $G$ is
\begin{enumerate}
\item
{\em (a) positive}, $G \geq 0$, if: $\forall g \in_R G$: not($g \leq 0$),

{\em (b) not positive}, $\lnot(G \geq 0)$, or: not$(G \geq 0)$, if : $\exists g \in_R G $: $g \leq 0$,
\item
{\em (a) negative}, $G \leq 0$, if: $\forall g \in_L G$: not($g \geq 0$),

{\em (b) not negative}, $\lnot(G \leq 0)$, or: not$(G \leq 0)$, if: $\exists g\in_L G$: $g \geq 0$,
\end{enumerate}
and Case i.(a) is the logical negation of Case i.(b).
Therefore, every partizan game $G$ is of exactly one of the following four types:
\begin{enumerate}
\item
$G  \equiv 0$: $G$ is both positive and negative (also: "of zero type"),
\item
$G \Vert 0$: $G$ is both not positive and not negative (also: "fuzzy"),
\item
$G >0$: $G$ is positive and not negative (also: "strictly positive"),
\item
$G<0$: $G$ is negative and not positive (also: "strictly negative").
\end{enumerate}
Moreover, the operator $\sharp$ exchanges "positive" and "negative",
hence exchanges $G>0$ and $G<0$, and  it preserves
$G \equiv 0$ as well as $G \Vert 0$.
\end{theorem}

\begin{proof}
By induction: note that the zero-grame $0_G=\widehat \eset$ has no left and no right elements,
so it is positive and negative, and not (not positive or not negative).
If the claim holds at graded rank below $\alpha$, and $G$ has graded rank $\alpha$,
then all four statements are well-defined propositions, and 
(not positive/not negative) is the logical negative of (positive/negative). 
Therefore the properties are well-defined, and we get the four cases.
By induction, it is clear that $\sharp$ exchanges $>$ and $<$, but preserves $\equiv$ and
$\Vert$. 
\end{proof}


\nin
\textbf{Game interpretation.}
See explanations p.72/73 concerning Theorem 50  in \cite{Co01}.
(A description by words seems more difficult, to me, than the formalism given above.
There are two kinds of "runs", starting with $\in_L$, resp.\ with $\in_R$, 
and "strategy" refers to an existence quantifier "$\exists$": "make a good choice". In the end,

$G >0$: there is a winning strategy for $L$,

$G<0$: there is a winning strategy for $R$,

$G \equiv 0$: there is a winning strategy for the "second" player,

$G \Vert 0$: there is a winning strategy for the "first" player.

\nin
The letters $L$ and $R$ seem to be used here in several different ways, as names
for "players" -- a kind of variable --, and as fixed "urelements" denoting an absolute
choice of order. From a logical point of view, this is rather confusing.)
Notice that we are in "normal-play convention": the zero-game satisfies $0_G \equiv 0$, as it should.
Analogous definitions for "mis\`ere-play convention" would be clumsy, and will not be 
considered here.
Next, one would like to define an analogue of the "Grundy-map" $\Gamma$, but this turns out
to be much more involved. First of all, the arithmetic operations:

\begin{definition}
We define inductively the {\em disjunctive sum} of $G,H \in \gvN$ by 
$u \in_P G+H$ iff [ $\exists g \in_P G$: $u = g+H$, or $\exists h \in_P H$: $g = G + h$ ]:
$$
G + H := 
 \{ G +h, g +H \vert h \in_L H,g\in_L G \} 
 \sqcup  \{ G +h, g+H \vert h \in_R H, g \in_R G \} 
$$
(cf.\ Theorem \ref{th:Conway-sum} for the principle of such
Conway-style definition). 
\end{definition}

\begin{theorem}
The disjunctive sum  is associative and  commutative, and has neutral element $0_G$.
The operator $\sharp$ is an Automorphism of the additive monoid, of order two:
$(G + H)^\sharp = G^\sharp + H^\sharp$, 
$(G^\sharp)^\sharp = G$. 

For every partizan game $G$, the game $G + G^\sharp$ is of zero-type:
 $G + G^\sharp \equiv 0$.
 \end{theorem}
 
 \begin{proof}
By induction, via "one-line proofs": \cite{Co01}, Chapter 1.
\end{proof}

\begin{definition}
Following \cite{Co01,S}, we write $G - H:= G + H^\sharp$, and let
\begin{enumerate}
\item
$G  \equiv H$ iff $G - H \equiv 0$,
\item
$G \Vert H$ iff $G - H \Vert 0$,
\item
$G > H$ iff $G - H > 0$, and we write  $G \geq H$ iff $G - H \geq 0$,
\item
$G< H$ iff $G - H < 0$, and we write $G \leq H$ iff $G - H \leq 0$.
\end{enumerate}
\end{definition}


\begin{theorem}
The Relation $G \equiv H$ 
 is an Equivalence Relation on $\gvN$, defining a Congruence Relation
for the disjunctive sum $+$, and  compatible with with $\sharp$.

The Relations $\leq$ and $\geq$ are 
\href{https://en.wikipedia.org/wiki/Preorder}{preorders} on $\gvN$ (transitive and reflexive),
and define partial orders on the quotient $\gvN/\equiv$.

The quotient ${\mathbb G} :=\gvN/\equiv$  becomes an ordered abelian Group,
with Group inversion induced by  $\sharp$.
\end{theorem}

\begin{proof}
By induction: 
 \cite{Co01}, Chapter 1,  see also \cite{S} or \cite{Sch}.
\end{proof}
 
 \begin{definition}
 One distinguishes between a {\em game form} $G \in \gvN$, and its equivalence class,
 the {\em game value} $[G] \in {\mathbb G} = \gvN / \equiv$.
Here,  under
the "quotient" $\gvN/\equiv$ (which would be a quotient of proper classes),
 we understand the hierarchy of quotients of sets
$(\gvN_\alpha /\equiv)_\alpha$, indexed by ordinals $\alpha$.
 In game theory, one is rather interested in the game value, and therefore considers games as
 "equal" if  they have the same game form.
 \end{definition}

\begin{definition}
We define  the {\em Conway-product} of two games by
$G \cdot H := $
\begin{align*}
&  \{ G \cdot h +  g \cdot H - g \cdot h,  G \cdot h' +  g' \cdot H - g' \cdot h'
 \vert h \in_L H,g\in_L G, h' \in_R H, g' \in_R G \} \sqcup 
 \\
 &
 \{ G \cdot h +  g' \cdot H - g' \cdot h,  G \cdot h' +  g \cdot H - g \cdot h'
 \vert h \in_L H,g\in_L G, h' \in_R H, g' \in_R G \} .
 \end{align*}
 \end{definition}

\begin{theorem}
The Conway product  is commutative,  the game $1_G$ is neutral for multiplication, and
$G \cdot 0_G = 0_G$.
Also,  $(G \cdot H)^\sharp = G \cdot H^\sharp = G^\sharp \cdot H$.
\end{theorem}

\begin{proof}
By induction, via "one-line proofs": \cite{Co01}, Chapter 1.
\end{proof}

However, the Conway product does not descend to the quotient $\gvN/\equiv$, and it is not
even associative or distributive over disjunctive sum.

\section{Numbers as games}

We realize the class $\No$ as explained in Chapter \ref{chap:Conway}.

\begin{definition}
By induction, we define two Maps from $\No$ to $\gvN$, associating to a number
$x$ the "game of $x$": 
for $G(x)$ we use the maximal cut representation of $x$ (the Cuesta-Dutari cuts), and for
$g(x)$ the canonical cut:
\begin{align*}
g: \No \to \gvN,& \quad
x \mapsto \{ g(y) \mid y<x, y \prec x \} \sqcup \{ g(y) \mid y>x, y \prec x \},
\\
G: \No \to \gvN, & \quad
x \mapsto \{ G(y) \mid y<x, \Beta(y)< \Beta(x) \} \sqcup \{ G(y) \mid y>x, 
\Beta(y) < \Beta(x)  \}.
\end{align*}
Note that
$g(0_\Co)= \eset \sqcup \eset = G(0_\Co)$ is the zero-game.
Denote by $\kappa: \gvN \to \gvN/\equiv$, $E \mapsto [E]$ the quotient Map, and
$[g] := \kappa \circ g: \No \to \mathbb G$ and 
$[G]:= \kappa \circ G$. 
\end{definition}

\begin{theorem}
We have $[G] = [g]$, and we could have used any timely cut representation of numbers
$x$ in order to define the same map $\No \to \mathbb G$.
The image of $[G]$ is given by all classes of games 
$X$
such that all higher order elements
$y \in^n_O X$ (where $n\in \N$ and $O$ is a sequence of $L$ and $R$ of length $n$)
satisfy: 
\begin{equation}\label{eqn:number}
\forall a \in_L y, \,  \forall b \in_R y : \qquad a < b .
\end{equation}
The Map $[G]$ is strictly order preserving and in particular injective.
Its image,  with disjunctive sum and Conway product, forms a Field,
the Field of Conway numbers. 
\end{theorem}

In order to prove the theorem, there are a lot of things to check.
Every step is by induction, and some of these steps are easy
"one-line proofs" (Conway), but some are not. 
In principle, everything is contained in \cite{S}, and our 
presentation is just a bit more formal, following the spirit of "pure set theory".
All statements concerning order and the additive theory have rather easy
"one-line proofs".
The characterization of the image of $[G]$ 
 paraphrases \cite{S} p. 401:
"A long game $x$ is a surreal number if:  $y^L < y^R$ for every subposition $y$ of $x$ and
every $y^L$ and $y^R$." 
Statements concerning the multiplicative theory
(in particular, inverses), have mostly less straightforward proofs.
Here, the presentation given by Simons \cite{Sim} offers some interesting 
shortcuts: 
he uses the group structure of $\{ L, R\}$ in  order
to simplify, or to unify, some of the computations given by Conway.
This seems to indicate that indeed some $\Z/2\Z$-graded theory might be underlying 
the structures considered here; 
however, it is not clear to me how to present such a theory in a conceptual way. 

Compared to the impartial theory, the missing element in the partizan theory
is an analog of the Grundy-map $\Gamma$:  there is no Map
$\gvN \to \No$ of which $G$ or $g$ would be a kind of Section.
Possibly,
what comes closest to this, are the "left and right values of a game", defined for rather general
games in \cite{Co01}, p.98: they do not take values in $\No$, but in {\em pairs of Dedekind 
sections of $\No$}, so give rise to pairs of "gaps" (\cite{Co01}, p. 37).
But (p. 38 loc.\ cit.), {\em the collection of all gaps is not even a Proper Class, being an 
illegal object in most set theories.} 
 It is not clear if and how these objects can be formalized in our framework,
 potentially giving rise to a new algebraic Object, yet bigger and
more comprehensive than $\No$. 
Summing up , 
the  "partizan theory of numbers" is much more intricate than the corresponding one of the nimbers from the
impartial case, and my impression is that we are still lacking a good and conceptual understanding
of what is going on here.

\chapter{All things are number.  \\
What is number?}\label{chap:final}\label{chap:No-theory} 

Conway's invention, or discovery, of the surreal numbers offers  us
the opportunity to think anew about age-old meta-mathematical 
hashtags, starting by Pythagoras' "Everything is number". 
In this last chapter, triggered by
the approaches to surreal numbers described above, I add some 
personal remarks concerning "number and numbers".
I apologize, both to  professional mathematicians and
philosophers,  if these remarks appear naive, and
 my hope would be that 
scientists who are more competent then I am, 
 might take them up,  or disprove them in case  they are completely
mislead.

\section{Pure set theory, and the universe of mathematics}

Let me recall from the Introduction, Section \ref{sec:puresettheory}: 
I propose to call "pure set theory" the set theory obtained by taking 
seriously the \textbf{Credo of classical set theory} (ZF, NBG, or equivalent):

{\em Every set is a set whose elements are sets. 

Two sets are equal if they have
the same elements.}

\nin
Set theory based on this Credo is sometimes called {\em material set theory}.\footnote{
See, e.g., \url{https://ncatlab.org/nlab/show/material+set+theory}.}
I prefer the terms "pure", or "quantum":  
as advocated in Section \ref{sec:puresettheory}, "pure" set theory in this sense
could be opposed to "common sense set theory" in the same way as
"quantum physics" could be opposed to "classical physics".
Of course, every working mathematician knows that in her or his domain,
the "material" viewpoint appears to be completely useless -- just like quantum
physics made no sense in the world of classical physics. 
However, the approaches to Conway numbers described in the preceding chapters
might indicate that the "pure", or "quantum", viewpoint deserves to be taken more
seriously.

If you are willing to temporarily accept the viewpoint of pure set theory, you realize
that the main conceptual problem is the logical status of 
{\em  products and coproducts} in such an approach: 
every "material" definition, like the Kuratowski definition, is flawed by
arbitraryness, and after iteration becomes totally uncontrollable, hence should be
rejected (cf.\ Introduction, \ref{sec:puresettheory}). 
We propose two different kinds of "solutions" to this dilemma:
\begin{itemize}
\item
When working on a specific problem, located somewhere in the von Neumann universe
$\vN$,  one always can work with "local" definitions of products and
coproducts, by indexing objects by ordered pure sets that are 
"far away" from the local setting (disjoint, higher order disjoint).
In other words, the material implementation of 
 product and coproduct in the universe would be relative and not absolute --
so we need a kind of "relativity theory for (co)product structures".
\item
If we insist in an absolute, global and material
notion of product and coproduct, then
we need to work in a "graded universe", like $\gvN$, Definition \ref{def:gvN}.
However, care must be taken not to mix up 
the theories of $\vN$ and $\gvN$.
The former can be imbedded into the latter -- these are the "impartial games",
diagonally imbedded into the "partizan games".
But possibly one should rather use the language of "supermathematics", 
or of "Bosons" and "Fermions"...
\end{itemize}

\subsubsection{The mathematical universe}

{\em Language} is an important issue for mathematics.
For instance, the language of {\em Combinatorial Game Theory} (CGT) is rich, 
colorful, often amusing and taking up words from current English in a playful
way   -- 
 see \cite{Co01, WW, S}.
Likewise, notation of CGT is rather ad-hoc,   at the opposite end of Bourbaki-style or
category theory. For insiders, this certainly contributes to the beauty of the subject, but
unfortunately  has the effect that the general mathematical reader   quickly 
gets lost and  just "retains the music" without
understanding anything of the underlying hard mathematics.
Therefore, although CGT certainly contains important
and deep results about the mathematical universe, $\vN$ or $\gvN$,
the general mathematical community is not (yet) prepared to understand and to
appreciate them.

The language of {\em mathematical physics}  (MP) is taken much more seriously in the
mathematical community, for good reasons. 
Therefore I would like to propose a switch in language, replacing terminology
from CGT by terminology from MP.
The term {\em mathematical universe} has been introduced by Tegmark
\cite{Teg}, and popularized by his subsequent  book \cite{Teg2},
mixing up ideas from physics, philosophy, and mathematics.
He states and defends what he calls the
\textbf{Mathematical Universe Hypothesis (MUH)}: 

{\em
Our external physical reality is a mathematical structure}.

Since I am a mathematician, and not a physicist, I do not want to make any
statements about our external physical reality.
But for discussing the MUH seriously, it is necessary that mathematicians
contribute to answer the question: {\em What is
 a "mathematical structure"?}
Tegmark's answer (\cite{Teg}, Appendix A) seems unsatisfactory, to me.
I see two different ways to improve or correct his explanations:
\begin{enumerate}
\item[(A)]
Material:
By "mathematical structure" is meant a pure set. Thus the MUH  would be turned into
the MMUH (material mathematical universe hypothesis):

{\em 
Our external physical reality is the (graded) von Neumann universe.}
\item[(B)]
Categorical:
By "mathematical structure" are meant all objects or morphisms of all categories
or higher categories we can think of. (The leads to the CMUH: "categorical mathematical universe
hypothesis")
\end{enumerate}
Tegmark dismisses such a distinction: he postulates a kind of big "telephone book"
indexing all structures by {\em (natural) numbers}, such that isomorphic structures 
should have the same "telephone number". 
He does not ask what happens when we change the category, or
 if the "telephone book" itself is part of reality, or not, and what consequences
 this would have.
 I have been sympathizing with Answer (B) for a long time, and I still appreciate it.
However, for today let us try to promote Answer (A):

 \begin{enumerate}
 \item
 In mainstream, ZF-based mathematics, $\vN$ {\em is} the "universe of mathematics".
 So, it  certainly is  a legitimate  candidate for the "mathematical universe"
 (or maybe, better after being upgraded to $\gvN$). 
\item
Answer (A) would be "quantum" in nature: it allows to think of a globally defined
{\em equality relation} on the universe: two things are equal, or not.
Answer (B) would be "classical" in nature, and it would call for a "relativity theory of 
equality" (e.g., some kind of type theory, see \cite{Hott}):
equality is only defined locally, as isomorphy with respect to a certain category.
Since there are "categories of categories", it is absolutely necessary to include
{\em higher order} category theory into such an approach. 
Thus (B) is more complicated than it might appear at a first glance,
whereas (A) avoids such complications.
\item
In the other direction, adopting (A), it would always be possible to consider (B)
as a "classical limit", just by forgetting the globally defined equality relation and
replacing it by some locally defined isomorphy relation. Thus 
"without loss of generality" we may adopt (A), since it does not prevent us from
working classically, if we want to.
\item
Maybe Conway himself had similar ideas in mind when using language from
MP in CGT.  E.g.: {\em
"the exact form of Mach's principle is that the atomic weight of $G$ is at least
$1$ if and only if $G$ exceeds the remote stars"}  (\cite{Co01}, p.218).
It should be worth the effort to make such phrases understandable for a large
mathematical public!
\item
The hierarchical structure of $\vN$ or $\gvN$ might turn out to be an extremely
useful feature for understanding the universe.
Usual mathematical theories just mimic the step from rank $\alpha$ to its
successors $\alpha+1$ and $\alpha + 2$
(e.g., topology, measure theory: theories working with certain "sets of subsets").
Already the step from $\alpha$ to $\alpha+3$ is quite a challenge in conventional
approaches, 
but passing beyond a limit $\alpha + \lambda$  may  completely 
change the mathematical structure -- and
beyond $\alpha+ \epsilon_0$ it may contain  mathematics of which so
far we have no idea.

It would be tempting to interprete the
 "unreasonable ineffectiveness of mathematics in biology",  pointed out by
Israel Gelfand,\footnote{cf.\ \url{https://en.wikipedia.org/wiki/Unreasonable_ineffectiveness_of_mathematics}.}
by the fact that present day mathematics, at best, understands the step from a stage to its successor stage,
but not the step going beyond the next limit. 

\item
What is "random", what is "determined", what "information" do we have? 
Maybe a more profound theory of $\vN$ or $\gvN$ will give answers.

The structure of $\vN$ or $\gvN$ is all but "random": it is a
well-defined mathematical structure that seems to be "completely determined".
 However, already $\vN_7$ has more elements
then our physical universe has atoms, so there is no hope to "compute everything in
$\vN_n$ by brute force", even for rather small $n$. 
So, if it is impossible to get "complete information" on $\vN_n$, what does
 "information  about $\vN$" mean?
 For instance, will it be
 necessary to introduce some total order
on $\vN_\alpha$ to give each element an "order number" (like the tape of a Turing machine)? 
 -- for finite $\alpha$, there is a natural order (see Table \ref{table:beginning}:
 the colex-order), but for infinite $\alpha$, to fix some order one would need to invoke the
axiom of choice (which has been avoided so far!): would this introduce an element of randomness in the theory? 
\item
In Tegmark's setting, one might say  "the telephone book {\em is} the universe".
Thus self-reference would be inherently built into the concept of "universe".
Neither (A) nor (B) can escape of this. See \cite{Ru} for a logical and a philosophical
discussion of such issues.
\item
The universe $\vN$ itself  is a proper class and not a set,
and therefore is not an element of the universe.
In the material setting (A), it is possible to name things that are outside
the universe. (E.g., the "gaps in $\No$", defined in \cite{Co01}, are outside the universe, too, yet one
may conceive such a concept.) 
 In setting (B), this is more difficult, since an appropriate language 
is missing. 
Again, see \cite{Ru}, for both a mathematical and philosophical discussion of this issue. 
\item
In quantum theory, the {\em complex numbers} play a more important r\^ole
than the real ones. So, 
where are the complex, and the
surcomplex, numbers in $\vN$ or $\gvN$?  
There are natural models, see below, Section \ref{sec:surcomplex}.
\end{enumerate}

\section{Intrinsic $\No$-theory}

There is a lot of recent work dealing, in one way or another, intrinsically with surreal numbers and 
related topics.
With no pretention of being exhaustive, let us quote
\cite{BH1, BH2, BKMM, BM, CE, Fo, GN, KM, Ly, RS, Sch}, besides the work of Ehrlich already quoted earlier --
see also the conference report \cite{BEK} for an overview. 
By "intrinsic", we mean a viewpoint taking the existence of $\No$ for granted, and 
pursuing the intrinsic theory of $\No$ and other closely related fields or Fields, mostly
various kinds of "generalized power series fields", or "transseries fields".
In most of this work, Gonshor's definition \cite{Go} of $\No$ is used, 
and Conway's original one rather serves as a kind of motivation; more general games
usually are not in the scope of such work. 
The main tool used in "intrinsic $\No$-theory" is  {\em Conway's  omega-map} $x \mapsto \omega^x$,  which we have not yet
mentioned in our approach. 
In the present text, I decided not to develop this direction, since I feel it would start another
work, less general and more specialized than the focus of the present one,
and I will 
 just give the definition of the 
omega-map in terms of our approach (Section \ref{sec:omega-map}).

\subsection{Open problems and questions}\label{sec:questions}

The conference report \cite{BEK}, p.\ 3315--3317, contains an interesting
list of  "open problems and questions" arising in what I call intrinsic $\No$-theory.
Most of these questions concern generalized power series fields and valuation
theory, in particular, related to exponential and logarithm series and derivations.
Other questions are of a more general nature, such as :

\msk \nin
\textbf{Question 2.}
Describe the field operations of $\No$ using the sign sequence representation.

\msk \nin
\textbf{Question 3.}
Let $i=\sqrt{-1}$. Is there a good way to introduce sin and cos on $\No$ and an exponential
map on $\No[i]$? 
Is there a surreal version of the $p$-adic numbers?

\msk 
\nin
\textbf{Question 16.}
Can one characterize the subset $\bQ$ of $\No$ in terms of sign-se\-quen\-ces?

\msk
\nin
\textbf{Question 17.}
Can one extend the simplicity order of $\No$ to functions? In which sense $+$
is the simplest function increasing in both arguments?
Is exp the simplest homomorphism from $(\No,+)$ to $(\No^{>0},\times)$ such that for
all $n\in \N$ and positive infinite $x$ we have $\exp(x) >x^n$?

\msk
\nin
\textbf{Question 18.}
Can one describe an integer part of $\No$ which is a model of 
\href{https://en.wikipedia.org/wiki/True_arithmetic}{true arithmetic}?
The existence of such an integer part should follow by the saturation properties
of $\No$, but an one construct such an integer part explicitly (without the axiom of
choice, say)?

\msk
Question 16 has been answered by Moritz Schick (\cite{Sch}, cf.\ our Theorem 
\ref{th:rationals}).
Question 3 is, in our opinion, related to items to be discussed below (Section \ref{sec:surcomplex}).
Question 17 about the "simplicity order of functions" is certainly
a very important topic, and which probably needs a general and abstract context for
formalizing a suitable answer. 
Question 2 really asks for a definition of "Conway arithmetics" in the context of our 
definition of $\No$.
In the following, we give a  sketch of a programme how to attack it.

\subsection{An algebraic approach to Conway arithmetics}\label{sec:Grothendieck} 
 
 The above quoted Question 2 asks for defining sum $x+y$ and product $x \cdot y$
 of numbers $x,y \in \No$ in "combinatorial terms", starting from the sign-sequence,
 that is, the list of elements, of $x$ and $y$.
 As we have seen, this is easy for the Cantor-style operations
 $\oplus, \otimes, \ootimes$, since they are continuous in the second argument.
 Thus there might be some hope that the commutative, non-continuous field operations
 could be derived from the Cantor-style operations in a similar way as the
 "natural", Hessenberg operations are derived from Cantor's ordinal arithmetic of
 ordinals.
In the following, we sketch some ideas how this could be achieved.
The idea is quite naive:   build up the arithmetic structures of 
$\No$ from $\On$ in steps, copying the procedure
 of building up $\R$ from $\N$.
If we understand the combinatorial structure of each step well enough, one might use
this to give a "purely combinatorial" definition of the Conway arithmetic operations.
 However, filling in the missing details would need a more profound understanding
from which, at present, we are still remote.

\begin{enumerate}
\item[(1)]
We equip $\On$ with its natural, Hessenberg operations (see Subsection
\ref{ssec:natural}). 
Let $\Z_\On$ be the \href{https://en.wikipedia.org/wiki/Grothendieck_group}{Grothendieck Group}
 of $(\On,+)$, which we can identify  
with the Image of the Map (here, $\On$ are the Conway ordinals)
\begin{equation}
 \On \times \On \to \No, \quad (\alpha , \beta) \mapsto \alpha - \beta .
\end{equation}
In other terms, $\Z_\On$ is the smallest Subgroup of $(\No,+)$ containing the Conway ordinals.
It is a proper Subgroup of the omnific integers
 $(\Oz,+)$, and it is stable under multiplication, so is a Ring.
 
 \textbf{Task}: describe the combinatorial structure of addition and multiplication in
 $\Z_\On$ in terms of sign-expansions (set theory). See below.
\item[(2)]
One may consider the short numbers $\No_\omega = \Z[\frac{1}{2}]$ (dyadic fractions) to be
well-understood:  they are added by the usual carry-rules, which give, for instance,
for all short numbers $x$, 
\begin{equation}
\frac{x + x_+}{2}  = x_{+-},
\qquad
\frac{x + x_-}{2}  = x_{-+}.
\end{equation}
Let $\Z^{(2)}_\On$ be the class obtained from $\Z_\On$ by taking
all finite successors of elements $x \in \Z_\On$.
This should be a ring, tentatively a kind of scalar extension of
$\Z_\On $ by $  \Z[\frac{1}{2}]$, and whose combinatorics should be closely related
to the one of $\Z[\frac{1}{2}]$.
\item[(3)]
There should also be a well-defined
Field of fractions $\bQ_\On$ of $\Z_\On$, which 
 is a Subfield of $\No$, and which should extend
$\Z^{(2)}_\On $ in a way similar to the extension from  $\Z[\frac{1}{2}]$ to
$\bQ$ (cf.\ Theorem \ref{th:rationals}).
\item[(4)]
Finally, explain in what sense $\No$ itself can be considered as completion
of $\bQ_\On$, or of $\Z^{(2)}_\On$: 
can it be understood as a passage from "short reals" to "long reals", tensored with
$\Z_\On$? 
What can we say about the intermediate "Number Fields" and their combinatorial structure?
\item[(5)]
If this approach makes sense, it could also give a hint to answer the above quoted Question 3,
"is there a surreal version of the $p$-adic numbers?" -- there could be other 
"completions" of $\bQ_\On$ than $\No$, which one would like to realize by certain
subclasses of $\vN$.
\end{enumerate}

\nin
 Comment concerning the Task from Item (1):
An answer can be given by using the Conway normal form, eqn.\  (\ref{eqn:CoNo}),
and results by Gonshor \cite{Go}.

\begin{theorem}
The Grothendieck group $\Z_\On$ is the class of all surreal numbers $x$ of the form
$$
x = \oplus_{i=1}^n \omega^{\beta_i} r_i,
$$
i.e., a
finite ordinal type sum, with  $n\in \N$ and ordinals $\beta_1 > \ldots > \beta_n \geq 0$ and 
coefficients $r_i \in \Z$.
Put differently,
 in the Cantor normal form (\ref{eqn:CantorNoFo}) we allow the coefficients to be 
usual integers instead of usual natural numbers. 
Equivalently, the Conway normal form of $x$ is given by the same formula, with
$\oplus_i$ replaced by $\sum_i$ (Conway-sum of numbers).
Thus, Conway and (generalized) Cantor normal forms coincide in this case. 
\end{theorem}

\begin{proof}
We repeatedly use \cite{Go}, Theorem 5.12 (c), saying that the sign sequence of 
a number in Conway normal form (\ref{eqn:CoNo}) is given by juxtaposition of the
sign sequences for the successive $\OM^{y_\beta^o} r_\beta$, where
$y_\beta^o$ is the "reduced" sequence of $y_\beta$.

Let $u,v$ be Conway ordinals. We represent them in Cantor normal form.
Then, by the above quoted result, the Conway normal form of $u$ is given by the same
expression (with $\oplus$ replaced by $+$), and the one of $-v$ is given by
the same expression, replacing the coefficents $k \in \N$ by their negative $-k \in \Z$.
The sum $u-v = u + (-v)$, taken in $\No$, is commutative, so we order terms in
order of decreasing exponents  to get the Conway normal form of $u-v$.
Applying the above quoted theorem again, this sum is given by "juxtaposition"
(since the exponents are ordinals and the coefficients integer, the "reduced" sequence here
is just the usual sequence), that is, by an ordinal type sum $\oplus$, and we end up with
an expression as given in the theorem.
Conversely, every such expression corresponds to some difference $u-v$, so they describe
the Grothendieck group $\Z_\On$.
\end{proof}

\nin
Remark.
It seems that elements of the Grothendieck group correspond to the 
"surintegers", defined p.\ 27 in \cite{Re}.

\subsection{Conway's omega map}\label{sec:omega-map}

In line with the approach just sketched, one can give a definition of 
Conway's important {\em omega-map} in terms of sign-expansions.
In fact, this has already been done by Gonshor, \cite{Go}, Thm.\
5.11, and in the following we simply "translate" his result into our setting.
To state this,
recall the operations $\oplus,\otimes,\ootimes$ on $\No$, and
 from Definition \ref{def:width} the {\em width} $\w(x) $ of a number $x$. 
By induction, one gets immediately from its definition
\begin{equation}
\w(a \oplus b) = \w(a) \oplus \w(b).
\end{equation}

\begin{definition}\label{def:omega-map}
By transfinite induction,  for every number $x$, we define a number $\OM^x$ by $\OM^0 = 1$, and

$\OM^{x_+} = \OM^x \oplus \OM \ootimes (\w(x) \oplus 1)$

$\OM^{x_-} = \OM^x \oplus \OM^\sharp  \ootimes ( \w(x) \oplus 1)$

$\OM^{\lim_{\alpha \to\lambda} x_\alpha} = \lim_{\alpha \to \lambda} \OM^{x_\alpha}$
\end{definition}

\nin
Since, for ordinals, $\ootimes$ is usual ordinal exponentiation, and using $(\alpha \otimes \beta)^\sharp =
\alpha^\sharp \otimes \beta$,
 the formulae may also be written

$\OM^{x_+} = \OM^x \oplus \omega^{\w(x) \oplus 1}$

$\OM^{x_-} = \OM^x \oplus  (\omega^\sharp)^{ \w(x) \oplus 1}$.

\nin
As an immediate consequence, we get {\em monotonicity}:

 if $x \preceq y$, then
$\OM^x \preceq \OM^y$.

\begin{theorem}
The number $\OM^x$ defined above coincides with the number $\omega^x$ defined by
Conway (\cite{Co01}, p.\  31), that is, Conway's $\omega$-map is
$$
\No \to \No, \quad x \mapsto \OM^x.
$$
\end{theorem}

\begin{proof}
The sign-expansion of Conway's $\omega^x$ computed by Gonshor, \cite{Go}, Thm.\
5.11, p.80, coincides with the sign-expansion of the number inductively defined above.
(Gonshor denotes, for a number $a$, its width by $a^+$, and $a_\alpha = \w([a]_\alpha)$
is the width of the $\alpha$-truncation of $a$, that is, the "number of pluses in the initial
segment of $a$ of length $\alpha$". In our notation, his formula 
can be written 
$$
\omega^x = 
1 \oplus \bigoplus_{\alpha < \Beta(x)}
s_x(\alpha  ) \omega^{(\w ([x]_\alpha)  + 1 ) }
$$
which by induction is the same as the number $\OM^x$ defined above.)
\end{proof}

\begin{definition}
A number $y$ is called a {\em monomial} if there is a number $x$ such that
$y = \OM^x$.
The class of all monomials is denoted by $\mathcal M$.
\end{definition}

By \cite{Co01}, Theorem 21, every number $x$ has a unique expression,
its {\em Conway normal form},
\begin{equation}\label{eqn:CoNo}
x = \sum_{\beta < \alpha} \OM^{y_\beta} r_\beta,
\end{equation} 
where $\alpha$ is some ordinal, the numbers $r_\beta$ are non-zero reals,
and the numbers $y_\beta$ form a descending sequence of numbers.
Such expressions behave like generalized power series with real coefficients,
and they are added and multiplied like generalized series in the monomials.
As mentioned above, this is the point of departure of most recent work on
intrinsic $\No$-theory.

\section{The surcomplex numbers}\label{sec:surcomplex}

Conway introduces the Field $\No[i]$ of \href{https://en.wikipedia.org/wiki/Surreal_number#Surcomplex_numbers}{\em surcomplex numbers} simply as the
class of all numbers $x+iy$ with $x,y \in \No$,  $i^2 = -1$
(\cite{Co01}, p.42).
How can we realize $\No[i]$ inside the von Neumann universum $\vN$?
And if there are several realizations, is there a "natural" one, comparable to the
one of $\No$ that we have given? 
For the moment, I can only offer half of an answer.

\subsubsection{Shifting $\On$ and $\No$}

First, some remarks on various "realizations" of $\On$ and $\No$ inside $\vN$:
there are several ones, but the ones considered in Chapter \ref{chap:numbers} seem to be
by far the most natural ones.
As explained in Section \ref{sec:vN-hierarchy}, 
besides the von Neumann realization of the ordinals
$\On$ as pure sets,  there are other realisations of ordinals:

-- the Zermelo ordinals $\On_Z$,

-- taking the set $\vN_\alpha$ as ordinal $\alpha_X$,

-- or any other choice $\alpha_C \in \vN_{\alpha + 1} \setminus \vN_\alpha$.

\nin
In all cases, such ordinals would form a well-ordered class of pure sets starting with 
$0 = \eset$. 
Also, we could shift each of these classes,
$\alpha_{S+1} := (\alpha + 1)_S$.
For instance, the Conway ordinals are shifted von Neumann ordinals,
$\alpha_\Co = (\alpha + 1)_\vn$.
We can also shift by $2$, etc. 

To every such class we can associate the corresponding (surreal) numbers:
a number $x$ is a set of ordinals belonging to the class in question, having
a maximal element.
The theory of the "Conway numbers" thus obtained is, essentially, isomorphic to the one developed in
Chapter \ref{chap:numbers}. 
However, it would be less natural and might have some confusing features.
For instance,  if we work with Zermelo ordinals instead of von Neumann ordinals, then
the passage from a Zermelo ordinal $\alpha_Z$ to the corresponding
Conway number would not be described by a simple shift $+1$, but by a much
more complicated
formula
(e.g., 
$\{ 1_Z \} = 2_Z$  would correspond to the
Conway number $-1$, and $\{ 2_Z \} = 3_Z$ to $-2$).

\subsubsection{Defining $\bC\No$}

 Next we fix two realisations $\On_i$, $i=1,2$, of $\On$ in $\vN$ such that
 $\On_1 \cap \On_2 = \eset$. 
 For instance, take $\On_1 = \On_\vn $, the von Neumann ordinals,
 and $\On_2 = \On_{Z+2}$, the shifted Zermelo ordinals, which start
 at $2_Z = 0_{Z+2} = \II$, and have no element in common with the von Neumann
 ordinals.
 Let $\No_i$, $i=1,2$, be the classes of numbers defined by using the ordinals
 $\On_i$, in the sense explained above.
 
 \begin{definition}
 A {\em surcomplex number} is a subset
 $z \subset \On_1 \cup \On_2$ such that both
 $x:= z \cap \On_1$ and $y:= z \cap \On_2$ have a maximal element.
 The number $x \in \No_1$ is called the {\em real part},
 and $y \in \No_2$ the {\em imaginary part} of $z$.
 We define structures on the class $\bC\No$ of surcomplex number
 in the usual way from those of $\No$, by identifying $z$ with the pair
 $(x,y)$.
 \end{definition}
 
 Obviously, 
 this gives a realisation of Conway's Field $\No[i]$ inside $\vN$, and it avoids
 using the Kuratowski pair definition.
 Indeed, what we have done is what I meant by "local" definition of ordered pair:
 fixing the disjoint copies $\No_1$ and $\No_2$ defines, locally, a direct product
 structure on the power sets of stages of $\No_1 \cup \No_2$.
 The drawback of this definition is, however, that
  no choice of $\No_2$ seems to be
 very natural, and that
 the maps $\No_i\to \bC\No$ are not given by set inclusion.
 For instance, taking $\On_2 = \On_{Z+2}$, as suggested above,
 would give the zero element of $\bC\No$ (in the notation of Table \ref{table:beginning}) 
 $$
 0_{\bC\No} = \{ 0_\vN , \II \} = \{ \eset , \II \} = \mbox{V}.
 $$
 
 \subsubsection{Shuffle of von Neumann and Zermelo ordinals}
 
 To get a more natural realization of $\bC\No$, we define first the {\em shuflle of 
 $\No_\vn$ ad $\No_Z$} to be their union
 $$
 \No_{\vn Z}:= \No_\vn \cup \No_Z ,
 $$
 whose elements are called Co-Z-suffle ordinals,
 together with the total order given as follows:
 we have $0_\vn = 0_Z = \eset$,
 $1_\vn = 1_Z = \I$,
 $2_Z = \II < 2_\vn = \III$, and so on:
 $$
 \alpha_Z < \alpha_\vn < (\alpha + 1)_Z < (\alpha + 1)_\vn < \ldots ,
 $$
 and for limit ordinals,
 $\alpha_Z < \alpha_\vn$.
 This total order is  a well-order:
 if $S \subset \On_{\vn Z}$ is a non-empty set, then
 $S \cap \On_Z$ or $S \cap \On_\vn$ is non-empty, so has a minimal element, and the smaller
 of both minima is the minimal element of $S$. 
 
 \subsubsection{Cocomplex numbers}
 
 \begin{definition}
 A {\em Cocomplex number} (Conway complex number) is a set $z$ of Co-Z-shuffle ordinals
 having a maximal element, which we call again {\em birthday} and denote by
 $\Beta(z)$.
 \end{definition}
 
 This definition is tentative and ad-hoc:
 a full justification would be by proving that the Cocomplex numbers have all the 
 good properties that one would expect from $\bC\No$.
 Tentatively, the first "new" number in the list,  $\{ 2_Z \} = 3_Z$ 
 should correspond to a "very simple" complex, non-real number.
 For symmetry reasons, I would guess that it corresponds to 
 the unit $j$ (or $-j$) of the Eisenstein integers $\Z[j] = \Z \oplus j \Z$,
 \begin{equation}
 j:= e^{2\pi i / 3} = \frac{1}{2} + i \frac{\sqrt 3}{2}
 \end{equation}
rather than to the usual unit $i$ of the Gaussian integers.
Namely, this would fit quite well with the following tentative table of correspondence, 
 taking up notation from Table  \ref{table:beginning}:
 \begin{center}
 \begin{table}[h]\caption{The beginning of cocomplex numbers}\label{table:Coco}
\begin{tabular}{|*{10}{c|}}
\hline
nr. & sb. & rk.  & depth $1$ &  ordinal  & $\in \No$  & Cocomplex (speculative)
\\
\hline
$0$ & $\eset$ & $0$ & $\eset$ & $0_\vn= 0_Z$ &  -- & -- 
\\
\hline
$1$ & $\I$ & $1$ & $\{ \eset \}$ &  $1_\vn = 1_Z $ & $0_\Co$ & $0_\Co$
\\
\hline
$2$ & $\II$ & $2$ & $\{ \I \}$ &  $2_Z$  & $-1_\Co$ & $(-1)_\Co$ 
\\
$3$ & $\III$ & $2$ & $\{ \I,\eset \}$  &  $2_\vn$ & $1_\Co$ &  $1_\Co$ 
\\
    \hline
$4$ & IV & $3$ & $\{ \II \}$   &  $3_Z$  & -- & $ - j  $ 
\\
 $5$ & V & $3$ &  $\{ \II,\eset \}$ &  -- & -- &  $j$ 
 \\
 $6$ & VI &  $3$ &  $\{ \II,\I \}$ &  -- & -- & $ - j - 1 $ 
 \\
 $7$ &VII  &  $3$ &  $\{ \II,\I ,\eset\}$ &  -- & -- & $1+j$ 
 \\
 \hline
 $8$ &VIII  & $3$ & $\{ \III \}$   &  -- & $-2_\Co$ & $-2_\Co$
\\
 $9$ & IX & $3$ & $\{ \III , \eset \}$   &  -- & $\frac{1}{2}$ & $\frac{1}{2}$ 
\\
 $10$ & X  & $3$ & $\{ \III , \I \}$   &  -- & $ - \frac{1}{2} $ & $ - \frac{1}{2} $ 
\\
 $11$ & XI & $3$ & $\{ \III , \I, \eset  \}$   &  $3_\vn$  & $2_\Co$ & $2_\Co$
\\
\hline
 $12$ & XII  & $3$ & $\{ \III , \II  \}$   &  -- & -- & $ j-1  $
\\
 $13$ &  XIII & $3$ & $\{ \III , \II ,\eset \}$   &  -- & -- & $1-j$
\\
 $14$ & XIV  & $3$ & $\{ \III , \II ,\I \}$   &  -- & -- & $ j+2  $
\\
 $15$ & XV  & $3$ & $\{ \III , \II ,\I ,\eset \}$   &  -- & -- & $2 - j $ 
\\
\hline
$16$ & XVI & $4$ & $\{ \mbox{ IV } \}$ &  $4_Z$  & -- & $- 2 j$ 
\\ 
    \end{tabular}
    \end{table}
\end{center}
Here, the 5th column names the (von Neumann and Zermelo) ordinals, the 6th column names the
surreal numbers (7 out of the 15 elements of rank $\leq 3$), and the 7th column is purely
speculative and gives names to Cocomplex numbers.
Note that  {\em all} elements of rank $\leq 3$ (except
$\eset$, which is not a number) are indeed Cocomplex numbers.
Thus the list of "simplest complex numbers" should start with zero, followed by the six
 Eisenstein units:
 $$
 0, \qquad 1, - 1,  j,-j, 1+j, - 1 - j .
 $$
 Still by the same speculation, the "short Cocomplex numbers" then would be the
 ring $\Z[\frac{1}{2}, j ]$, that is, the scalar extension of the Eisenstein integers by
 dyadic rationals.
 The tree-structure should be {\em hexagonal}: 
 every surcomplex number that is not a limit number should have {\em six successors} -- 
 one in each main direction of the hexagonal lattice, at distances which either are $1$ or half
 of the distance to the nearest neighbor, just like in the surreal number tree for a single dimension.
Of course, the main task would be to describe the Field structure of the Cocomplex numbers in these terms.
If there is such a construction, it should have pleased Conway, since it would bring together two main strands of his work,
surreal numbers and his work on lattices, groups, and sphere packings
(\cite{CSl}, cf.\ \cite{Be24}).   
It would be a strong argument in favor of considering the universe $\vN$ as a rich mathematical structure,
deserving to be studied in its own right.


\subsubsection{Function theory: Question 3 revisited}

Question 3 (Subsection \ref{sec:questions}) is certainly one of the most important topics
concerning analysis on surreal numbers:
is there a general method of defining, or "extending", real (smooth, or real-analytic)
functions $f$ to Functions $f_\No$ defined on certain Domains in $\No$?
For the time being, there exists no general answer, but there are some examples, most notably
Gonshor's extension of the exponential function, \cite{Go}, Chapter 10.

One might think that techniques used in non-standard analysis, defining a 
 non-standard extension ${}^*f$ to the field ${}^*\R$ of hyperreal numbers
 (see \cite{Numbers}, Chapter 12), could serve here
 as a model. However, so far there seems to be no link between non-standard analysis
 and analysis on surreal numbers, as was predicted by Conway himself (\cite{Co01}, p.44):
 {\em So we can say that in fact the field $\No$ is really irrelevant to non-standard analysis.}
 I'm not so sure if this should be the last word.
 
 Be this as it may, a "most natural realization" of the field of surcomplex numbers should be
 important when considering the problem of extending {\em complex} functions to
 $\No[i]$, as suggested by "Question 3".
 Besides of the exponential map and trigonometric functions, one might also think of
 the Riemann zeta-function, and one may wonder if working on a 
    "\href{https://www.reddit.com/r/math/comments/11h0pol/the_surcomplex_riemann_hypothesis/}{surreal/surcomplex Riemann hypothesis}" might help in attacking the original one.

\end{document}